\documentclass[10pt]{amsart}
\usepackage{amssymb,amsfonts}
\usepackage[all,arc]{xy}
\usepackage{enumitem}
\usepackage{mathrsfs}
\usepackage{ stmaryrd }
\usepackage{url}

\usepackage
[pdfauthor={Preston Wake and Carl Wang-Erickson},
 pdftitle={Pseudo-modularity and Iwasawa theory},
 bookmarks=false]
{hyperref}

\newcounter{qcounter}

\newenvironment{abcs}{
  \begin{list}
  {(\arabic{qcounter}) ~}
  {
    \usecounter{qcounter}
    \setlength{\itemsep}{0in}         
    \setlength{\topsep}{0in}
    \setlength{\partopsep}{0in}
    \setlength{\parsep}{0.05in}
    \setlength{\leftmargin}{0.285in}   
    \setlength{\rightmargin}{0in}
    \setlength{\itemindent}{0in}
    \setlength{\labelsep}{0in}
    \setlength{\labelwidth}{\leftmargin}
    \advance \labelwidth by-\labelsep
  }
}{
  \end{list}
}

\define\isoto{\xrightarrow{\sim}}
\define\onto{\twoheadrightarrow}
\define\coker{\mathrm{coker}}
\DeclareMathOperator{\Spec}{Spec}
\define\Ch{\mathrm{Char}_\Lambda}

\define\E{\mathcal{E}}

\define\cC{\mathcal{C}}

\define\zinfdm{\{0,\infty\}_{DM}}

\newcommand{\dia}[1]{\langle #1 \rangle}

\newcommand{\ttmat}[4]{\left( \begin{array}{cc}
#1 & #2 \\
#3 & #4
\end{array}
\right)}

\newcommand{\Z}{\mathbb{Z}}
\newcommand{\Q}{\mathbb{Q}}

\newcommand{\F}{\mathbb{F}}

\newcommand{\fn}{\mathfrak{n}}

\newcommand{\fH}{\mathfrak{H}}
\newcommand{\h}{\mathfrak{h}}

\newcommand{\fP}{\mathcal{P}}

\newcommand{\sO}{\mathcal{O}}

\newcommand{\tH}{\tilde{H}}
\newcommand{\Lam}{\Lambda}
\newcommand{\I}{\mathcal{I}}
\newcommand{\p}{\mathfrak{p}}
\newcommand{\m}{\mathfrak{m}}

\newcommand{\cP}{\mathcal{P}}

\newcommand{\cJ}{\mathcal{J}}

\newcommand{\U}{\Upsilon}

\newcommand{\chii}{{\chi^{-1}}}
\newcommand{\cL}{\mathcal{L}}

\newcommand{\Hom}{\mathrm{Hom}}
\newcommand{\Gal}{\mathrm{Gal}}

\newcommand{\Aut}{\mathrm{Aut}}

\newcommand{\Ext}{\mathrm{Ext}}

\newcommand{\End}{\mathrm{End}}

\newcommand{\Fr}{\mathrm{Fr}}

\newcommand{\HDM}{\tH_{DM}}

\newcommand{\X}{\mathfrak{X}}

\newcommand{\zinf}{\{0,\infty\}}

\newcommand{\lb}{{[\![}}
\newcommand{\rb}{{]\!]}}

\newcommand{\red}{\mathrm{red}}

\newcommand{\Lamf}{\Lambda_{(f_\chi)}}
\newcommand{\Lamfnochi}{\Lambda_{(f)}}

\define\cA{\mathcal{A}}

\define\Pic{{\mathrm{Pic}}}
\define\ord{{\mathrm{ord}}}
\define\Cl{{\mathrm{Cl}}}
\define\GL{{\mathrm{GL}}}
\define\kcyc{\kappa_{\mathrm{cyc}}}

\define\EL{\mathrm{E}_\Lambda}
\define\Hc{H_{(c)}}

\newtheorem{thm}{Theorem}[subsection] 
\newtheorem*{thm*}{Theorem}
\newtheorem*{thmA}{Theorem A}
\newtheorem*{corB}{Corollary B}
\newtheorem*{corC}{Corollary C}

\newtheorem*{thmD}{Theorem D}
\newtheorem*{thmE}{Theorem E}

\newtheorem{cor}[thm]{Corollary}
\newtheorem{prop}[thm]{Proposition}
\newtheorem{lem}[thm]{Lemma}
\newtheorem{conj}[thm]{Conjecture}

\theoremstyle{definition}
\newtheorem{defn}[thm]{Definition}

\newtheorem{exmp}[thm]{Example}
\newtheorem{exmps}[thm]{Examples}

\theoremstyle{remark}
\newtheorem{rem}[thm]{Remark}

\newcommand{\ra}{\rightarrow}

\newcommand{\lra}{\longrightarrow}
\newcommand{\lrisom}{\buildrel\sim\over\lra}
\newcommand{\risom}{\buildrel\sim\over\ra}
\newcommand{\rinj}{\hookrightarrow}

\newcommand{\rsurj}{\twoheadrightarrow}

\newcommand{\PsR}{{\mathrm{PsR}}}
\newcommand{\bF}{\mathbb{F}}
\newcommand{\bQ}{\mathbb{Q}}
\newcommand{\bZ}{\mathbb{Z}}
\newcommand{\bT}{\mathbb{T}}
\newcommand{\cE}{\mathcal{E}}
\newcommand{\cO}{\mathcal{O}}
\newcommand{\Db}{{\bar D}}
\newcommand{\mDb}{{\mathfrak m_{\bar D}}}

\newcommand{\Ann}{\mathrm{Ann}}
\newcommand{\Fitt}{\mathrm{Fitt}}
\newcommand{\Jac}{\mathrm{Jac}}

\newcommand{\EcL}{\mathrm{E}_{\cL}}

\DeclareMathOperator{\Tr}{\mathrm{Tr}}

\newcommand{\bro}{{\bar\rho}}

\usepackage{color}

\makeatletter
\let\c@equation\c@thm
\makeatother
\numberwithin{equation}{subsection}

\title{Pseudo-modularity and Iwasawa theory}

\author{Preston Wake}
\address{Department of Mathematics, UCLA, Los Angeles, CA 90095-1555, USA}
\email{wake@math.ucla.edu}

\author{Carl Wang-Erickson}
\address{Department of Mathematics, Imperial College London \\
	London SW7 2AZ, UK}
\email{c.wang-erickson@imperial.ac.uk}

\begin{document}

\begin{abstract}
We prove, assuming Greenberg's conjecture, that the ordinary eigencurve is Gorenstein at an intersection point between the Eisenstein family and the cuspidal locus. As a corollary, we obtain new results on Sharifi's conjecture. This result is achieved by constructing a universal ordinary pseudodeformation ring and proving an $R = \bT$ result. 
\end{abstract}

\maketitle

\tableofcontents

\section{Introduction}

In their proof of the Iwasawa Main Conjecture, Mazur and Wiles \cite{MW1984}, following ideas of Ribet \cite{ribet1976}, used the geometry of modular curves to study class groups of cyclotomic fields. Sharifi \cite{sharifi2011} has formulated remarkable conjectures that refine the Main Conjecture. Roughly, Sharifi's conjecture states that very fine information about the arithmetic of cyclotomic fields can be captured by the geometry of modular curves. Recently, Fukaya and Kato \cite{FK2012} have proven partial results on Sharifi's conjecture.

We prove new results on Sharifi's conjecture by building upon developments in the deformation theory of Galois representations of Bella\"iche and Chenevier \cite{BC2009, bellaiche2012, chen2014} and one of us \cite{WE2017}. The ordinary eigencurve $\cC^\ord$ of $p$-adic eigenforms of tame level $N$ provides a setting in which both the results and the technique may be discussed. The Hida Hecke algebra $\fH$ and its cuspidal quotient $\h$ are canonical integral models for $\cC^\ord = \Spec \fH[1/p]$ and its cuspidal locus $\cC^{\ord,0}:= \Spec \h[1/p]$. We are interested in the singularity that occurs at an intersection point $\p$ of $\cC^{\ord,0}$ and the Eisenstein family, which we call an \emph{Eisenstein intersection point}. The respective local rings at such a point will be written $\fH_\p$ and $\h_\p$. We were motivated by a conjecture of one of us proposing that $\fH_\p$ and $\h_\p$ are Gorenstein \cite[Conj.\ 1.2]{wake2}. Indeed, the Gorensteinness of $\fH_\p$ is equivalent to certain weakened versions of Sharifi's conjecture and Greenberg's conjecture \cite[Thm.\ 1.3]{wake2}. 

Our main result is that this Gorenstein conjecture is equivalent to a certain case of Greenberg's conjecture. In geometric terms, we also show that a certain case of Greenberg's conjecture is equivalent to the singularity of $\fH$ at $\p$ being a plane singularity (see \S\S \ref{subsec:plane singularity}, \ref{subsec:eigencurve}). We deduce from this a new result about Sharifi's conjecture, conditional only upon Greenberg's conjecture; we also prove an $R = \bT$ result. 

In order to apply the deformation theory of Galois representations to $\cC^\ord$, we consider the Galois modules arising from the cohomology of modular curves. The ordinary parabolic part $H$ of this cohomology is a finitely generated $\h$-module with a $\h$-linear $G_\bQ$-action, but is not necessarily locally free on $\cC^{\ord,0}$. While the generic rank of $H$ is 2, its localizations $H_\p$ at Eisenstein intersection points are locally free $\h_\p$-modules if and only if $\h_\p$ is Gorenstein. This raises a significant obstacle: the usual deformation theory of Galois representations only addresses locally free modules with Galois action, so we must know that $\h_\p$ is Gorenstein in advance to apply this theory. Consequently, we resort to the deformation theory of Galois \emph{pseudorepresentations}. We construct a \emph{universal ordinary pseudodeformation ring} and use it to control $\fH_\p$. Indeed, even if we were to assume that $H_\p$ is free, it is more natural to study pseudorepresentations because there exists a $G_\bQ$-pseudorepresentation over $\cC^\ord$, while the $G_\Q$-action only gives a representation valued in $\h_\p$, not $\fH_\p$. Moreover, our technique allows us to make weaker assumptions about class groups than in \cite{SW1997}. 

After stating the main results in \S\ref{subsec: main results}, we introduce in \S\ref{subsec:OP} our notion of ordinary pseudorepresentation, which is of independent interest. In particular, our formulation is \emph{not} ``a global pseudorepresentation is ordinary if its restriction to a decomposition group at $p$ is ordinary.''

\subsection{Setup} 
\label{subsec:setup}

In order to state our results, we introduce the main objects of study.

Let $p \ge 5$ be a prime number, and let $N$ be an integer such that $p \nmid N\phi(N)$. Let
$$
\theta: (\Z/Np\Z)^\times \to \overline{\Q}_p^\times
$$
be an even character. Let $\chi=\omega^{-1}\theta$, where 
$$
\omega: (\Z/Np\Z)^\times \to (\Z/p\Z)^\times \to \Z_p^\times
$$
is the Teichm\"uller character. Our assumption on $N$ implies that each of these characters is a Teichm\"uller lift of a character valued in a finite extension $\bF$ of $\bF_p$ generated by the lifts. Abusing notation, we also use $\theta,\chi,\omega$ to refer to these characters. 

We assume that $\theta$ satisfies the following conditions,
\begin{list}{$\bullet$}{}
\item $\theta$ is primitive,
\item if $\chi|_{(\Z/p\Z)^\times}=1$, then $\chi|_{(\Z/N\Z)^\times}(p) \ne 1$, and 
\item if $N=1$, then $\theta \ne \omega^2$.
\end{list}
These assumptions, along with $p \nmid N \phi(N)$, allow us to apply the work of Ohta \cite{ohta2003}. 

A subscript $\theta$ or $\chi$ on a module refers to the eigenspace for an action of $(\Z/Np\Z)^\times$. A superscript $\pm$ will denote the $\pm 1$-eigenspace for complex conjugation. Let $S$ denote the set of primes dividing $Np$ along with the infinite prime, and let $G_{\bQ,S}$ be the unramified outside $S$ Galois group of $\bQ$.

Let $\Lambda=\Z_p\lb\Z_{p,N}^\times\rb_\theta$, a regular local ring with residue field $\bF$, where $\Z_{p,N}=\Z_p \times (\Z/N\Z)$. Fix a system $(\zeta_{Np^r})$ of primitive $Np^r$-th roots of unity such that $\zeta_{Np^{r+1}}^p=\zeta_{Np^r}$ for all $r$, let $\Q_\infty = \Q(\zeta_{Np^\infty})$. Let $\Gamma = \Gal(\Q_\infty/\Q(\zeta_{Np}))$, $\X=\Gal(M/\Q_\infty)$ and $X=\Gal(L/\Q_\infty)$ where $M$ and $L$ are the maximal abelian pro-$p$ extensions of $\Q_\infty$ that are, respectively, unramified outside $Np$ and unramified. Let $\xi_\chi\in \Lambda$ be a characteristic power series for $X_\chi(1)$ (it is related to the Kubota-Leopoldt $p$-adic $L$-function for $\chi$), and assume that $\xi_\chi \not \in \Lambda^\times$.

Let $\fH = \fH_\theta$ and $\h = \h_\theta$ be the Eisenstein components of the Hida Hecke algebras  for modular forms and cusp forms, respectively. Let $\I$ and $I$ be their respective Eisenstein ideals. Let $I_\fH$ be the preimage of $I$ in $\fH$.  Let 
\[
H = \varprojlim_r H^1(X_1(Np^r),\Z_p)_\theta^\ord
\]
be the cohomology group on which $\h$ acts.

\subsection{Main results}
\label{subsec: main results}
We now introduce the weakly Gorenstein conjecture. It is known that $\h/I \cong \Lambda/\xi_\chi$. Let $f_\chi$ be a prime divisor of $\xi_\chi$, and let $\p \subset \h$ be the kernel of the composite $\h \to \Lambda/\xi_\chi \to \Lambda/f_\chi$; note that $\p$ has height $1$, that $I \subset \p$, and that $p \not \in \p$. By abuse of notation, we let $\p \subset \fH$ be the inverse image of $\p$ in $\fH$. One of us has conjectured that $\h_\p$ and $\fH_\p$ are Gorenstein \cite[Conj.\ 1.2]{wake2}. We call $\fH$ (resp.\ $\h$) \emph{weakly Gorenstein} when this is true for all such $\p$. 

An important consequence of our main result is the following. 

\begin{thmA}
Assume that $X_{\theta,(f_\chi)}=0$. Then $\fH_\p$ is Gorenstein.
\end{thmA}

\begin{rem}
\label{rem: Xtheta=0 is weak}
Note that $X_{\theta,(f_\chi)}=0$ is a relatively weak assumption: it is conjectured by Greenberg that $X^+$ (which contains $X_\theta$) has finite cardinality (see Conjecture \ref{conj: greenberg}), and so any localization of $X_\theta$ at a characteristic $0$ prime of $\Lam$ is $0$. Moreover, it is known that the support of $X_\theta$ is contained in the support of $X^\#_\chii(1)$, which is $\Spec \Lam/^*\xi_{\chii}$ for some element $^*\xi_\chii \in \Lambda$. For $X_{\theta,(f_\chi)}$ to be non-zero, $X_\theta$ would have to be infinite \emph{and} $f_\chi$ would have to be a common divisor of $\xi_\chi$ and $^*\xi_\chii$. 
\end{rem}

\begin{rem}
\label{rem: Xtheta=0 is necessary}
The main result of \cite{wake2} is that $X_{\theta,(f_\chi)}=0$ is a necessary condition for $\fH_\p$ to be Gorenstein, and that it is also sufficient if Sharifi's conjecture is true. Our main result says that $X_{\theta,(f_\chi)}=0$ is sufficient without assuming Sharifi's conjecture. In fact, we use our result to prove (a weak version of) Sharifi's conjecture.
\end{rem}

Along with Theorem A, we show that if Greenberg's conjecture holds, then $\h_\p$ is also Gorenstein, and hence that the weakly Gorenstein conjecture of \cite{wake2} follows from Greenberg's conjecture. In fact, we show that if Greenberg's conjecture holds, then both $\fH_\p$ and $\h_\p$ are complete intersection rings, and that their Eisenstein ideals are principal. See Theorem \ref{thm: eigencurve equivalences}.

Theorem A has consequences for Sharifi's conjecture. Sharifi has constructed a map
\[
\Upsilon: X_\chi(1) \lra H^-/I H^-
\]
which he conjectures to be an isomorphism. He also constructed a map 
\[
\varpi:  H^-/IH^- \lra  X_\chi(1)
\]
in the opposite direction, which he conjectures to be the inverse of $\Upsilon$. (See \S\ref{subsec: sharifi's conj} for references and more discussion.) Theorem A directly implies 

\begin{corB}
\begin{enumerate}[leftmargin=2em]
\item If $X_{\theta,(f_\chi)}=0$, then the localization of $\Upsilon$ at $(f_\chi)$ is an isomorphism. 
\item If $X_\theta \otimes_\Lambda (\Lambda/\xi_\chi)$ is finite, then $\U$ is injective and has finite cokernel.
\end{enumerate}
\end{corB}

This corollary provides strong evidence for Sharifi's conjecture -- it implies that the map $\U$ is an isomorphism (up to $p$-torsion) under the assumption of Greenberg's conjecture. Fukaya and Kato \cite{FK2012} have also made great progress towards Sharifi's conjecture, by entirely different methods, proving parts of the conjecture under different assumptions. Unlike our Corollary B, they can also prove results about the map $\varpi$. Combining the results of \cite{FK2012} with ours, we can prove that both maps of Sharifi's conjecture are isomorphism (up to $p$-torsion) under the assumption of Greenberg's conjecture. 

\begin{corC}
Assume Greenberg's Conjecture \ref{conj: greenberg}. Then the induced maps
$$
\Upsilon: X_\chi(1) \lra (H^-/I H^-)/\mathrm{(tor)}, \ \  \varpi: ( H^-/IH^-)/\mathrm{(tor)} \lra  X_\chi(1)
$$
are isomorphisms. Here $\mathrm{(tor)} \subset H^-/I H^-$ is the $p$-power-torsion subgroup.
\end{corC}
\begin{rem}
Since $X_\chi(1)$ is $p$-torsion-free by Ferrero-Washington \cite{FW1979}, Sharifi's conjecture would imply that $\mathrm{(tor)}=0$.
\end{rem}

This is the first result to show that both maps of Sharifi's conjecture are isomorphisms (up to $p$-torsion) under the assumption of a well-established conjecture. For a statement of Corollary C involving the precise form of Greenberg's conjecture we need, see Corollary \ref{cor: corC stronger}. Corollary B appears as Corollary \ref{cor: Upsilon iso main}.

\subsection{Ordinary pseudorepresentations} 
\label{subsec:OP}
We prove Theorem A by developing new techniques in Galois deformation theory. In this part of the introduction, we discuss these techniques, which are of independent interest. Since the reader interested in Iwasawa theory and Sharifi's conjecture may be unfamiliar with deformation theory and pseudorepresentation theory, we provide context in \S\ref{subsec: summary of deformation theory}. Such a reader may like to consult \S\ref{subsec: summary of deformation theory} before proceeding.

We prove Theorem A by constructing a universal ordinary pseudodeformation ring and using the numerical criterion of Wiles \cite[Appendix]{wiles1995} to prove an $R=\mathbb{T}$ theorem comparing the ordinary pseudodeformation ring to $\fH_\p$. We use a version of the numerical criterion due to Lenstra \cite{lenstra1993}.

Recall that a representation of $G_{\Q_p} := \Gal(\overline{\bQ}_p/\bQ_p)$ on a $2$-dimensional $p$-adic vector space $V$ is {\em ordinary} if there exists a $1$-dimensional quotient representation $V \onto W$ such that $W(1)$ is unramified. A representation $\rho$ of $G_{\bQ,S}$ is {\em ordinary} if the restriction to a decomposition group $\rho\vert_{G_{\Q_p}}$ is ordinary. (This notion of ``ordinary'' is restrictive compared to some other contexts where the same term is used.) 

A \emph{pseudorepresentation} $D:G \to A$ of a group $G$ is a collection of polynomials over a commutative ring $A$, one for each group element, that satisfy the same algebraic conditions that the collection of characteristic polynomials of an $A$-linear representation of $G$ must (see Definition \ref{defn:psrep} for the precise definition). Given a representation $V$ of $G$, the actual collection of characteristic polynomials gives a pseudorepresentation which we denote by $\psi(V)$. When $A$ is a field, the pseudorepresentation $\psi(V)$ depends only on the semi-simplification $V^{ss}$ of $V$. Consequently, it may initially appear that pseudorepresentations are too crude to capture the ordinary condition, which depends on the order of the composition factors. However, extensions between $G_{\Q_p}$-characters often become visible in $G_{\bQ,S}$-pseudorepresentations, and this allows us to impose the ordinary condition on $G_{\bQ,S}$-pseudorepresentations. 

Our definition of ordinary pseudorepresentation of $G_{\bQ,S}$ may be be thought of, when the coefficient ring is a field, to be the na\"ive definition, ``a pseudorepresentation for which there exists an ordinary representation inducing it.'' Indeed, with field-valued coefficients $F$, every pseudorepresentation $D: G \ra F$ is associated to a unique semi-simple representation $\rho^{ss}_D$ after a finite extension of $F$ (Theorem \ref{thm:PsR_ss}). The following example explains our definition of ordinary pseudorepresentation in the field-valued case. Proofs of the statements in the example can be found in \S\ref{sec: ord C-H}.

\begin{exmp}
Let $F/\bQ_p$ be a finite extension, and let $D: G_{\bQ,S} \to F$ be a $2$-dimensional pseudorepresentation realized by the representation $\rho^{ss}_D: G_\bQ \ra \GL_2(F)$. Then $D$ is ordinary in our sense if and only if $\rho^{ss}_D$ is ordinary. When $\rho^{ss}_D$ is reducible, this means that at least one of the two composition factors of $\rho^{ss}_D(1)$ is unramified at $p$, and this can be seen by $D\vert_{G_{\Q_p}}$. When $\rho^{ss}_D$ is irreducible, it retains all information about composition series of $\rho_D^{ss}\vert_{G_{\Q_p}}$, even though $D\vert_{G_{\Q_p}}$ only knows Jordan-H\"older factors of $\rho_D^{ss}\vert_{G_{\Q_p}}$. 
\end{exmp}

Another serious obstacle arises when the coefficient ring is not a field, which we must address because we want to deform to rings such as $F[\varepsilon]/(\varepsilon^2)$. In this generality, not every pseudorepresentation is induced by some representation, rendering the na\"ive definition useless. We solve this problem by broadening the category of representations to include what are called \emph{Cayley-Hamilton representations}, which we learned from \cite[\S1.22]{chen2014} and which was adapted for use with representations of profinite groups in \cite[\S3.2]{WE2017}. They are defined in Definition \ref{defn:C-H_rep}. Every pseudorepresentation arises from a Cayley-Hamilton representation. We show in \S\ref{sec: ord C-H} that the ``ordinary'' condition can be reasonably imposed on Cayley-Hamilton representations (using a generalized matrix algebra structure in the sense of \cite[\S1]{BC2009}), after which we can define ordinary pseudorepresentations to be those for which there exist ordinary Cayley-Hamilton representations inducing them. 

The following theorem is an important case of what we prove in \S\ref{subsec:RDb}.
\begin{thmD} 
Let $\F$ be a field, and let
$
\chi_1, \  \chi_2: G_{\bQ,S} \to \F^\times
$ 
be characters such that $\chi_1\omega$ is unramified at $p$ and $\chi_2\omega$ is not. Let $\bar{D}$ be the pseudorepresentation associated to $\chi_1 \oplus \chi_2$. Then the ordinary deformation functor of $\bar{D}$ is representable by a quotient $R_{\bar{D}}^{\ord}$ of the universal pseudodeformation ring $R_{\bar{D}}$.
\end{thmD}

\subsection{Pseudo-modularity} 

From work of Ohta and Sharifi \cite{ohta2005, sharifi2011}, we deduce that $\End_\h(H)$ has a natural Cayley-Hamilton algebra structure, with $G_{\bQ,S}$-action making the Cayley-Hamilton representation $G_{\bQ,S} \to \Aut_{\h_\p}(H_\p)$ ordinary. The resulting $G_{\bQ,S}$-pseudorepresentation $D_\p:G_{\bQ,S} \to \h_\p$ is then ordinary by definition, with residual pseudorepresentation $\bar{D}_\p$. 

Write $R_\p^\ord$ for ordinary pseudodeformation ring for the residual pseudorepresentation $\bar{D}_\p$. By construction, there is a canonical surjection $R_{\p}^\ord \onto \hat\h_\p$ onto the completion of $\h_\p$ corresponding to $D_\p$, from which we also construct a surjective map $R_{\p}^\ord \onto \hat\fH_\p$. 

We show that the tangent space of $R_{\p}^\ord$ can be controlled in terms of Galois cohomology. Using the Iwasawa Main Conjecture, and under the assumption $X_{\theta,(f_\chi)}=0$, we are able to verify the conditions of the numerical criterion of Wiles and Lenstra \cite{wiles1995,lenstra1993}. This pseudo-modularity theorem is our main result. It implies Theorem A.

\begin{thmE}
Assume $X_{\theta,(f_\chi)}=0$. Then the map 
$$
R_\p^\ord \onto \hat \fH_\p
$$
is an isomorphism and both rings are complete intersection. Moreover, the image of $\hat \h_\p[G_{\bQ,S}]$ in $\End_{\hat\h_\p}(\hat H_\p)$ is the universal ordinary cuspidal Cayley-Hamilton algebra.
\end{thmE}

For the precise statement of second statement in Theorem E, see Theorem \ref{thm: GMA R=T}. See \S\ref{subsec: comm alg} for a discussion of Gorenstein and complete intersection rings.

\begin{rem}
\label{rem:no taylor-wiles}
By Remark \ref{rem: Xtheta=0 is necessary}, if $X_{\theta,(f_\chi)} \ne 0$, then $\fH_\p$ is not Gorenstein, so it cannot be complete intersection. Since the numerical criterion for proving $R=\mathbb{T}$ proves complete intersection as a byproduct, it seems unlikely that we can remove the assumption $X_{\theta,(f_\chi)}=0$. 

In particular, while we use the numerical criterion of Wiles, we do not use the Taylor-Wiles method, and we do not believe that the Taylor-Wiles method can be used to improve this result. Of course, we expect that $X_{\theta,(f_\chi)}=0$ always (see Remark \ref{rem: Xtheta=0 is weak}).  
\end{rem}

\subsection{Outline}
\label{subsec: outline}
In \S\ref{sec: iwasawa}, we review some classical Iwasawa theory related to the conditions in the main theorems. In \S\ref{sec: modular}, we discuss known results on the structure of the cohomology group $H$, with a view towards using $H$ to produce an ordinary pseudorepresentation. In \S\ref{sec: gorenstein}, we prove some sufficient conditions for the weakly Gorenstein conjecture. We begin \S\ref{sec: ord C-H} with a review of Galois deformation theory for non-experts. Then we introduce the notation of ordinary Cayley-Hamilton representation and ordinary pseudorepresentation, and construct universal objects in these categories to prove Theorem D. In \S\ref{sec:galois_cohom}, we compute Galois cohomology. In \S\ref{sec: modularity}, we compare universal ordinary objects with modular objects, using the numerical criterion to prove Theorem A and Theorem E. In \S\ref{sec:sharifi}, we deduce our results towards Sharifi's conjecture, Corollary B and Corollary C, and also our results on the geometry of the ordinary eigencurve.

\subsection{Acknowledgements} 
We would like to thank the organizers of the Glennfest conference at Boston University for providing the environment where this collaboration was initiated, and, in particular, Jo\"el Bella\"iche and Rob Pollack for suggesting that we collaborate. We would also like to recognize the hospitality of the mathematics departments at the University of Chicago, Brandeis University, and UCLA. We thank Patrick Allen, Jo\"el Bella\"iche, Frank Calegari, Brian Conrad, Matt Emerton, Takako Fukaya, Kazuya Kato, Matsami Ohta, and Romyar Sharifi for helpful conversations. We are grateful to the anonymous referee for supplying many helpful comments, and we also recognize Eric Urban for comments that helped us to improve the main result.

P.W.~was supported by the National Science Foundation under the Mathematical Sciences Postdoctoral Research Fellowship No.~1606255. C.W.E.~would like to thank the AMS and the Simons Foundation for support for this project in the form of an AMS-Simons Travel Grant.

\subsection{Notation and conventions} 
If $K$ is a field, let $G_K$ denote the absolute Galois group of $K$. In particular, we fix $\overline{\bQ} \rinj \overline{\bQ}_p$, allowing us to  treat $G_{\bQ_p}$ as a decomposition group for $p$ with a map $G_{\bQ_p} \ra G_\bQ$. If $\rho$ is a representation, we let $\psi(\rho)$ denote the associated pseudorepresentation. 

Unless remarked otherwise, ``ring" means commutative ring, and ``algebra" means associative (but not necessarily commutative) algebra. Rings and algebras are assumed to be unital, and morphisms are assumed to preserve the unit element.

If $R$ is a ring, $M$ is an $R$-module, and $r \in R$, we let $_r M$ denote the $r$-torsion submodule of $M$. We let $M\{r\} = \displaystyle \cup_{n\ge 1} (_{r^n} M)$. We sometimes use $M/r$ to denote $M/rM$ to save space. We let $Q(R)$ denote the total quotient ring of $R$ -- that is, the localization $S^{-1}R$ where $S$ is the set of non-zero divisors. 

In general, we reserve the notation $M \cong N$ for a \emph{canonical} isomorphism between objects $M$ and $N$. If there is no canonical choice of isomorphism, we write $M \simeq N$ instead.

\section{Preliminaries on Iwasawa theory}
\label{sec: iwasawa}
In this section, we discuss Iwasawa theory for cyclotomic fields. Recall the notations for class groups and cyclotomic fields established in \S\ref{subsec:setup}.

Let $\Cl(\Q(\zeta_{Np^r}))$ be the class group. By class field theory, there is an isomorphism
\[
X \cong \varprojlim \Cl(\Q(\zeta_{Np^r}))\{p\}.
\]
There is a continuous action of $\Gamma$ on $X$.

For $z \in \Z_{p,N}^\times$, let $[z] \in \Z_p\lb\Z_{p,N}^\times\rb$ denote the corresponding group-like element. Notice that the action on $(\zeta_{Np^r})$ gives an isomorphism $\Gamma \simeq \ker(\Z_p^\times \to (\Z/p\Z)^\times)$. 

Note that $\Z_p\lb\Z_{p,N}^\times\rb$ is a semilocal ring, while $\Lam \simeq \sO\lb\Gamma\rb \simeq \sO\lb T\rb$ where $\sO$ is the valuation ring of the extension of $\Q_p$ generated by the values of $\theta$. This makes it easier to work with $\Lam$ rather than $\Z_p\lb\Z_{p,N}^\times\rb$.

Let $M \mapsto M^\#$ and $M \mapsto M(r)$ be the functors on $\Z_p\lb\Z_{p,N}^\times\rb$-modules as defined in \cite[\S2.1.3]{wake2}.  Namely, $M^\#$ and $M(r)$ are identical as $\bZ_p$-modules to $M$, and $z \in \Z_{p,N}^\times$ acts on $M^\#$ as $z^{-1}$ acts on $M$. Such $z$ acts on $M(r)$ as $\bar z^r[z]$ acts on $M$, where $\bar z$ denotes the projection of $z$ to $\bZ_p^\times$. Especially when using duality, we are forced to consider $\Z_p\lb\Z_{p,N}^\times\rb$-modules that are isotypical for characters other than $\theta$, but we use these functors to make the actions factor through $\Lam$ so we can treat all modules uniformly. 

We define $\xi_\chi, ^*\xi_\chii \in \Lam$ to be generators of the principal ideals $\Ch(X_\chi(1))$ and $\Ch(X_\chii^\#(1))$, respectively. By the Iwasawa Main Conjecture, these may be taken to be power series associated to Kubota-Leopoldt $p$-adic $L$-functions. By the theory of Iwasawa adjoints, $\Ch(\X_\theta)=(^*\xi_\chii)$ (see, for example, \cite[Prop.\ 2.2]{wake2}).

\subsection{Units and finiteness conjectures}
\label{subsec:greenberg}

We explain the relationship between $X$ and $\X$, as well as their relationships to groups of units. We then recall some conjectures about the finiteness and cyclicity of class groups. 

There is a quotient map $\X \onto X$, and, by class field theory, the kernel measures the difference between local units and global units. More precisely, for a primitive character $\psi$ of $(\Z/Np\Z)^\times$, there is an exact sequence
\begin{equation}\label{eq: euxx}
0 \lra \E_\psi \lra U_\psi \lra \X_\psi \lra X_\psi \lra 0, 
\end{equation}
where $U$ is the pro-$p$ part of $\varprojlim (\Z[\zeta_{Np^r}]\otimes \Z_p)^\times$, and $\E$ is closure of the image of the global units in $U$.

The structure of the local units $U$ was studied by Iwasawa, and later by Coleman, using the theory of Coleman power series. The following result is the known as Iwasawa's theorem. For the original work of Iwasawa, see \cite{iwasawa1964}; for Coleman power series, see \cite{coleman1979}; for the abelian number field case see, for example, \cite{greither1992}. Note that we use the Main Conjecture to relate $p$-adic $L$-functions to $^*\xi_\chii$.

\begin{thm}
There is an isomorphism 
\[
U_\theta \lrisom \Lambda.
\]
It sends the group of circular units $C_\theta \subset \E_\theta$ to the ideal $(^*\xi_\chii) \subset \Lambda$.
\end{thm}

We will use the following corollary, which is well-known.

\begin{cor}
\label{cor: units}
There is an exact sequence of $\Lambda$-modules 
\[
0 \lra \E_\theta/C_\theta \lra \Lambda/^*\xi_\chii \lra \X_\theta \lra X_\theta \lra 0.
\]
In particular, $\Ch(\E_\theta/C_\theta) = \Ch(X_\theta)$.

For any height $1$ prime $(f) \subset \Lambda$, the $\Lambda_{(f)}$-module $\E_{\theta,(f)}$ is cyclic.
\end{cor}
\begin{proof}
The first statement is immediate from the theorem and the sequence (\ref{eq: euxx}). The second statement follows from this since $\Ch(\X_\theta)=(^*\xi_\chii)$. For the final statement, notice that $\E_{\theta,(f)} \subset U_{\theta,(f)}$, which is a free $\Lambda_{(f)}$-module of rank $1$ by the theorem. The statement follows since $\Lambda_{(f)}$ is a PID.
\end{proof}

Recall that two $\Lam$-modules $M$ and $M'$ are said to be {\em pseudoisomorphic} if there is a morphism $M \to M'$ with finite kernel and cokernel. Pseudoisomorphism is an equivalence relation on torsion $\Lambda$-modules. A $\Lam$-module $M$ is said to be {\em pseudocyclic} if $M$ is pseudoisomorphic to a cyclic module. 

\begin{conj}[Kummer-Vandiver]
Assume $N=1$. Then $X^+=0$.
\end{conj}

By Corollary \ref{cor: units}, we see that if $X^+=0$, then $\Lam/\xi_\chii \to \X_\theta$ is an isomorphism. Using Iwasawa adjunction, this implies that $X_\chii$ is cyclic. For general $N>1$, we don't expect that $X_\chii$ is cyclic, but we do expect it to be pseudocyclic. This expectation follows from the following statement, which is known as Greenberg's conjecture.

\begin{conj}[{\cite[Conj.\ 3.4]{greenberg2001}}]
\label{conj: greenberg}
For any $N$, $X^+$ is of finite cardinality. In particular, $X^-$ is pseudocyclic.  
\end{conj}

The second sentence follows from the first by Corollary \ref{cor: units}, as above. There is the following relation with multiple roots of the $p$-adic zeta function.

\begin{lem}\label{multi roots and cyclicity lem}
If $\xi_\chi$ has no prime factors that occur with multiplicity greater than $1$, then $X_\chi(1)$ is pseudocyclic. If $f$ is a prime factor of $\xi_\chi$ that occurs with multiplicity $1$, then $X_\chi(1) \otimes_{\Lam} Q(\Lam/(f))$ has dimension $1$.
\end{lem}
\begin{proof}
By the structure theorem for Iwasawa modules \cite[Thm.\ 5.3.8, pg.\ 292]{NSW2008}, there is a pseudoisomorphism
\[
X_\chi(1) \lra \bigoplus_i \Lam/f_i
\]
with the property that $\prod_i f_i = \xi_\chi$. If $\xi_\chi$ has no prime factors that occur with multiplicity greater than $1$, then the right hand side is a cyclic $\Lam$-module, so this gives the first statement. If $f$ is a prime factor of $\xi_\chi$ that occurs with multiplicity $1$, then $f$ divides exactly one of the $f_i$ exactly once, and the second statement follows.
\end{proof}

\section{Modular forms and Hecke algebras}
\label{sec: modular}
We review the results of Hida, Ohta, Sharifi, and Fukaya-Kato on the structure of the ordinary integral cohomology of modular curves in preparation to match their Hecke module and Galois module structures with universal ordinary Galois modules. 

The exact form of the results of this section depend on several choices -- the model of the modular curve, Hecke operators versus dual Hecke operators, etc.\ -- and every author seems to make a different set of choices. We choose to follow the choices made by Fukaya-Kato in \cite{FK2012}, since that paper has a very thorough treatment of the subject. We refer the reader to \cite[\S1]{FK2012} for further details.

\subsection{Preliminaries on Hecke algebras and modular forms}
\label{subsec: H prelims}
Let $Y_1(Np^r)$ be the moduli space of elliptic curves together with a point of order $Np^r$. It is a smooth curve over $\Q$ (note that this is a different model for the modular curve than the one used in the works of Ohta cited below). Let $X_1(Np^r)$ be the compactification by adding cusps. 

Let $\tH', H'$ be the ordinary parts of the \'etale cohomology of the modular curves
\begin{equation}
\label{eq:H_limit}
\tH' = \varprojlim H^1_{\text{\'et}}(Y_1(Np^r) \otimes_\bQ \overline{\bQ},\Z_p)_\theta^{\ord}, \qquad H' = \varprojlim H^1_{\text{\'et}}(X_1(Np^r) \otimes_\bQ \overline{\bQ},\Z_p)_\theta^{\ord}
\end{equation}
where the $\theta$-eigenspace is taken for the action of the diamond operators. Let $\fH'$ and $\h'$ be the corresponding algebras of dual Hecke operators.

There is a unique maximal ideal containing the Eisenstein ideals of $\fH'$ and $\h'$ respectively. Let $\fH$ and $\h$ be the localizations at each Eisenstein maximal ideal. Let $\tH = \tH'\otimes_{\fH'} \fH$ and $H=H' \otimes_{\h'} \h$. Let $\I \subset \fH$ and $I \subset \h$ be the Eisenstein ideals.

Let $M_2(Np^r)_{\Z_p}$ and $S_2(Np^r)_{\Z_p}$  be the spaces of modular and cuspidal forms, respectively, of weight $2$ and level $Np^r$ with coefficients in $\Z_p$. Let
\begin{equation}
\label{eq:MF_limits}
M_\Lambda = (\varprojlim M_2(Np^r)^\ord_{\Z_p}) \otimes_{\fH'} \fH \text{ and } S_\Lambda = (\varprojlim S_2(Np^r)^\ord_{\Z_p}) \otimes_{\h'} \h,
\end{equation}
which we think of as ordinary $\Lambda$-adic forms. Here the inverse limit is over the trace maps on modular forms.

There is a homomorphism $\Lam \to \fH$ induced by sending $a \in \Z_{p,N}^\times$ to the diamond operator $\dia{a}$. This homomorphism and the composite $\Lam \to \fH \to \h$ are injective, and we sometimes think of $\Lambda \subset \h, \fH$ as the the subalgebra of diamond operators. 

Hida's control theorem \cite[\S3]{hida1986a}, \cite[Thm.\ 3.1]{hida1986} states that $\fH$, $\h$, $H$, $\tH$, $S_\Lambda$, or $M_\Lambda$ are all free $\Lambda$-modules of finite rank, and that the finite level versions may be recovered using the $\Lambda$-action. Hida's duality theorem \cite[\S2]{hida1986a} states that the maps
\[
\h \times S_\Lambda \lra \Lambda, \qquad \fH \times M_\Lambda \lra \Lambda
\]
given by $(T,f) \mapsto a_1(Tf)$ are perfect pairings of $\Lambda$-modules. (Note that our assumptions on the character $\theta$ in \S\ref{subsec:setup} imply that the Eisenstein series in $M_\Lambda$ has constant coefficient in $\Lambda$, so there is no need for denominators in the statement of the duality.) For a $\Lambda$-module $M$, let $M^\vee = \Hom_\Lambda (M,\Lambda)$ -- we have  $S_\Lambda \cong \h^\vee$ and $M_\Lambda \cong \fH^\vee$.

The following theorem is due to Ohta {\cite{ohta2000}} and can be thought of as a $p$-adic version of the Eichler-Shimura isomorphism from the complex setting. Ohta's statement is slightly different since he uses a different model of modular curve; the exact statement below can be found in \cite[\S1.7]{FK2012}.

\begin{thm}
\label{e-s theorem}
There are canonical exact sequences of $\fH[G_{\Q_p}]$-modules
\[
0 \lra \tH_{sub} \lra \tH \lra \tH_{quo} \lra 0
\]
and
\[
0 \lra H_{sub} \lra H \lra H_{quo} \lra 0 
\]
characterized by the following properties:
\abcs
\item The actions of $G_{\Q_p}$ on $\tH_{quo}(1)$ and $H_{quo}(1)$ are unramified.
\item There are isomorphisms $\tH_{quo} \simeq M_\Lambda \cong \fH^\vee$, $H_{quo} \simeq S_\Lambda \cong \h^\vee$, and $H_{sub} \simeq \h$ of $\fH$-modules. 
\endabcs
Moreover, the inclusion map $H_{sub} \to \tH_{sub}$ is an isomorphism.
\end{thm}

In particular, $H \otimes_\h Q(\h)$ is free of rank $2$ over $Q(\h)$. The Galois action on $H \otimes_\h Q(\h)$ gives a representation $\rho_H: G_\Q \to \GL_2(Q(\h))$. For a prime $q \nmid Np$, $\rho_H$ is unramified at $q$, and we have the usual formula
\begin{equation}
\label{eq:modular_psr}
\det(1-\mathrm{Fr}_q^{-1}t)=1-\dia{q}T^*(q)t+q \dia{q}t^2
\end{equation}
for $\mathrm{Fr}_q \in G_\Q$ an arithmetic Frobenius (see \cite[\S1.7.14]{FK2012}). In other words, we have $\det(\rho_H(\mathrm{Fr}_q))=q^{-1}\dia{q}^{-1}$ and $\Tr(\rho_H(\mathrm{Fr}_q))=q^{-1}T^*(q)$.

Let $\kcyc$ denote the $p$-adic cyclotomic character, and let $\dia{-}: G_\Q \to \h^\times$ be the character $\dia{\sigma}=\dia{a_\sigma}$, where $a_\sigma \in \Z_{p,N}^\times$ is defined by $\sigma(\zeta_{Np^r})=\zeta_{Np^r}^{a_\sigma}$. By the Chebotarev density theorem and the formulas for the trace and determinant of $\rho_H$ on Frobenius elements, we have the following lemma.
\begin{lem}
\label{lem: det and tr of H} 
The determinant of $\rho_H$ is $\det(\rho_H)=\kcyc^{-1}\dia{-}^{-1}$, and it is valued in the subalgebra $\Lambda \subset \h \subset Q(\h)$ of diamond operators. The trace of $\rho_H$ is valued in the subalgebra $\h \subset Q(\h)$.
\end{lem}

From this and Theorem \ref{e-s theorem}, we see that the inertia subgroup of $G_{\Q_p}$ acts through the character $\dia{-}^{-1}$ on $H_{sub}$.

The usual Poincar\'e duality on $H$ interchanges the Hecke operators with the dual Hecke operators. Ohta has constructed a twisted pairing which is equivariant for the Hecke action. 

\begin{thm}[{\cite[Cor.\ 4.2.8]{ohta1995}}]
\label{poincare thm}
There is a perfect pairing of free $\Lambda$-modules
\begin{equation}
\label{eq:ohta_pairing}
( \ , \ )_\Lambda: H \times H \to \Lambda
\end{equation}
satisfying the following for all $x, y \in H$
\abcs
\item For all $T \in \h$, $(Tx,y)_\Lambda=(x,Ty)_\Lambda$.
\item For all $\sigma \in G_\Q$,  $(\sigma x,\sigma y)_\Lambda=\det(\rho_H(\sigma))(x,y)_\Lambda$.
\endabcs
\end{thm}

\subsection{The cusp group} Ohta has analyzed the structure of the cusp group \cite[Thm.\ 1.5.5]{ohta2003}. We summarize his result in the following theorem. 

\begin{thm}[Ohta]
There is an exact sequence
\begin{equation}
\label{HtHL}
0 \lra H \lra \tH \buildrel\partial\over\lra \Lambda \lra 0.
\end{equation}
There is a canonical element $\zinf \in \tH$ such that $\partial(\zinf)=1$.

Under the Eichler-Shimura isomorphisms of Theorem \ref{e-s theorem}, the sequence (\ref{HtHL}) gives rise to the exact sequence
\begin{equation}\label{SML}
0 \lra S_\Lambda \lra M_\Lambda \lra \Lambda \lra 0
\end{equation}
where the map $M_\Lambda \to \Lambda$ is $f \mapsto a_0(f)$. 
\end{thm}

There is a $\Lambda$-adic Eisenstein series $E_\Lambda \in M_\Lambda$ with $a_0(E_\Lambda)=\xi_\chi$. This form is an eigenform for all Hecke operators, so the induced map $\fH \to \Lambda$ is a ring homomorphism. We define $\I = \ker(\fH \to \Lambda) = \mathrm{Ann}_\fH(E_\Lambda)$. We let $I$ denote the image of $\I$ in $\h$.

From the sequence \eqref{SML}, we see that, for any $f \in M_\Lambda$, there exist $g \in S_\Lambda$ and $a \in \Lambda$ such that $\xi_\chi f= g+ aE_\Lambda$. In particular, 
\begin{equation}\label{eq: ann eis}
\mathrm{Ann}_\fH(S_\Lambda) \cap \I=\mathrm{Ann}_\fH(M_\Lambda) = 0.
\end{equation}
We can now compare the Hecke algebras $\h$ and $\fH$.

\begin{prop}
\label{prop: fH and h}
The quotient rings maps from $\fH$ to $\fH/\I \cong \Lambda$ and $\h$ lie in a commutative diagram of $\fH$-modules with exact rows
\[
\xymatrix{
0 \ar[r] & \I \ar[r] \ar[d]^\wr & \fH \ar@{->>}[d] \ar[r] & \Lambda \ar@{->>}[d] \ar[r] & 0 \\
0 \ar[r] & I \ar[r] & \h \ar[r] & \Lambda/\xi_\chi \ar[r] & 0. 
}
\]
The map $\fH \to \Lambda$ is given by $T \mapsto a_1(TE_\Lambda)$.
\end{prop}
\begin{proof}
We break the proof into two lemmas.
\begin{lem}\label{lem: ker h to fh}
The ideal $\ker(\fH \onto \h)$ of $\fH$ is principal. It is generated by the unique element $T_0 \in \fH$ satisfying $a_1(T_0f)=a_0(f)$ for all $f \in M_\Lambda$.
\end{lem}
\begin{proof}
The dual of the map $\fH \onto \h$ under Hida duality is the natural inclusion $S_\Lambda \to M_\Lambda$. The result follows from the exact sequence \eqref{SML}.
\end{proof}

\begin{lem}\label{I=I cor}
The natural map $\I \onto I$ is an isomorphism of $\fH$-modules.
\end{lem}
\begin{proof}
Indeed, any element of the kernel is a multiple of $T_0$, so it must annihilate $S_\Lambda$. But it is also an element of $\I$, so it is zero by \eqref{eq: ann eis}.
\end{proof}

We now claim that the composite map $\fH \to \Lambda \to \Lambda/\xi_\chi$ factors through $\h$. Since $\h= \fH/T_0 \fH$, it suffices to show that $T_0$ is sent to $0$ in $\Lambda/\xi_\chi$. But this is true since $T_0$ is sent to $a_1(T_0 E_\Lambda) = a_0(E_\Lambda) = \xi_\chi$. This completes the proof.
\end{proof}
\begin{rem}
The first result of this type was proven by Mazur and Wiles \cite{MW1984}. Our proof closely follows the one given by Emerton \cite{emerton1999}, which has been generalized recently by Lafferty \cite{lafferty2015}. 
\end{rem}

This proposition can be restated using the pullback in the category of commutative rings, which will be useful in \S\ref{subsec:modCHrep}. If $f: A \to C$ and $g: B \to C$ are homomorphisms of commutative rings, then 
$$
A \times_C B = \{(a,b) \in A \times B\ | \ f(a)=g(b)\}.
$$
\begin{lem}
\label{lem:H_pullback}
The natural surjections $\fH \rsurj \h$ and $\fH \rsurj \Lambda$ induce an isomorphism $\fH \risom \h \times_{\Lambda/\xi_\chi} \Lambda$. In particular, $\Ann_{\fH}(\I)=\ker(\fH \to \h)$ and $\Ann_{\fH}(\ker(\fH \to \h))=\I$.
\end{lem}
\begin{proof}
The pullback statement follows from chasing the diagram in the proposition. The statement about annihilators is a formal consequence.
\end{proof}

\subsection{Drinfeld-Manin modification}
\label{subsec:dm modification}
For a $\fH$-module $M$, let $M_{DM}=M \otimes_\fH \h$. This name comes from a relation with the Drinfeld-Manin splitting (see \cite[Lem.\ 4.1]{sharifi2011} or \cite[Lem.\ 6.2.3]{FK2012}), which is not used here. 

Using the isomorphisms $\Lambda \cong \fH/\I$ and $\h/I \cong \Lambda/\xi_\chi$, we see that $\Lambda_{DM}\cong \Lambda/\xi_\chi$. Tensoring (\ref{HtHL}) with $\h$, we obtain an exact sequence
\begin{equation}\label{HHDML}
0 \lra H \lra \HDM \lra \Lambda/\xi_\chi \lra 0.
\end{equation}
(Exactness on the left follows from \eqref{eq: ann eis}.) In particular, if we let $\zinfdm \in \HDM$ denote the image of $\zinf$, then $\xi_\chi \zinfdm \in H$, and, for any $T \in I$, we have $T\zinfdm \in H$.

\subsection{Global Galois action} 

The following theorem essentially follows from Sharifi's study of the global Galois action on $H$ \cite[Thm.\ 4.3]{sharifi2011}. It appears as stated in \cite[\S6.3]{FK2012}, where there is a simple and self-contained proof. 
\begin{thm}
\label{thm:H/I_sub}
The subgroup $H^-/IH^- \subset H/IH$ is a $\h[G_\Q]$-submodule, and the natural map $H^- \to H_{quo}$ is an isomorphism. The Galois group $G_\Q$ acts on $H^-/IH^-$ through the character $\kcyc^{-1}$.
\end{thm}

As a complement to Theorem \ref{thm:H/I_sub}, we sum up the conclusions of \S\ref{sec: modular} about $H$.
\begin{cor} 
\label{cor:H-structure}
The short exact sequence of $\h[G_{\bQ_p}]$-modules
$$
0 \lra H_{sub} \lra H \lra H_{quo} \lra 0
$$
of Theorem \ref{e-s theorem} splits as $\h$-modules, and there are isomorphisms 
\[
H^- \cong H_{quo} \simeq \h^\vee \quad \text{ and } \quad H^+ \simeq H_{sub} \simeq \h.
\]
Moreover, the $\h[G_{\bQ_p}]$-quotient $H_{quo}(1)$ is unramified at $p$. 
\end{cor}	
\begin{proof}
The splitting can be given by the isomorphism $H_{quo} \risom H^- \subset H$. The remaining content comes from Theorem \ref{e-s theorem}. 
\end{proof}

\section{Gorenstein and weakly Gorenstein Hecke algebras}
\label{sec: gorenstein}
In this section, we discuss ring-theoretic conditions on Hecke algebras. In particular, we recall the definition of weakly Gorenstein, and recall the weakly Gorenstein conjecture of \cite{wake2}. We also discuss an equivalent formulation of the conjecture in terms of the principality of the Eisenstein ideal. 

\subsection{Some commutative ring theory} 
\label{subsec: comm alg}
We recall some results from commutative algebra; a good reference for this is \cite{BH1993}. These statements are well-known, and we use them freely below, often without comment. We include references here for completeness. The proofs are relatively elementary, and the reader unfamiliar with the statements may like to derive their own proofs.

Let $R$ be a Cohen-Macaulay local ring of dimension $n$ with maximal ideal $\m$ (for example, $R$ is Cohen-Macaulay if it is finite flat over a regular subring \cite[Prop.\ 2.2.11, pg.\ 67]{BH1993}). Let $\hat{R}$ denote the $\m$-adic completion of $R$. An $R$-module $M$ is said to be a {\em dualizing module} (or \emph{canonical module}) if 
\[
\dim_{R/\m}\Ext^i_R(R/\m,M)=\delta_{in}
\]
(i.e. ~ it is $0$ for all $i\ne n$ and $1$ for $i=n$). A dualizing module is unique up to isomorphism if it exists \cite[Thm.\ 3.3.4, pg.\ 108]{BH1993}. Assume now that a dualizing module $M$ exists, and let $\p \subset R$ be a prime ideal and $\mathbf{x}=(x_1,\dots,x_m)$ be an $R$-regular sequence in $\m$.

The dualizing module $M$ enjoys the following permanence properties \cite[Thm.\ 3.3.5, pg.\ 109]{BH1993}: $M_\p$ is a dualizing module for $R_\p$; $M/\mathbf{x}M$ is a dualizing module for $R/\mathbf{x}R$; $\hat{M}$ is a dualizing module for $\hat{R}$. If $R \to S$ is a finite and flat (hence local) homomorphism of Cohen-Macaulay local rings, then $\Hom_R(S,M)$ is a dualizing module for $S$ \cite[Thm.\ 3.3.7, pg.\ 111]{BH1993}.

The ring $R$ is said to be {\em Gorenstein} if a dualizing module exists and is free of rank $1$ as an $R$-module. It follows from the above properties of dualizing modules that: if $R$ is Gorenstein, then $R_\p$ is Gorenstein; $R$ is Gorenstein if and only if $R/\mathbf{x}R$ is Gorenstein; $R$ is Gorenstein if and only if $\hat{R}$ is Gorenstein.

The ring $R$ is said to be \emph{complete intersection} if $\hat{R}$ is a quotient of a regular local ring by a regular sequence. Regular local rings are complete intersection, and $R$ is complete intersection if and only if $\hat{R}$ is complete intersection. The ring $R$ is complete intersection if and only if $R/\mathbf{x}R$ is complete intersection \cite[Thm.\ 2.3.4, pg.\ 75]{BH1993}. By the above properties of dualizing modules, complete intersection rings are Gorenstein.

For our purposes, an important consequence of the above is the following. If $R$ is a local ring, finite flat over a regular local subring $S$, then $R$ is Gorenstein if and only if $\Hom_S(R,S) \simeq R$ as $R$-modules.

We also use the following lemma, which gives a criterion for certain rings to be complete intersection. We thank Masami Ohta for pointing out the utility of this lemma.

\begin{lem}
\label{lem: LCI if I principal}
Let $(R,\m_R,k_R)$ be a local ring, and assume that there is be a regular local subring $(S,\m_S,k_S)$ of $R$ such that $R$ is a finite, free $S$-module. Let $I$ be an ideal of $R$ such that $\m_R=\m_SR+I$. If $I$ is principal, then $R$ is complete intersection.
\end{lem}
\begin{proof}
Let $\bar{R}=R/\m_S R$. Since $R$ is finite and flat over $S$, we have that $\dim(R)=\dim(S)$, and that any $S$-regular sequence in $S$ is an $R$-regular sequence. Since $S$ is regular, $\m_S$ is generated by an $S$-regular sequence, and so $\dim(\bar{R})=0$ and $R$ is complete intersection if and only if $\bar{R}$ is complete intersection. We will show that, if $I$ is principal, then $\bar{R}$ is complete intersection.

We know that $\bar{R}$ is an Artinian local ring that contains the field $k_S$, and that the image $\bar{I}$ of $I$ in $\bar{R}$ is the maximal ideal of $\bar{R}$. If $I$ is principal, then $\bar{R}$ must be complete intersection. Indeed, if $T$ is a generator of $\bar{I}$, then the homomorphism
\[
k_S\lb X\rb \to \bar{R}
\]
sending $X$ to $T$ is surjective, and the kernel is principal and so must be generated by a regular sequence.
\end{proof}

\subsection{Weakly Gorenstein} 
\label{subsec: weak gorenstein}
In \cite{wake1}, it is proven that $\fH$ is not Gorenstein in general. However, there is a weaker condition, called weakly Gorenstein, that both Hecke algebras $\fH, \h$ are conjectured to satisfy. The weakly Gorenstein condition was introduced in \cite{wake2}. 

Let $I_\fH$ be the kernel of the composite arrow $\fH \to \h \to \h/I$. Note that $\I \subsetneq I_\fH$. Let $\fP_\fH$ denote the set of height $1$ primes of $\fH$ that contain $I_\fH$. Let $\fP_\h$ denote the set of height $1$ primes of $\h$ that contain $I$. 

We say that $\fH$ (resp.\ $\h$) is {\em weakly Gorenstein} if $\fH_\p$ (resp.\ $\h_\p$) is Gorenstein for each $\p \in \fP_\fH$ (resp.\ $\p \in \fP_\h$).

The following conjecture is \cite[Conj.\ 1.2]{wake2}.

\begin{conj}\label{conj: weakly gorenstein}
Both $\h$ and $\fH$ are weakly Gorenstein.
\end{conj}

We now introduce some conditions that imply this conjecture.

\begin{lem}
The maps $\Spec \fH \to \Spec \Lambda$ and $\Spec \h \to \Spec \Lambda$ given by the inclusion of $\Lambda$ as diamond operators induce bijections
\[
\fP_\fH \lrisom \fP_\Lambda; \quad \fP_\h \lrisom \fP_\Lambda
\]
where $\fP_\Lambda$ is the set of height 1 prime divisors of $\xi_\chi$. 
\end{lem}

\begin{proof}
The canonical maps $\Lambda \to \fH$ and $\Lambda \to \h$ induce isomorphisms $\fH/I_\fH \cong \h/I \cong \Lambda/\xi_\chi$, giving bijections between the prime ideals containing $I_\fH$, $I$, and $\xi_\chi$, respectively. Since $\fH$ and $\h$ are Cohen-Macaulay (because they are finite flat over the regular local subring $\Lambda$), these bijections induce bijections on the height $1$ subsets (see \cite[Corollary 2.1.4]{BH1993}).
\end{proof}

Throughout this section, we may abuse notation and use the same letter $\p$ for an element of $\fP_\h$ or the corresponding element of $\fP_\fH$. Note that, by the Ferrero-Washington theorem \cite{FW1979}, if $\p \in \fP_\h$, then $p \not \in \p$. Similarly, if $\p \in \fP_\fH$, then $p \not \in \p$. 

\begin{lem}
\label{lem: weakly free over Lam}
Let $\p \in \fP_\h$, and let $(f)$ be the corresponding element of $\fP_\Lambda$. Then $\h_\p$ and $\fH_\p$ are free $\Lambda_{(f)}$-modules of finite rank.
\end{lem}
\begin{proof}
Let $R$ be either $\h$ or $\fH$. This essentially follows from Hida's theorem that $R$ is $\Lambda$-free of finite rank. The finite generation follows immediately. To see the freeness, note that by the previous lemma we have $\p \cap \Lambda = (f)$. Then the composite map $\Lambda \to R \to R_\p$ factors through $\Lambda_{(f)}$, and so $R_\p$ is $\Lambda_{(f)}$-module. Since $R$ is $\Lambda$-free, $f$ acts injectively on $R$ and thus on $R_\p$. So $R_\p$ is $\Lambda_{(f)}$-torsion free, and therefore $\Lambda_{(f)}$-free since $\Lambda_{(f)}$ is a DVR.
\end{proof}

We now introduce an equivalent formulation of Conjecture \ref{conj: weakly gorenstein} based on the Eisenstein ideal.

\begin{prop}\label{prop: both weak gorenstein}
The following are equivalent:
\abcs
\item Both $\fH_\p$ and $\h_\p$ are Gorenstein.
\item The ideal $I_\p \subset \h_\p$ is generated by a single non-zero divisor.
\item Both ideals $I_\p \subset \h_\p$ and $\I_\p \subset \fH_\p$ are principal.
\item Both $\fH_\p$ and $\h_\p$ are complete intersection.
\endabcs
\end{prop}
\begin{proof} The proof of $(1) \Rightarrow (2)$ is based on \cite[Thm.\ 3.3.8]{ohta2005}. By the same argument as \cite[Lem.\ 2.4]{wake1}, we see that if $\fH_\p$ is Gorenstein, then the map $I_\p \to H^-_\p$ given by $T \mapsto T\zinfdm$ is an isomorphism (recall the definition of $\zinfdm$ from \S \ref{subsec:dm modification}). If $\h_\p$ is also Gorenstein, then we have
$$
I_\p \lrisom H^-_\p \lrisom \h_\p
$$
which gives (2). 
 
The implication $(2) \Rightarrow (3)$ follows from Lemma \ref{I=I cor}. The implication $(3) \Rightarrow (4)$ follows from Lemma \ref{lem: LCI if I principal}. Indeed, by Lemma \ref{lem: weakly free over Lam}, $\h_\p$ and $\fH_\p$ are finite free over the regular subring $\Lambda_{(f)}$,  and since $\h_\p/I_\p \cong \Lambda_{(f)}/\xi_\chi\Lambda_{(f)}$ and $\fH_\p/\I_\p \cong \Lambda_{(f)}$, we see that the maximal ideal of $\h_\p$ (resp.\ $\fH_\p$) is $(f)+I_\p$ (resp.\ $(f)+\I_\p$).

The implication $(4) \Rightarrow (1)$ is a fact about commutative rings (see \S \ref{subsec: comm alg}).
\end{proof}

Finally, we will show that Conjecture \ref{conj: weakly gorenstein} takes on a particularly simple form when $\xi_\chi$ has no multiple roots. First, we need the following
\begin{lem}
\label{I_p max lem}
Let $\p \in \fP_\h$, and let $(f)$ be the corresponding element of $\fP_\Lambda$. Then the length of $\h_\p/I_\p$ over $\Lambda_{(f)}$ is equal to the multiplicity of $(f)$ in the prime factorization of $\xi_\chi$. In particular, the ideal $I_\p \subset \h_\p$ is maximal if and only if $(f)$ occurs with multiplicity $1$.
\end{lem}
\begin{proof}
We have
$$
\h_\p/I_\p = (\h/I)_\p \cong (\Lambda/\xi_\chi)_\p = (\Lambda/\xi_\chi)_{(f)} = \Lambda_{(f)}/\xi_\chi \Lambda_{(f)}.
$$
If $(f)$ occurs with multiplicity $r$, we have $\xi_\chi \Lambda_{(f)} = f^r \Lambda_{(f)}$.
\end{proof}

When $(f)$ occurs with multiplicity $1$, $I_\p$ is the maximal ideal of $\h_\p$, so it is principal if and only if $\h_\p$ is regular. 

\subsection{Plane singularity}
\label{subsec:plane singularity}
 One way to measure the ``badness" of a singular point $P$ on variety $X$ is to ask what is the minimal dimension $d$ needed in order to (locally around $P$) embed $X$ in a smooth variety of dimension $d$. Clearly $d \ge\dim(X)$ with equality if and only if $P$ is a regular point. The larger $d$ is, the ``worse" the singularity at $P$ is. We can state this formally as follows.

\begin{defn}
Let $X$ be a Noetherian scheme with $x \in X$ a closed point, and let $\hat{\sO}_{X,x}$ denote the complete local ring at $x$. The \emph{embedding dimension} $\mathrm{embdim}(x)$ of $x$ is the minimal $d$ such that $\Spec \hat{\sO}_{X,x}$ can be embedded in a regular scheme of dimension $d$. Equivalently, $d$ is the minimum dimension among all regular local rings that surject onto $\hat{\sO}_{X,x}$. If $\dim(X)=1$, we say $X$ has a \emph{plane singularity} at $x$ if $\mathrm{embdim}(x)=2$.

If $(R, \m)$ is a Noetherian local ring, then $\mathrm{embdim}(R)$ is the minimum dimension among all regular local rings that surject onto the completion of $R$. If $\dim(R)=1$, we say $R$ has a \emph{plane singularity} if $\mathrm{embdim}(R)=2$.
\end{defn}

Note that, by the Cohen Structure Theorem (\cite[Thm.\ A.21, pg.\ 373]{BH1993}), there is a Noetherian regular local ring that surjects onto $\hat{\sO}_{X,x}$, and so the embedding dimension is well-defined. As discussed in \cite[pg.~ 72-73]{BH1993}, the embedding dimension of a local ring $(R,\m)$ is $\dim(\m/\m^2)$, or, in other words, the minimal number of generators of $\m$. We now relate embedding dimension to the weak Gorenstein conjecture.

\begin{lem}
Let $\p \in \fP_\fH$. Then $\I_\p$ is principal if and only if $\mathrm{embdim}(\fH_\p)=2$.
\end{lem}
\begin{proof}
Let $(f)=\Lam \cap \p$ and let $\m$ denote the maximal ideal of $\fH_\p$. We first note that, by Proposition \ref{prop: fH and h}, there is are isomorphisms $\fH_\p \cong \I_\p \oplus \Lamfnochi$ and $\m \cong \I_\p \oplus f \Lamfnochi$ of $\Lamfnochi$-modules. From the first isomorphism it follows that $f\fH_\p \cap \I_\p= f \I_\p$ and from the second that $\m^2 \subset \I_p \oplus f^2\Lamfnochi$ and so $f \not \in \m^2$.

Since $\m$ is generated by $\I_p$ and $f$, if $\I_\p$ is principal, then $\m$ is generated by $2$ elements. Conversely, suppose that $\m$ can be generated by 2 elements. Since $f \not \in \m^2$, we can assume that $\m= f \fH_\p + g \fH_\p$ for some $g \in \fH_\p$. Then, since $f \fH_\p \cap \I_\p= f\I_\p$, we have
\[
\fH_\p \onto \m/f\fH_\p  = (\I_p + f\fH_\p)/f\fH_\p \isoto \I_\p / (f \fH_\p \cap \I_\p)= \I_\p/  f \I_\p
\]
where the first map is $1 \mapsto g$. By Nakayama's lemma, $\I_\p$ is principal.
\end{proof}

\subsection{Summary} We summarize the discussion of this section as a theorem.

\begin{thm}
\label{thm:simple_equiv}
Let $\p \in \fP_\h$. Consider the following conditions:
\abcs
\item $\h_\p$ is regular.
\item The ideal $I_\p \subset \h_\p$ is generated by a single non-zero divisor.
\item Both ideals $I_\p \subset \h_\p$ and $\I_\p \subset \fH_\p$ are principal.
\item $\mathrm{embdim}(\fH_\p)=2$.
\item Both $\fH_\p$ and $\h_\p$ are Gorenstein.
\item Both $\fH_\p$ and $\h_\p$ are complete intersection.
\endabcs
Then $(2)-(6)$ are equivalent; also, $(1) \Rightarrow (2)$. Moreover, if $(f) = \p \cap \Lambda$ occurs with multiplicity $1$ in the factorization of $\xi_\chi$, then $(2) \Rightarrow (1)$ and so $(1)-(6)$ are equivalent. 
\end{thm}

In particular, if $\xi_\chi$ has no multiple roots, then Conjecture \ref{conj: weakly gorenstein} is equivalent to the statement that $\h_\p$ is regular for all $\p \in \fP_\h$.

\section{Ordinary Pseudorepresentations}
\label{sec: ord C-H}

In this section, our goal is to construct and control an ordinary pseudodeformation ring. This is crucial to the proof of Theorem E, and we believe that it is the most novel part of this paper. 

This paper may be viewed as introducing new techniques in deformation theory of Galois representations to the study of Iwasawa theory. The reader interested in Iwasawa theory may not be familiar with deformation theory, so we begin with summary of this theory in \S \ref{subsec: summary of deformation theory}, which may provide context. 

\subsection{Summary of the deformation-theoretic approach}
\label{subsec: summary of deformation theory}
This is a deep and complex theory, and it is not our intention to provide an introduction to the subject in any technical way (many such introductions exist, for example \cite{mazur1989, mazur1997, bockle2013}). Instead, we give only the most simplified overview of the subject, so that the reader can see the parallels between our work and existing literature, and also see what is new to our theory.

\subsubsection{Modularity}
Deformation theory of $2$-dimensional Galois representations arises naturally in the study of the Langlands correspondence for $GL_2$. Very roughly, the Langlands correspondence establishes a bijection
\[
\{\text{Modular forms } f\} \longleftrightarrow \{\text{Galois representations } \rho: G_\Q \to \GL_2(\overline{\Q}_p)\}
\]
such that the $L$-function of $f$ matches with the $L$-function of the corresponding $\rho$. (Strictly speaking, we should add extra adjectives to both sides of the correspondence, but we ignore this here for simplicity of exposition.)

Work of Eichler, Shimura, Deligne, Serre and others established a map $f \mapsto \rho_f$ that is injective and respects $L$-functions. Galois representations in the image are called \emph{modular}. The proof of the Langlands correspondence, then, comes down to showing that all Galois representations $\rho$ (satisfying some conditions) are modular.

The strategy to do this is to break the problem into pieces depending on the reduction modulo $p$ of $\rho$. That is, any $\rho: G_\Q \to \GL(V)$ will fix a $\overline{\Z}_p$-lattice $T \subset V$, and the resulting representation $T \otimes \overline{\F}_p$ is called the \emph{residual representation} of $\rho$. By the Brauer-Nesbitt theorem, the semi-simplification of the residual representation does not depend on the choice of $T$. Then to prove modularity, one has to show that, for any fixed semi-simple representation $\bar{\rho}$ over $\overline{\F}_p$, all the representations $\rho$ whose residual semi-simplification is $\bar{\rho}$ are modular. This is where deformation theory comes in.

\subsubsection{Deformation functors} 
\label{subsub:def}

Now fix a representation $\bar{\rho}:G_{\bQ,S} \lra \GL(V_\bro) \simeq \GL_2(\bF)$, where $\bF$ is a finite extension of $\F_p$. We want to consider representations that are residually isomorphic to $\bro$, so we consider deformations of $V_\bro$ with coefficients in objects of $\hat\cC_{W(\F)}$, the category of local Noetherian $W(\bF)$-algebras $(A, \m_A)$ with residue fields $A/\m_A \cong \bF$. Then the \emph{deformation functor} of $\bro$ is
\[
\mathrm{Def}_\bro: \hat\cC_{W(\F)} \to \mathrm{Sets}; ~ A \mapsto \{(V_A, \rho_A) \text{ with an isomorphism } V_A \otimes_A \F \risom V_\bro\}/ \sim,
\]
where $(V_A,\rho_A)$ is a free $A$-module with an $A$-linear $G_{\bQ,S}$-action, and the equivalence relation $(V_A,\rho_A) \sim (V'_A,\rho'_A)$ is an isomorphism of $A[G_{\bQ,S}]$-modules preserving the isomorphism of $\F[G_{\bQ,S}]$-modules to $V_\bro$. An element of $\mathrm{Def}_\bro(A)$ is called a \emph{deformation} of $\bro$ to $A$. 

For applications to arithmetic, one often considers subfunctors of $\mathrm{Def}_\bro$ where one insists that the deformations $\rho_A$ satisfy certain additional local conditions. For example, one may consider 
\[
\mathrm{Def}_\bro^\ord(A)= \{\rho_A \in \mathrm{Def}_\bro(A) ~  | ~ \rho_A \text{ is ordinary}\}.
\]
Recall that we call $\rho_A$ ordinary if $\rho_A|_{G_{\Q_p}}$ has a $A$-rank $1$ quotient representation $\eta$ such that the Tate twist $\eta \otimes_{\Z_p} \Z_p(1)$ is unramified.

\subsubsection{The residually irreducible case} 

Deformation theory is best understood when $\bro$ is irreducible. In this case, Mazur \cite{mazur1989} has proven that $\mathrm{Def}_\bro$ is represented by a ring $R_\bro$, and $\mathrm{Def}_\bro^\ord$ is represented by a quotient $R^\ord_\bro$ of $R_\bro$.

If, in addition, $\bro$ is a residual representation of a modular representation, then there is a component of the Hida Hecke algebra $\h'$ denoted $\bT_\bro$ (in the deformation theorists' notation) such that
\[
\Spec \mathbb{T}_\bro \cong \{\text{Ordinary modular forms } f ~ | ~ \rho_f \otimes \F \simeq \bro\}.
\]
Under mild conditions, there exists a rank $2$ representation $\rho_{\bT_\bro}: G_\Q \to \GL_2(\bT_{\bro})$ coming from Hida theory with $\rho_{\bT_\bro} \in \mathrm{Def}_\bro^\ord(\bT_\bro)$. It interpolates the Galois representations on \'etale cohomology of modular curves over the limit on level as in \eqref{eq:H_limit}. This representation reflects the fact that there is a Galois representations $\rho_f$ attached to an ordinary modular form $f$. 

The fact that $R^\ord_\bro$ represents $\mathrm{Def}_\bro^\ord$ then implies that there is a ring homomorphism $\varphi: R_\bro^\ord \to \bT_\bro.$
The corresponding map on spectra is
\[
\varphi^*: \Spec \mathbb{T}_\bro  \lra \Spec R_\bro^\ord \cong \{\text{Ordinary deformations of $\bro$}\}; \quad  f \mapsto \rho_f
\]
The fact that $R_\bro^\ord \onto \bT_\bro$ is surjective reflects the fact that $f \mapsto \rho_f$ may be thought of as injective. If $R_\bro^\ord \risom \bT_\bro$, the isomorphism reflects that every ordinary deformation of $\bro$ is modular. Hence the modularity of deformations of $\bro$ is reduced to the injectivity of $R_\bro^\ord \onto \bT_\bro$, a problem that can be attacked using methods of commutative algebra.

\subsubsection{Tangent spaces and Wiles' numerical criterion}

To overview the numerical criterion accurately, we now replace $R_\bro^\ord$ and $\bT_\bro$ with their restriction to a fixed level (in the limit on levels as in \eqref{eq:MF_limits}) and weight $2$. The rings $R_\bro^\ord$, $\bT_\bro$ are both complete local $W(\F)$-algebras, and $\bT_\bro$ is a finite $W(\F)$-module. 

To analyze the map $\varphi: R_\bro^\ord \onto \bT_\bro$, one fixes a modular form $f$ with coefficients in $W(\F)$ such that $\rho_f$ deforms $\bro$. This induces a map $\pi : \bT_\bro \rsurj W(\F)$, and the map $\pi\circ \varphi : R_\bro^\ord \ra W(\F)$ corresponds to $\rho_f : G_\bQ \ra \GL_2(W(\F))$. The one studies the ideals $\wp_R := \ker (\pi \circ \varphi) \subset R_\bro^\ord$ and $\wp_\bT := \ker (\pi) \subset \bT_\bro$ and the maximal ideals $\m_R = (\wp_R,p) \subset R^\ord_\bro$ and $\m_\bT = (\wp_\bT, p) \subset \bT_\bro$. 
An initial deformation-theoretic interpretation of these ideals is that there is a canonical identification
\[
\Hom_{\F}(\m_R/(\m_R^2,p), \F) \cong \mathrm{Def}_\bro^\ord(\F[\epsilon]/(\epsilon)^2)
\]
An element $\rho \in \mathrm{Def}_\bro^\ord(\F[\epsilon]/(\epsilon)^2)$ can be written in the form $\rho=\bro+\epsilon \phi$ for a map $\phi:G_\Q \to \End_\F(V_\bro)$. The condition that $\rho$ be a homomorphism implies that $\phi$ is a cocycle, and the isomorphism class of $\rho$ is identified with the cohomology class of $\phi$. Moreover, the condition that $\rho$ is ordinary translates to a Selmer condition on $\phi$. In this way, one sees that $\m_R/(\m_R^2,p)$ is dual to a Selmer-type Galois cohomology group. 

This implies that $R_\bro^\ord$ has a presentation of the form
\[
R_\bro^\ord \simeq W(\F) \lb x_1, \dots, x_n \rb / (f_1, \dots , f_m)
\]
where $n$ is the dimension of a certain Selmer-type Galois cohomology group. One can show that $m$ is given by the dimension of a Galois $H^2$, and using global duality, one sees that, in fact, $n=m$. In other words, $R_\bro^\ord$ looks like a finite, complete intersection over $W(\bF)$, except one doesn't know that $(f_1, \dots , f_n)$ is a regular sequence, or, equivalently, that $R_\bro^\ord$ is finite over $W(\bF)$.

On the other hand, one does know a priori that $\bT_\bro$ is finite over $W(\F)$. Consequently, if $R_\bro^\ord \onto \bT_\bro$ is an isomorphism, that forces $R^\ord_\bro$ to be finite, and hence complete intersection. It is then natural to try to prove that $R_\bro^\ord$ and $\bT_\bro$ are complete intersection at the same time as proving that $R_\bro^\ord \onto \bT_\bro$ is an isomorphism. Along these lines, Wiles developed an ingenious ``numerical criterion'' for $R_\bro^\ord \onto \bT_\bro$ to be an isomorphism by comparing $\wp_R/\wp_R^2$ and $\wp_\bT/\wp_\bT^2$ \cite[Appendix]{wiles1995}. In this paper, we use a strengthening of his criterion due to Lenstra \cite{lenstra1993}. (In fact, to apply the criterion in Wiles's case, an additional tool, known as the Taylor-Wiles method, is needed. Since it is not used in this paper, we do not discuss this further (see Remark \ref{rem:no taylor-wiles}).) We remark that in the notation of \S\ref{subsec:intersect}, the numerical criterion is applied where ``$\cJ_\m$'' plays the role of $\wp_R$ and an Eisenstein ideal ``$\I_\p$'' plays the role of $\wp_\bT$. 

\subsubsection{Difficulties in the residually reducible case} 

We now turn to the case where $\bro$ is reducible, and let $\bar\eta_1$ and $\bar \eta_2$ denote the characters appearing in the Jordan-H\"older series $0 \ra \bar\eta_2 \ra \bro \ra \bar\eta_1 \ra 0$ of $\bro$. In this case, there arise problems in the approach outlined above. These problems are not just technical issues, but, as the following points illustrate, fundamental flaws that cause the whole strategy to break down. 
\begin{enumerate}[leftmargin=2em]
\item
\label{item: not representable}
If $\bro$ is reducible and semi-simple, then the deformation problem $\mathrm{Def}_\bro$ is not representable by a ring, that is, there exist no rings $R_\bro$ and $R_\bro^\ord$ with the appropriate properties. 
\item
\label{item: well-def up to ss}
 If $\bro$ is reducible but not semi-simple (and hence indecomposable), then Mazur's theory still applies such that the rings $R_\bro$ and $R_\bro^{\ord}$ exist. However, it is not clear that these are natural rings to work with because the residual representation is only well-defined (independently of the choice of lattice) up to semi-simplification. 
\item 
\label{item: not free}
 On the modular side, the Hecke algebra $\bT_\bro$ is replaced by the Eisenstein Hecke algebra $\h$ of \S \ref{sec: modular}. The natural Hecke module $H$ that parameterizes modular forms with residual Galois representation $\bro^{ss}$ is not free of rank $2$ over $\h$ in general. In fact, as described in \S \ref{sec: modular}, $H$ is free if and only if the Hecke algebra $\h$ is Gorenstein. Sharifi's conjecture predicts that $\h$ is not Gorenstein exactly when $X_\chi(1)$ is non-cyclic (c.f. ~ \S \ref{subsec: sharifi's conj}). In fact, it is known in some cases that $\h$ is not Gorenstein when $X_\chi(1)$ is non-cyclic \cite{kurihara1993}.
\item
\label{item: other extensions}
 Even if a lattice $H'$ that is free over $\h$ can be found, there may be modular forms $f$ such that $\rho_f \otimes_\h \F$ has the same semi-simplification as $\bro$, but such that $\rho_f \otimes_\h \F \not \simeq \bro$, reflecting issue \eqref{item: well-def up to ss}.
\end{enumerate}

Some of these difficulties can be worked around, as in the works of Skinner and Wiles \cite{SW1997}, \cite{SW1999}.  In \cite{SW1997}, the authors show that $R_\bro^\ord = \bT_\bro$ when $\Ext_{G_\Q}^1(\bar\eta_1,\bar \eta_2)$ is assumed to be spanned by the class of $\bro$. In the language of this paper, this assumption equates to $X_\chii = 0$, which is equivalent to $\X_\theta = 0$. 

In \cite{SW1999}, they do not assume that certain class groups are small, but instead they consider many possible $\bro$ with $\bro^{ss} \simeq \bar\eta_1 \oplus \bar\eta_2$. They prove a very general modularity result, but do not prove that any $R_\bro^\ord$ is isomorphic to $\h$. Indeed, items  \eqref{item: not representable} and \eqref{item: other extensions} suggest that there is no such isomorphism (see also \cite[\S1]{SW1997}). 

In this paper, we are interested in Sharifi's Conjecture, which deals with the natural lattice $H$; also, we investigate the Gorenstein property of the Hecke algebras $\h$ and $\fH$. For these purposes, we cannot work around these difficulties and must address them head on. Indeed, Sharifi's Conjecture suggests that the failure of $H$ to be free has great arithmetic significance, so we do not prefer a lattice $H'$ as in item \eqref{item: other extensions}. By item \eqref{item: not free}, we know that $H$ is free if and only if $\h$ is Gorenstein, so, if we want to prove Gorensteinness, we cannot assume $H$ is free. Finally, we feel that, in some sense, Sharifi's Conjecture is stating that the object $H$ is the universal object for some kind of Galois deformation problem with $\bro$ semi-simple (c.f. \S \ref{subsec: sharifi's conj}). To make sense of this, we must deal with item \eqref{item: not representable}.

\subsubsection{Pseudorepresentations} Sometimes an $A$-linear representation $\rho$ will have the property that the characteristic polynomial $\chi(\rho):G \to A[t]$ factors through $B[t] \subset A[t]$ for a proper subring $B \subset A$, even if $\rho$ itself does not have such a factorization. For some applications, a $\rho$ with this property may be used as a replacement for a truly $B$-valued representation. Wiles \cite{wiles1988} invented the notion of \emph{pseudorepresentations} (for 2-dimensional representations) to systematically exploit this property. The idea is that a pseudorepresentation is a collection of central functions that satisfy the same algebraic properties that the coefficients of the characteristic polynomial of a representation do. In particular, the coefficients of $\chi(\rho)$ as above, thought of as functions $G \to B$, will constitute a pseudorepresentation. In this way, one can see that the module $H$, defined in \S\ref{sec: modular}, gives a pseudorepresentation valued in $\h$, even if $H$ is not free over $\h$.

``Pseudorepresentation'' may seem, at first, to be an ad hoc notion. However, as their theory been developed and simplified (by Taylor \cite{taylor1991}, Bella\"iche-Chenevier \cite{BC2009}, Chenevier \cite{chen2014}, and others), it has become clear that pseudorepresentations are a central object in the theory of Galois representations and, in particular, deformation theory.  

There are two properties that explain why pseudorepresentations are crucial for understanding deformations of reducible representations. Let $\bro$ be as in the previous section, and let  $\Db: G_\Q \to \F$ be the associated pseudorepresentation. The first property is that the pseudodeformation functor $\mathrm{PsDef}_\Db$ (defined analogously to $\mathrm{Def}_\bro$ above), is always representable by a ring $R_\Db$ (see Theorem \ref{thm:chen-PsR}, in contrast to item \eqref{item: not representable} above). The second property has to to with the natural transformation $ \psi: \mathrm{Def}_\bro \to \mathrm{PsDef}_\Db$ defined by ``take the associated pseudorepresentation.'' The functor $\mathrm{Def}_\bro$ is a local ring in an ambient stack of all Galois representations with residual pseudorepresentation $\Db$, and $\psi$ identifies $\mathrm{PsDef}_\Db$ as the coarse moduli scheme of this stack \cite[Thm.\ A]{WE2017}. 

Roughly, these properties are saying that $\mathrm{PsDef}_\Db$ is the scheme that is ``closest to representing" the stacky moduli of all Galois representations with residual pseudorepresentation $\Db$. The $R=\bT$ approach relates the Hecke ring $\bT$ to moduli of Galois representations, and if this moduli problem is not represented by a ring, it is natural to instead consider the ``closest" ring, i.e.\ $R_\Db$. Another reason that it is natural to expect that $\bT$ can be related to a pseudodeformation ring is that automorphic representations are often characterized by their Satake parameters at unramified places, and the Satake data corresponds to the characteristic polynomial of the Frobenius action on the associated Galois representation. 

\subsubsection{Our strategy}
Using the fact that the module $H$ defined in \S \ref{sec: modular} gives rise to a pseudorepresentation, we obtain a map $R_\Db \to \fH$, where $R_\Db$ is the pseudodeformation ring of the residual pseudorepresentation $\Db$ (more precisely, we produce a map $R_\Db \to \h$ using $H$ and then extend it to $R_\Db \to \fH$). We can show the map is surjective using the same method as above, but it is far from injective. Indeed, in order for it to have any hope of being injective, we have to first impose the ordinary condition. The main goal of this section (indeed, the main new ingredient in this paper) is to develop a adequate notion of ``ordinary pseudorepresentation,'' and show that it is represented by a quotient $R_\Db^\ord$ of $R_\Db$. For an introduction to this problem, see \S \ref{subsec:OP}.

\subsection{Pseudorepresentations}
\label{subsec:psrep}
We introduce the notion of pseudorepresentation that we use, and recall some of the basic properties. We follow \cite{chen2014} and \cite{WE2017}.

\begin{defn}[{\cite{roby1, roby2}, \cite[\S1.5]{chen2014}}]
\label{defn:psrep}
Let $A$ be a commutative ring and let $E,S$ be associative $A$-algebras. A \emph{polynomial law} over $A$ from $E$ to $S$, written as $D : E \ra S$, is the data of a function
\begin{equation}
\label{eq:psrep-map}
D_B: E \otimes_A B \lra S \otimes_A B
\end{equation}
for every commutative $A$-algebra $B$, such that $D_B$ is functorial in $A$-algebras. 

We say $D$ is \emph{multiplicative} if, for every commutative $A$-algebra $B$,  $D_B(1)=1$ and $D_B(xy)=D_B(x)D_B(y)$ for all $x,y \in E \otimes_A B$.

We say $D$ has  \emph{degree} $d$ if, for every commutative $A$-algebra $B$, $D_B$ is homogeneous of degree $d$ in $B$, i.e.
\[
\forall \ b \in B, \forall\ x \in E \otimes_A B, \quad D_B(bx) = b^d D_B(x).
\]
\end{defn} 

\begin{defn}
A \emph{pseudorepresentation of dimension $d$} of $E$ over $A$ a degree $d$ multiplicative polynomial law from $E$ to $A$.
\end{defn}

\begin{exmps}
\label{exmps: pseudo basic exmps}
These important examples illustrate the idea of pseudorepresentations. We also use them to introduce some notation and basic facts.
\begin{enumerate}[leftmargin=2em]
\item The fundamental example is that, for any ring homomorphism $\rho: E \to M_d(A)$, the functions $E \otimes_A B \to B$ given by $\det \circ (\rho \otimes_A B)$ constitute a pseudorepresentation. We denote this pseudorepresentation by $\psi(\rho)$ and call it the \emph{associated} or \emph{induced} pseudorepresentation of $\rho$. This example explains the notation ($D$ is for ``determinant"). 
\item This is a variant of the first example. A representation $\rho: G \to \GL_d(A)$ of a group $G$ gives rise to a ring homomorphism $A[G] \to M_d(A)$ and we also call the associated pseudorepresentation $\psi(\rho)$. In the case of a group algebra $E=A[G]$ we denote a pseudorepresentation $D: A[G] \to A$ by $D:G \to A$.
\item If $D: E \to S$ is a polynomial law over $A$, and $E' \ra E$ is an $A$-algebra homomorphism, then there is a polynomial law $D':E' \to S$ over $A$, where $D_B'$ is the composite of $E'\otimes_A B \to E \otimes_A B$ with $D_B$. If $D$ is multiplicative or degree $d$, then so is $D'$.
\item If the map $A \to S$ is surjective, then there is only one degree $d$ multiplicative polynomial law $D: A \to S$ over $A$, namely the $d$-power map $D_B(x)=x^d$.
\item If $B$ is a commutative $A$-algebra, a degree $d$ multiplicative polynomial law $D: E \to B$ over $A$ induces a canonical $d$ dimensional pseudorepresentation of $E\otimes_A B$ over $B$ (see \cite[Rem.\ 1.4]{chen2014}). We denote this pseudorepresentation by $D \otimes_A B: E \otimes_A B \to B$. 
\item If the structure homomorphisms $A \to E$ and $A \to S$ both factor through a commutative ring homomorphism $A \to A'$ making $E$ and $S$ into $A'$-algebras, then a polynomial law $D: E \to S$ over $A$ induces a canonical polynomial law $E \to S$ over $A'$ using the same functions $D_B$. We will also denote this polynomial law over $A'$ by $D:E \to S$.
\end{enumerate}
\end{exmps}

We define $\PsR^d_E(B)$ to be the set of $d$-dimensional pseudorepresentations $D: E \otimes_A B \ra B$. It is evident that $\PsR^d_E$ is a functor on $A$-algebras. 

When $E=A[G]$, the $A$-algebra structure map $A \to E$ is injective. In fact, this is true whenever there is a pseudorepresentation $D: E \to A$, as the next lemma shows.
\begin{lem}
\label{lem:str-inj}
Given a $d$-dimensional pseudorepresentation $D : E \ra A$, the $A$-algebra structure map $A \ra E$ is necessarily injective.
\end{lem}
\begin{proof}
Let $I=\ker(A \to E)$ and identify the image of $A$ in $E$ with $A/I$. The restriction of $D$ to $A/I$ is a $d$-dimensional pseudorepresentation $D': A/I \ra A$. By \cite[Prop.\ 1.6]{chen2014}, this corresponds to an $A$-algebra map $\Gamma^d_A(A/I) \ra A$. Here, when $S$ is an $A$-algebra, $\Gamma^d_A(S)$ is the $A$-algebra defined in \cite{roby2} (cf.\ the discussion in \cite[\S1.1]{chen2014}). As an $A$-module, it is the $d$th graded piece of the free divided power algebra on the $A$-module $S$, and its multiplication law is determined in \cite{roby2}. One can calculate that $\Gamma^d_A(A/I)$ is isomorphic as an $A$-algebra to $A/I$, cf.\ \cite[Ex.\ 2.5]{chen2014}. Then $D'$ induces an $A$-algebra homomorphism $A/I \to A$, which implies $I = 0$ as desired. 
\end{proof}

\subsection{Characteristic polynomials and kernels}
Let $D: E \to S$ of degree $d$ be a multiplicative polynomial law of degree $d$ over $A$. The \emph{characteristic polynomial} $\chi^D(r,t) \in S[t]$ of $D$ is given by $\chi^D(r,t)=D_{A[t]}(t-r)$, 
and is written
\begin{equation}
\label{eq:chi_poly}
\chi^D(r,t) = t^d - \Lambda_{1}^D(r)t^{d-1} + \dotsm + (-1)^d \Lambda_d^D(r) = t^d + \sum_{i=1}^d (-1)^i\Lambda_i^D(r) t^{d-i}.
\end{equation}
Then each $\Lambda_i^D$ is a polynomial law of degree $i$. We note that $\Lambda_d^D=D$. We call $\Lambda_1^D: E \ra S$ the \emph{trace} of $D$, denote it by $\Tr_D$, and note that it is $A$-linear (see \cite[Example 1.2 (i)]{chen2014}).

The polynomials $\chi^D(r,t)$ for $r \in E$ uniquely characterize $D$ \cite[Lem.\ 1.12(ii)]{chen2014}. Consequently, a pseudorepresentation $D: E \ra A$ may be thought of as an ensemble of polynomials with coefficients in $A$, one for each element of $E$, satisfying compatibility properties as if they arose as characteristic polynomials of a representation of $E$. In fact, the original definition of pseudorepresentation by Wiles \cite{wiles1988} and the refined version of Taylor \cite{taylor1991} are closer to this formulation. 

It is well-known that the characteristic polynomials of a field-valued representation remember precisely the Jordan-H\"older factors of the representation. The converse is true in the following way.

\begin{thm}[{\cite[Thm.\ A]{chen2014}, \cite[Cor.\ 2.1.10]{WE2017}}]
\label{thm:PsR_ss}
Let $k$ be a perfect field and let $E$ be an associative $k$-algebra. Given a $d$-dimensional pseudorepresentation $D: E \ra k$, there exists a semi-simple representation $\rho^{ss}_D: E \otimes_k k' \ra M_d(k')$, unique up to isomorphism, such that $\psi(\rho^{ss}_D) = D \otimes_k k'$ and $k'/k$ is an extension of degree at most $d$. In particular, when $k$ is algebraically closed, we have a natural bijective correspondence between isomorphism classes of $d$-dimensional semi-simple representations with coefficients in $k$ and $d$-dimensional pseudorepresentations valued in $k$.
\end{thm}

Following Chenevier, we define the \emph{kernel} $\ker(D)$ of a multiplicative polynomial law $D: E \to S$ as follows: 
\begin{equation}
\label{eq:ker_defn}
\ker(D) = \{ x \in E \mid \forall B, \Lambda_{i,B}^D(xy)=0 \text{ for all }y \in E \otimes_A B, i \ge 1\}
\end{equation}
where $B$ varies over all commutative $A$-algebras. The name kernel is justified by the following lemma.
\begin{lem}[{\cite[Lem.\ 1.19]{chen2014}}]
\label{lem: ker of pseudo}
Let $D:E \to S$ be a multiplicative polynomial law of degree $d$. Then $\ker(D)$ is a two-sided ideal of $E$ that is maximal among two-sided ideals $J$ of $E$ such that there is a multiplicative polynomial law $\tilde{D}: E/J \to S$ of degree $d$ with $D = \tilde{D} \circ \pi$, where $\pi: E \to E/J$ is the quotient map.
\end{lem}

\subsection{Continuous pseudorepresentations and deformations}

Chenevier also discusses continuous pseudorepresentations of profinite groups $G$; Theorem \ref{thm:PsR_ss} also holds in this topological setting \cite[\S2.30]{chen2014}, giving us an understanding of field-valued pseudorepresentations. We will concern ourselves with continuous pseudorepresentations valued in coefficient rings in $\hat\cC_{W(\F)}$ and $\cC_F$, which are defined as follows. Given a finite field $\F$, we let $\hat\cC_{W(\F)}$ be the category of complete local Noetherian $W(\bF)$-algebras $(A, \m_A)$ with residue field $A/\m_A \cong \bF$. For a finite extension $F/\bQ_p$, we let $\cC_F$ denote the category of local Artinian $F$-algebras with residue field $F$. A pseudorepresentation $D: G \ra A$ is \emph{continuous} when each of the characteristic polynomial coefficients $\Lambda^D_{i,A} : A[G] \ra A$ is continuous, cf.\ \cite[\S2.30]{chen2014}. The term ``pseudorepresentation'' will be used to refer to continuous pseudorepresentations without further comment. 

We now discuss deformation theory, in analogy with the case of representations discussed in \S\ref{subsub:def}.

Fix a pseudorepresentation $\Db: G \ra \bF$ valued in a finite field $\bF$. Its deformation functor $\mathrm{PsDef}_{\Db} : \hat\cC_{W(\F)} \ra \mathrm{Sets}$ is 
\begin{equation}
\label{eq:PsR_functor}
A \mapsto \{\text{continuous } D_A : A[G] \ra A \text{ such that } D_A \otimes_A \bF \simeq \Db\}.
\end{equation}

Likewise, given a finite extension $F/\bQ_p$ and a continuous pseudorepresentation $D: G \ra F$, we define the deformation functor $\mathrm{PsDef}_D : \cC_F \ra \mathrm{Sets}$ just as in \eqref{eq:PsR_functor}.

\begin{defn}
Elements $D_A$ of the set $\mathrm{PsDef}_{\Db}(A)$ are called \emph{pseudodeformations} of $\Db$, $\Db$ is called the \emph{residual pseudorepresentation} of $D_A$, and $\mathrm{PsDef}_{\Db}$ is the \emph{pseudodeformation functor}. 
\end{defn}

We say that $G$ satisfies the \emph{$\Phi_p$-finiteness condition} (see \cite[\S1.1]{mazur1989}) if the maximal pro-$p$ quotient of any finite index subgroup of $G$ is topologically finitely generated. For example the groups $G = G_{\bQ,S}$ and $G = G_{\bQ_p}$ satisfy this condition. 

\begin{thm}[{\cite[Cor.\ 3.14 and \S4.1]{chen2014}}]
\label{thm:chen-PsR} 
Assume that $G$ satisfies the $\Phi_p$-finiteness condition. 
\begin{enumerate}[leftmargin=2em]
\item Let $\Db : G \ra \bF$ be a pseudorepresentation. Then the functor $\mathrm{PsDef}_\Db$ is representable by a complete Noetherian local $W(\bF)$-algebra $(R_\Db, \mDb)$. 
\item Let $F/\bQ_p$ be a finite extension and let $D : G \ra F$ be a continuous pseudorepresentation. Then the functor $\mathrm{PsDef}_D$ is pro-represented by a complete Noetherian local $F$-algebra $R_D$. 
\item There exists a factorization of $D$ through the ring of integers $\cO_F \subset F$. Denote the residual representation of this factorization by $\Db$, so that there is a map $R_\Db \ra O_F$. Writing $\m \subset R_\Db[1/p]$ for the maximal ideal arising as the kernel of $R_\Db[1/p] \ra F$, $k(\m)$ for the residue field of $\m$, and $(-)^\wedge_\m$ for $\m$-adic completion, there is then a canonical isomorphism 
\begin{equation}
\label{eq:residual_to_0}
R_D \lrisom R_\Db[1/p]^\wedge_\m \otimes_{k(\m)} F. 
\end{equation}
\end{enumerate}
\end{thm}

The ring $R_\Db$ (resp.\ $R_D$) is known as the \emph{pseudodeformation ring} for $\Db$ (resp.\ $D$). The universal pseudodeformation of $\Db$ will be denoted $D^u_\Db: R_\Db [G] \ra R_\Db$. 

\begin{rem}
In light of Theorem \ref{thm:chen-PsR}(3), given a $A$-valued pseudorepresentation $D: G \ra A$ for $A \in \cC_F$ and $F/\bQ_p$ a finite extension, there is a finite field valued pseudorepresentation $\Db : G \ra \F$ that deserves to be called the residual pseudorepresentation of $D$. Namely, one first reduces $D$ modulo the maximal ideal of $A$, producing an $F$-valued pseudorepresentation. Then there is an associated residual pseudorepresentation by part (3). Notice that the residual field of $F$ may not be precisely the field of definition $\F$ of $\Db$. One only knows that $F$ is a finite extension of $W(\F)[1/p]$. 
\end{rem}

\subsection{The universal Cayley-Hamilton algebra} 

We will now introduce Cayley-Hamilton algebras and introduce background on representations of a profinite group valued in Cayley-Hamilton algebras, called Cayley-Hamilton representations. We are motivated by the desire to study the $G_{\bQ,S}$-action on $H$ of \S\ref{sec: modular}, which is a Cayley-Hamilton representation even though $H$ is not necessarily a free $\h$-module. 

\begin{defn}[{\cite[\S1.17]{chen2014}}]
\label{defn:CH_alg_rep}
Let $E$ be an $A$-algebra and let $D: E \ra A$ be a $d$-dimensional pseudorepresentation. 
\begin{itemize}[leftmargin=2em]
\item We call $D$ \emph{Cayley-Hamilton} provided that for every commutative $A$-algebra $B$ and every element $x \in E \otimes_A B$, the characteristic polynomial $\chi^D(x, t) \in B[t]$ satisfies $\chi^D(x,x) = 0$.
\item A \emph{Cayley-Hamilton $A$-algebra} of dimension $d$ is the pair $(E,D)$ consisting of an $A$-algebra and a $d$-dimensional Cayley-Hamilton pseudorepresentation. A morphism $f : (E,D) \ra (E',D')$ of Cayley-Hamilton $A$-algebras is an $A$-algebra map $f: E \ra E'$ such that $D' \circ f = D$. 
\item Let $B$ be a commutative $A$-algebra. A \emph{Cayley-Hamilton representation of $(E,D)$ over $B$} is a Cayley-Hamilton $B$-algebra $(E',D')$ and a morphism of Cayley-Hamilton $B$-algebras $f: (E \otimes_A B, D \otimes_A B) \ra (E',D')$. 
\end{itemize}
\end{defn}

For example, a matrix algebra $M_d(A)$ along with the determinant pseudorepresentation $\det: M_d(A) \ra A$ is Cayley-Hamilton, by the Cayley-Hamilton Theorem. A Cayley-Hamilton algebra has a characteristic polynomial, trace, etc.\ from \eqref{eq:chi_poly}, and the Cayley-Hamilton theorem holds. However, unlike matrix algebras, it need not be the case that $A$ is precisely the center of $E$.

We formulate Galois representations valued in a Cayley-Hamilton algebra as follows, where $F/\bQ_p$ is a finite extension. 

\begin{defn}[{\cite[\S1.22]{chen2014}}]
\label{defn:C-H_rep}
A \emph{Cayley-Hamilton representation of $G$ with residual pseudorepresentation $\Db$} is a triple $(A, (E,D), \rho)$ where $A \in \hat\cC_{W(\F)}$ or $A \in \cC_F$, $(E,D)$ is a Cayley-Hamilton $A$-algebra, and $\rho: G \ra E^\times$ is a homomorphism such that $D \circ \rho$ is a continuous $A$-valued pseudorepresentation with residual pseudorepresentation $\Db$.
\end{defn}

For example, a representation $(\rho,V)$ of $G$ on an $F$-vector space $V$ with residual pseudorepresentation $\Db$ amounts to the data of a Cayley-Hamilton representation $(F, (\End_F(V), \det), \rho)$. 

For a profinite group $G$ and a residual pseudorepresentation $\Db: G \to \F$, there exists a \emph{universal Cayley-Hamilton representation} in the context of deformations of $\Db$, in the following sense.

\begin{prop}[{\cite[Prop.\ 1.23]{chen2014}, \cite[Prop.\ 3.2.2, Thm.\ 3.2.3]{WE2017}}]
\label{prop:univ_C-H}
Let $G$ be a profinite group satisfying the $\Phi_p$-finiteness condition, and let $\Db: G  \ra \bF$ be a pseudorepresentation. Then there exists a maximal quotient algebra $E(G)_\Db$ of $R_\Db[G]$ with the properties that $D^u_\Db$ factors through $E(G)_\Db$ and $(E(G)_\Db, D^u_\Db)$ is a Cayley-Hamilton algebra. Moreover,
\begin{enumerate}[leftmargin=2em]
\item There exists a universal Cayley-Hamilton representation of $G$ with residual pseudorepresentation $\Db$, namely 
\[
(E_\Db, (E(G)_\Db, D^u_\Db), \rho^{u}: G \ra E(G)_\Db^\times).
\]
In other words, for every Cayley-Hamilton representation $(A, (E,D), \rho)$ of $G$ with residual pseudorepresentation $\Db$ (where $A \in \hat\cC_{W(\F)}$ or $A \in \cC_F$), there exists a unique map $g: R_\Db \ra A$ and a unique morphism of Cayley-Hamilton $A$-algebras 
\[
f: (E(G)_\Db \otimes_{R_\Db, g} A, D^u_\Db \otimes_{R_\Db, g} A) \lra (E, D)
\]
such that $\rho = f \circ \rho^u$. 
\item $E(G)_\Db$ is finite as a $R_\Db$-module.
\item $E(G)_\Db$ is a continuous quotient of $R_\Db[G]$, where its topology is the pullback under the natural map $R_\Db[G] \ra R_\Db\lb G\rb$ of the standard profinite topology on $R_\Db\lb G\rb$ of a completed group algebra over a profinite ring. 
\end{enumerate}
(By abuse of notation, we have written $D^u_\Db$ for the factorization of $D^u_\Db : R_\Db [G] \ra R_\Db$ through $E(G)_\Db$.)
\end{prop}
We will consistently use the following perspective, which is justified by part (1) of the proposition: Cayley-Hamilton representations of $G$ with residual pseudorepresentation deforming $\Db$ are equivalent to Cayley-Hamilton representations of $E(G)_\Db$ (as defined in Definition \ref{defn:CH_alg_rep}). 

\subsection{Generalized matrix algebras} 
\label{subsec:GMA}

In the applications we will pursue, the Galois actions on modular forms are Cayley-Hamilton representations such that the induced residual pseudorepresentation $\Db$ is multiplicity-free. We call $\Db$ \emph{multiplicity-free} when the semi-simple representation $\rho^{ss}_\Db \otimes_\F \overline{\bF}$ associated to $\Db$ by Theorem \ref{thm:PsR_ss} has no multiplicity among its Jordan-H\"older factors. When this condition holds, we see in \cite[Cor.\ 2.9(2), Def.\ 3.4]{WE2017} that there is a bijective correspondence between irreducible factors of $\rho^{ss}_\Db$ and $\rho^{ss}_\Db \otimes_\F \overline{\F}$. This gives us access to an additional structure on Cayley-Hamilton representations with residual pseudorepresentation $\Db$, namely that of a \emph{generalized matrix algebra}. 

\begin{defn}[{\cite[\S1.3]{BC2009}}]
\label{defn:GMA}
Let $A$ be a commutative ring and let $E$ be an $A$-algebra. A \emph{generalized matrix $A$-algebra} or $A$-\emph{GMA} structure of type $(d_1, \dotsc, d_r)$ on $E$ is the data $\cE$ of 
\begin{enumerate}[leftmargin=2em]
\item a set of $r$ orthogonal idempotents $e_1, \dotsc, e_r$ with sum $1$, and
\item a set of $r$ isomorphisms of $A$-algebras $\phi_i : e_i E e_i \risom M_{d_i}(A)$,
\end{enumerate}
such that the trace map $\Tr_\cE : E \ra A$ defined by
\[
\Tr_\cE(x) := \sum_{i=1}^r \Tr \phi_i(e_ixe_i)
\]
is a central function, i.e.~$\Tr_\cE(xy) = \Tr_\cE(yx)$ for all $x,y \in E$. 

We call $\cE= (\{e_i\}, \{\phi_i\})$ the \emph{data of idempotents} of $E$ and write $(E, \cE)$ for a GMA.
\end{defn}

It follows from the definition that a generalized matrix algebra has a direct sum decomposition as an $A$-module of the form
\begin{equation}
\label{eq:GMA_form}
E \lrisom \begin{pmatrix}
M_{d_1}(\cA_{1,1}) & M_{d_1 \times d_2}(\cA_{1,2}) & \dotsm & M_{d_1 \times d_r}(\cA_{1,r}) \\
M_{d_2 \times d_1}(\cA_{2,1}) & M_{d_2}(\cA_{2,2}) & \dotsm & M_{d_2 \times d_r}(\cA_{2,r}) \\
\vdots & \vdots & \vdots & \vdots \\
M_{d_r \times d_1}(\cA_{r,1}) & M_{d_r \times d_2}(\cA_{r,2}) & \dotsm & M_{d_r}(\cA_{r,r}) \\
\end{pmatrix},
\end{equation}
equipped with $A$-module morphisms $\varphi_{i,j,k}: \cA_{i,j} \otimes_A \cA_{j,k} \ra \cA_{i,k}$ satisfying properties of \cite[\S1.3.2]{BC2009} so that the ensemble $\{\varphi_{i,j,k}\}$ induces the multiplication map $E \otimes_A E \ra E$. One of these properties is that $\cA_{i,i} \cong A$ as an $A$-algebra. 

For an $A$-GMA $(E,\cE)$, there is a canonical pseudorepresentation $D_\cE: E \ra A$ characterized as follows. The characteristic polynomial $\chi^{D_\cE}$ is computed using \eqref{eq:GMA_form} just as one computes the characteristic polynomial of a matrix (using the maps $\{\varphi_{i,j,k}\}$ to multiply entries) \cite[Prop.\ 2.4.5]{WE2017}. This canonical pseudorepresentation $D_\cE$ is Cayley-Hamilton, making $(E, D_\cE)$ a Cayley-Hamilton algebra. In certain cases, there is a converse to this statement: 

\begin{thm}[{\cite[Thm.\ 2.22(ii)]{chen2014}, \cite[Thm.\ 2.4.10]{WE2017}}]
\label{thm:CH_is_GMA}
Given a Cayley-Hamilton pseudorepresentation $D: E \ra A$ where $A$ is a Henselian Noetherian local ring with residue field $F$, $E$ is finitely generated over $A$, and the residual pseudorepresentation $\Db := D \otimes_A F$ is multiplicity-free, $E$ admits the structure of a generalized matrix algebra $(E,\cE)$ such that $D = D_\cE$. In particular, $\Tr_\cE = \Tr_D$. 
\end{thm}

We sketch the proof in order to introduce Definition \ref{defn:LoI}.
\begin{proof}
By \cite[Lem.\ 2.10]{chen2014}, the preimage of $\ker(\Db) \subset E \otimes_A F$ in $E$ is the Jacobson radical $\Jac(E)$. Since $\Db$ is multiplicity-free, $E/\Jac(E)$ is a product of matrix algebras $\prod_{i=1}^r M_{d_i}(F)$ where $\sum_{i=1}^r d_i = d$. In the proof of \cite[Thm.\ 2.22(b)]{chen2014}, Chenevier shows that idempotents lift over $\Jac(E)$, i.e.\ that $E$ is semiperfect, in the sense of \cite[Defn.~2.7.16, pg.~217]{rowen1988}. It follows that a complete set of primitive (orthogonal) idempotents of $E/\Jac(E)$ lifts to a complete set of primitive idempotents of $E$ and, moreover, certain sums $e_i$ over subsets of these idempotents have the property that $e_i E e_i \simeq M_{d_i}(A)$. The summands of $e_i$ determine coordinates on $e_i E e_i$, i.e.\ an isomorphism $\phi_i : e_i E e_i \risom M_{d_i}(A)$. This makes $\cE = (\{e_i\}, \{\phi_i\})$ a data of idempotents for $E$, making $(E, \cE)$ an $A$-GMA. In \cite[Thm.\ 2.4.10]{WE2017} it is verified that $D = D_\cE$. 
\end{proof}

\begin{defn}
	\label{defn:LoI}
	With assumptions as in Theorem \ref{thm:CH_is_GMA} and a choice of a complete set of primitive idempotents of $E/\ker(\Db)$, a complete set of primitive lifts of these idempotents will be called a \emph{lift of idempotents}. 
\end{defn}

When a Cayley-Hamilton representation is valued in a Cayley-Hamilton algebra equipped with the structure of a GMA, we call it a \emph{GMA representation.}

\begin{exmp}
\label{exmp:2x2GMA}
In this paper, we will only consider GMA structures of type $(1,1)$ (i.e.~ the case $r=2$ and $d_1=d_2=1$ of Definition \ref{defn:GMA}). 

\textit{Data of idempotents and lifts of idempotents.} In the case that $d_i = 1$ for all $i$, there is a single possible choice of the isomorphisms $\phi_i$ in a data of idempotents, so we refer to the idempotents $\{e_i\}$ as a data of idempotents $\cE$, dropping the $\phi_i$. Likewise, the $e_i$ are primitive, so a lift of idempotents is a data of idempotents in this case. In particular, note that the residual data of idempotents is canonical (up to ordering) because $E/\Jac(E) \simeq \prod_{i=1}^r F$. 

\textit{Coordinate decomposition.} 
Given a data of idempotents $\cE$ of type $(1,1)$, the decomposition of $E$ in \eqref{eq:GMA_form} has the simple form
\begin{equation}
\label{eq:E-decomp}
E \cong \begin{pmatrix} A & B \\ C & A \end{pmatrix},
\end{equation}
where we write $B$ for $\cA_{1,2}$ and $C$ for $\cA_{2,1}$. This is a convenient way to record several facts: $E$ admits a direct sum decomposition into the constituent $A$-modules $A \oplus B \oplus C \oplus A$, and the multiplication map $E \times E \ra E$ decomposes into coordinates in the conventional manner of matrix algebras. Namely, one has the multiplication map $A \times A \ra A$, the $A$-module structure maps $A \times B \ra B$ and $A \times C \ra C$, and a multiplication map $m = \varphi_{1,2,1} : B \otimes_A C \ra A$ to multiply off-diagonal entries into the $(1,1)$-coordinate. Here $m$ is a morphism of $A$-modules. The ``commutativity'' GMA axiom called ``(COM)'' in \cite[\S1.3.2]{BC2009} implies that the multiplication map $\varphi_{2,1,2} : C \otimes_A B \ra A$ to the $(2,2)$-coordinate is symmetric to $m$. Even when we have isomorphisms of $A$-modules $A \simeq B \simeq C$, it is still possible that $m = 0$! 
\end{exmp}

\begin{rem}
\label{rem:mult_dot}
In the sequel, when we have a decomposition of the form \eqref{eq:E-decomp}, we will use ``$\cdot$'' to denote the multiplication map $m: B \otimes_A C \ra A$. In other words, $b \cdot c := m(b \otimes c)$ for $b \in B$, $c \in C$. 
\end{rem}

We record this basic fact about lifts of idempotents in semiperfect algebras.
\begin{lem}
\label{lem:lifts_conj}
With assumptions as in Theorem \ref{thm:CH_is_GMA} and a choice of a complete set of primitive idempotents of $E/\ker(\Db)$, for any two lifts of idempotents $\{e_i\}$, $\{e'_i\}$ in $E$, there exists a unit $u \in E^\times$ such that $e_i'=u^{-1}e_iu$.
\end{lem}
\begin{proof}
This follows the Krull-Schmidt theorem -- see, for example, \cite[Thm.\ 2.9.18(iii), pg.~242]{rowen1988}. (As noted in the proof of Theorem \ref{thm:CH_is_GMA}, Chenevier has shown that $E$ is semiperfect in the sense of \emph{loc.~cit}.)
\end{proof}

\subsection{Reducibility and the reducibility ideal}
For a field-valued pseudorepresentation $D:G \to F$, the associated semi-simple representation $\rho_D^{ss}$ may or may not be reducible, and one could use the reducibility of $\rho_D^{ss}$ to define reducibility for $D$. We will need a notion of reducibility for more general coefficients, which has been developed in \cite[\S1.5]{BC2009}. For simplicity, we will restrict ourselves to the case of residually multiplicity-free pseudorepresentations with two factors, i.e.\ $r = 2$ in Definition \ref{defn:GMA}. 

\begin{defn}
Let $\Db: G \to \bF$ be a residual pseudorepresentation, and assume that $\rho^{ss}_\Db = \bar\rho_1 \oplus \bar\rho_2$ where $\bro_1, \bro_2$ are irreducible. For $A \in \hat\cC_{W(\F)}$ or $A \in \cC_F$, a pseudodeformation $D: G \ra A$ of $\Db$ is called \emph{reducible} if $D = \psi(\rho_1 \oplus \rho_2)$ for deformations $\rho_i$ of $\bar\rho_i$ to $A$, $i = 1,2$. Otherwise, call $D$ \emph{irreducible}.

Similarly, when $D : E \ra A$ is a pseudorepresentation and $A$ is a local ring with residue field $F$ such that $\Db = D \otimes_A F$ arises as $\psi(\bar\rho_1 \oplus \bar\rho_2)$ where $\bro_1, \bro_2$ are irreducible, $D$ is called \emph{reducible} if $D = \psi(\rho_1 \oplus \rho_2)$ for deformations $\rho_i$ of $\bro_i$ to $A$, $i = 1,2$. Otherwise, call $D$ \emph{irreducible}. 

Given a deformation $D : G \ra A$ (resp.\ $D : E \ra A$) of $\Db : G \ra \bF$ (resp.\ $D: E \otimes_A F \ra F$) to $A$, its associated \emph{reducibility ideal} is the minimal ideal $\cJ \subset A$ with the property that $D \otimes_A A/\cJ$ is reducible. 
\end{defn}

The reducibility ideal exists and, in the residually multiplicity-free case ($\bro_1 \not\simeq \bro_2$), can be expressed in terms of the GMA structure, which will be useful in the sequel. For simplicity we restrict to the case $d = r = 2$, i.e.\ $\cE$ is type $(1,1)$. 
\begin{prop}[{\cite[\S1.5.1]{BC2009}}]
\label{prop:red_ideal}
Let $A$, $E$ and $D: E \to A$ be as in Theorem \ref{thm:CH_is_GMA}. Choose a lift of idempotents $\cE$ in $E$ and assume that it is of type $(1,1)$. The image of $\cA_{1,2} \otimes_A \cA_{2,1}$ in $A$ under the multiplication map $\varphi_{1,2,1}$ (as in Example \ref{exmp:2x2GMA}) is the reducibility ideal $J_D$ of $D$. For any commutative ring map $f: A \ra B$, the induced pseudorepresentation $f \circ D: E \otimes_A B \ra B$ is reducible if and only if $f$ factors through $A/J_D$. 
\end{prop}

In the same way, we call a $2$-dimensional Cayley-Hamilton representation $f: (E,D) \ra (E',D')$ of $(E,D)$ \emph{reducible} when the pseudorepresentation $D' \circ f$ is reducible, and call it \emph{irreducible} otherwise. Likewise, we call a $2$-dimensional Cayley-Hamilton representation $(A, (E,D), \rho: G \ra E^\times)$  of a group $G$ with residual pseudorepresentation $\Db$ \emph{reducible} when the induced pseudorepresentation $D: G \ra A$ is reducible, and call it \emph{irreducible} otherwise. 

We will use the following lemma, which follows immediately from Proposition \ref{prop:univ_C-H}(1) and the content of this section, in what follows with no further comment. 

\begin{lem}
\label{lem:G-to-E}
Let $G$ be a profinite group satisfying the $\Phi_p$-finiteness condition and let $\Db$ be a residual pseudorepresentation such that $\rho^{ss}_\Db \simeq \bro_1 \oplus \bro_2$ where $\bro_i$ are irreducible for $i = 1,2$.  
\begin{enumerate}
\item A pseudodeformation $D: G \ra A$ of $\Db$ is reducible if and only if its factorization $E(G)_\Db \ra A$ is reducible. 
\item Likewise, a Cayley-Hamilton representation of $G$ $(A, (E,D), \rho: G \ra E^\times)$ (with residual pseudorepresentation $\Db$) is reducible if and only if its factorization $(E(G)_\Db, D^u_\Db) \ra (E,D)$ is reducible. The reducibility ideals of $D \circ \rho: G \ra A$ and $D : E \ra A$ are equal. 
\item Assume $\bro_1 \not\simeq \bro_2$ and $\dim \bro_1 = \dim \bro_2 = 1$. For any choice of lift of idempotents on $E(G)_\Db$, the reducibility ideal of the universal pseudodeformation $D^u_\Db : G \ra R_\Db$ of $\Db$ is given by Proposition \ref{prop:red_ideal} in the case $A = R_\Db$, $E = E(G)_\Db$, and $D = D^u_\Db: E(G)_\Db \ra R_\Db$. 
\end{enumerate}
\end{lem}

The following lemma explicitly constructs a lift of idempotents for certain two-dimensional reducible Cayley-Hamilton representations. For these representations, it also shows that the diagonal characters are independent of the choice of lift of idempotents.

\begin{lem}
\label{lem:distinguished gives lifts}
Let $G$ be a profinite group satisfying the $\Phi_p$-finiteness condition and let $\Db : G \ra \bF$ be a two-dimensional residual pseudorepresentation such that $\rho^{ss}_\Db \simeq \bro_1 \oplus \bro_2$,  where $\bro_1, \bro_2$ are distinct characters.

Let $(A,(E,D), \rho)$ be a reducible Cayley-Hamilton representation with residual pseudorepresentation $\Db$. Write $D=\psi(\rho_1 \oplus \rho_2)$ with $\rho_i$ deforming $\bro_i$.

\begin{enumerate}
\item Let $g\in G$ be such that $\bro_1(g) \ne \bro_2(g)$ in $\F^\times$. Then $(e_1,e_2)$ is a lift of idempotents in $E$, where
\[
e_1 = \frac{\rho(g)-\rho_2(g)}{\rho_1(g)-\rho_2(g)}, \quad e_2=1-e_1.
\]
\item Let $(e_1,e_2)$ be any lift of idempotents in $E$. Then, for $i=1,2$, we have $e_i \rho e_i \simeq \rho_i$.
\end{enumerate} 
\end{lem}
\begin{proof}
(1) The map $E \to E/\ker(\Db) \cong \bF \oplus \bF$ sends $\rho(\sigma)$ to $(\bro_1(\sigma),\bro_2(\sigma))$, so we see that $e_1 \mapsto (1,0)$ and $e_2 \mapsto (0,1)$. It remains to show that $e_1$ is an idempotent. Since $\rho$ is Cayley-Hamilton, $\rho(g)$ satisfies the characteristic polynomial $\chi_D(g,t) = t^2 - (\rho_1(g)+\rho_2(g))t + \rho_1(g)\rho_2(g)$. It is then a simple computation to see that $e_1^2=e_1$.

(2) Write $\rho_i':= e_i \rho e_i$ for $i=1,2$; by the definition of lift of idempotents, we have $\rho_i \otimes_A \F \cong \bro_i$. Since $\rho$ is reducible, we see that $\rho_1',\rho_2'$ are characters and that $D=\psi(\rho'_1 \oplus \rho'_2)$. Choose $g \in G$ as in (1), and note that, as in the proof of (1), we have $(\rho-\rho_1)(\rho-\rho_2)=0$. Evaluating this at $g$ and using the fact that $\rho$ is reducible, we obtain
\[
0=(\rho_1'(g)-\rho_1(g))(\rho_1'(g)-\rho_2(g))=(\rho_2'(g)-\rho_1(g))(\rho_2'(g)-\rho_2(g))
\]
(from the $(1,1)$- and $(2,2)$-entries, respectively).
Since $\rho_1'(g)-\rho_2(g)$ and $\rho_2'(g)-\rho_1(g)$ are units, this gives $\rho_1'(g)=\rho_1(g)$ and $\rho_2'(g)=\rho_2(g)$.

Now we continue as in the proof of linear independence of characters. For any $\sigma \in G$, consider the two equations
\[
0 = \rho_1(g) ( \rho_1(\sigma) + \rho_2(\sigma) - \rho'_1(\sigma) -\rho'_2(\sigma))
\]
and 
\[
0 = \rho_1(\sigma g) + \rho_2(\sigma g) - \rho'_1(\sigma g) - \rho'_2(\sigma g)
\]
coming from $D=\psi(\rho'_1 \oplus \rho'_2) = \psi(\rho_1 \oplus \rho_2)$. Taking their difference (and recalling that $\rho_1'(g)=\rho_1(g)$ and $\rho_2'(g)=\rho_2(g)$), we obtain
\[
0 = (\rho_1(g)-\rho_2(g))\rho_2(\sigma) -(\rho_1(g)-\rho_2(g))\rho_2'(\sigma).
\]
Since $\rho_1(g)-\rho_2(g)$ is a unit, this yields $\rho_2(\sigma) = \rho_2'(\sigma)$. The proof that $\rho_1=\rho_1'$ is similar.
\end{proof}

\subsection{Deformations of a modular residual pseudorepresentation}
\label{subsec:Dbar-setup}

For the remainder of the section, we fix $d = r = 2$, $G = G_{\bQ,S}$, and $\Db = \psi(\omega^{-1} \oplus \theta^{-1}) : G_{\bQ,S} \ra \bF$. (Here $\bF$ is the field defined in \S\ref{subsec:setup}.) Later, in Theorem \ref{thm:H_ord}(1), we will confirm that the $G_{\bQ,S}$-action on $H$ described in \S\ref{subsec: H prelims} induces an $\h$-valued pseudorepresentation whose residual pseudorepresentation is $\Db$. Therefore $\Db$ is a ``modular'' residual pseudorepresentation. 

Write $R_\Db$ for the pseudodeformation ring and write $E_\Db := E(G_{\bQ,S})_\Db$ for the associated universal Cayley-Hamilton algebra. Keep in mind that the $R_\Db$-algebra structure map $R_\Db \to E_\Db$ is injective by Lemma \ref{lem:str-inj}. 

We write $\Db : G_{\Q_p} \ra \bF$ for the restriction of $\Db$ to $G_{\Q_p}$, which is multiplicity-free by the working assumptions of \S\ref{subsec:setup}. We write $R^p_\Db$ for its pseudodeformation ring, $D^p_\Db: G_{\Q_p} \ra R^p_\Db$ for the universal pseudodeformation, and $E_\Db^p := E(G_{\Q_p})_\Db$ for its associated universal Cayley-Hamilton algebra. 

Theorem \ref{thm:CH_is_GMA} applies to the Cayley-Hamilton pseudorepresentations $D^u_\Db : E_\Db \ra R_\Db$ and $D^p_\Db : E^p_\Db \ra R^p_\Db$. That is, there exist various choices for lifts of idempotents of type $(1,1)$. We will use the resulting matrix coordinate decomposition \eqref{eq:E-decomp} of Example \ref{exmp:2x2GMA} to discuss the ordinary condition. In order to do this, we fix an ordering of the idempotents in $E/\ker(\Db) \cong \F \oplus \F$ so that the composite map $G_{\Q,S} \to E/\ker(\Db) \cong \F \oplus \F$ is given by $\sigma \mapsto (\omega^{-1}(\sigma),\theta^{-1}(\sigma))$. We fix the convention that a lift of idempotents $(e_1,e_2)$ in $E^p_\Db$ or $E^u_\Db$ will have $e_1$ associated with $\omega^{-1}$ and $e_2$ associated with $\theta^{-1}$. 
	
By Proposition \ref{prop:univ_C-H}, every Cayley-Hamilton representation of $G_{\bQ,S}$ with residual pseudorepresentation $\Db$ admits a GMA structure. 

\begin{exmp}
\label{exmp:classical_case}
Let $(V, \rho)$ be a representation of $G_{\bQ,S}$ on a $F$-vector space $V$ with induced residual pseudorepresentation $\Db$, where $F/\bQ_p$ is a finite extension. The universal property of $(E_\Db, R_\Db)$ induces a unique morphism of Cayley-Hamilton algebras $f : (E_\Db \otimes_{R_\Db} F, D^u_\Db \otimes_{R_\Db} F) \ra (\End_F(V), \det)$. Then $\End_F(V)$ admits a GMA structure over $F$ drawing an isomorphism with $M_2(F)$, where the idempotents defining its GMA structure are the image under $f$ of a choice of idempotents in $E_\Db$. 
\end{exmp}

\begin{rem}
\label{rem:no_extn}
Because $\Db : G \ra \bF$ is valued in a finite field, \cite[Cor.\ 2.1.10(2)]{WE2017} implies that the representation $\rho^{ss}_\Db$ of Theorem \ref{thm:PsR_ss} is defined over $\bF$. (It can be seen directly that $\rho^{ss}_\Db \simeq \omega^{-1} \oplus \theta^{-1}$.) Likewise, given a $p$-adic field valued pseudorepresentation $D: G_{\bQ,S} \ra F$ deforming this $\Db$, $\rho^{ss}_D$ is defined over $F$ even though \emph{loc.\ cit.} no longer applies in general when $F$ is $p$-adic. The existence of $\rho^{ss}_D$ over $F$ follows from the existence of a GMA structure for $E_\Db$. The argument is visible in Example \ref{exmp:ord_PsR_field}. 
\end{rem}

\subsection{Ordinary Cayley-Hamilton representations}
\label{subsec: ord C-H}

Our goal is to define a notion of ``ordinary'' for Cayley-Hamilton $G_{\Q_p}$-representations. For this section, we let $I_{\bQ_p} \subset G_{\Q_p}$ be the inertia group. (Note that $I_{\bQ_p}$ has nothing to do with the Eisenstein ideal $I$.) The coefficient ring $A$ can be in $\hat\cC_{W(\F)}$ or in $\cC_F$ where $F$ is a finite extension of $W(\F)[1/p]$.

The notion of ordinary that we will use is specific to Cayley-Hamilton representations of $G_{\bQ,S}$ or $G_{\Q_p}$ with residual pseudorepresentation $\Db = \psi(\omega^{-1} \oplus \theta^{-1})$, and is more restrictive than the definition given in the introduction (\S\ref{subsec:OP}) when $\theta^{-1}\omega\vert_{I_{\bQ_p}} = 1$. 

\begin{defn}
\label{defn:ord_C-H}
Given a lift of idempotents in $E$, write $\rho_{i,j}$ for the composition of the homomorphism $\rho: G_{\Q_p} \ra E$ with the $(i,j)$th coordinate of \eqref{eq:E-decomp}. We call a Cayley-Hamilton representation $(A, (E,D), \rho)$ of $G_{\Q_p}$ with induced pseudorepresentation $\Db$ \emph{ordinary} provided that there exists a lift of idempotents to $E$ such that 
\begin{enumerate}[leftmargin=2em]
\item $\rho_{1,2} = 0$, and
\item $\rho_{1,1}\vert_{I_{\bQ_p}} = \kcyc^{-1} \otimes_{\bZ_p} A$.
\end{enumerate}
We call a Cayley-Hamilton representation of $G_{\bQ,S}$ with induced pseudorepresentation $\Db$ \emph{ordinary} if its restriction to $G_{\Q_p}$ is ordinary. 
\end{defn}

\begin{rem}
\label{rem:H-ord}
Our running assumptions imply that $\omega\vert_{G_{\bQ_p}} \neq \theta\vert_{G_{\bQ_p}}$ but allow $\theta^{-1}\omega$ to be unramified at $p$ (i.e.~$\omega\vert_{I_{\bQ_p}} = \theta\vert_{I_{\bQ_p}}$), so that either $e_1$ or $e_2$ could correspond to the twist-unramified 1-dimensional quotient $G_{\bQ_p}$-representation. However, we know that the structure of the $\h[G_{\bQ,S}]$-module $H$ is such that the $e_1$-factor is always the twist-unramified quotient in the family. We have defined ``ordinary'' accordingly. A theory corresponding to the definition given in the introduction (\S\ref{subsec:OP}), even when $\theta^{-1}\omega$ is unramified at $p$, is given in \cite[\S7.3]{WE2017}. 
\end{rem}

By reading off Definition \ref{defn:ord_C-H}, we can readily construct a candidate quotient cutting out the ordinary condition. Given a lift of idempotents $(e_1, e_2)$ in $E^p_\Db$, consider the matrix coordinate functions $\rho_{i,j}: G_{\bQ_p} \ra E^p_\Db$. We consider the two-sided ideal $J^{\ord*}(e_1,e_2)$ of $E_\Db^p$ generated by the subsets 
\[
\rho_{1,2}(G_{\bQ_p}) \subset E_\Db^p \quad \text{ and } \quad (\rho_{1,1} - \kcyc^{-1} \otimes_{\bZ_p} R^p_\Db)(I_{\bQ_p}) \subset E_\Db^p. 
\]

\begin{lem}
\label{lem:no idem dependence}
If $(e_1,e_2)$ and $(e_1',e_2')$ are two lifts of idempotents, then $J^{\ord*}(e_1,e_2)=J^{\ord*}(e_1',e_2')$.
\end{lem}

\begin{proof}
By symmetry, it suffices to show that $J^{\ord*}(e_1',e_2') \subset J^{\ord*}(e_1,e_2)$. For the remainder of the proof, we set $J^{\ord*}=J^{\ord*}(e_1,e_2)$ and $E'=E^p_\Db/J^{\ord*}$. We write $\cA_{i,j} = e_jE^p_{\Db}e_i$ and $\rho_{i,j} = e_j\rho e_i$ and, similarly, $\cA_{i,j}' = e'_jE^p_{\Db}e'_i$ and $\rho'_{i,j} = e'_j\rho e'_i$.  Using Lemma \ref{lem:lifts_conj}, we fix $u \in E^\times$ such that $e'_i = ue_i u^{-1}$ for $i = 1,2$.

Writing $\cJ \subset R^p_\Db$ for the reducibility ideal of $D^p_\Db$, we claim that $\cJ \cdot E^p_\Db \subset J^{\ord*}$. This follows from Lemma \ref{lem:G-to-E}(3), as $\rho_{1,2}(G_{\Q_p})$ generates $\cA_{1,2} \subset E^p_\Db$ as an $R^p_\Db$-module. We know that $\cA_{1,2} \cdot \cA_{2,1}$ equals $\cJ$ in $\cA_{1,1} \cong R^p_\Db$ and $\cA_{2,1} \cdot \cA_{1,2}$ equals $\cJ$ in $\cA_{2,2} \cong R^p_\Db$, so altogether we have 
\[
\cJ \cdot E^p_\Db \subset (\cA_{1,2} \cdot \cA_{2,1} + \cA_{2,1} \cdot \cA_{1,2}) \cdot E^p_\Db \subset E^p_\Db \rho_{1,2}(G_{\Q_p}) E^p_\Db \subset J^{\ord*}
\]
as desired. 

We abbreviate $E^p_\Db/\cJ:=E^p_\Db/(\cJ \cdot E^p_\Db)$. Let $\overline{\cA}_{1,2}, \overline{\cA'}_{1,2} \subset E^p_\Db/\cJ$ denote the images of $\cA_{1,2}$ and $\cA'_{1,2}$, respectively, in the quotient. The key observation is that $\overline{\cA}_{1,2}$ is a two-sided ideal of $E^p_\Db/\cJ$, which follows from examining multiplication in a GMA under the implication $\overline{\cA}_{1,2} \cdot \overline{\cA}_{2,1} = \overline{\cA}_{2,1} \cdot \overline{\cA}_{1,2} = 0$ of Lemma \ref{lem:G-to-E}. Now check 
\[
\overline{\cA'}_{1,2} = (ue_2u^{-1})(E^p_\Db/\cJ)(u e_1 u^{-1}) = ue_2(E^p_\Db/\cJ) e_1 u^{-1} = u \overline{\cA}_{1,2} u^{-1} = \overline{\cA}_{1,2}. 
\]
Since $\cA_{1,2}$ maps to $0$ in $E'$ and $\cJ \cdot E^p_\Db \subset J^{\ord*}$, this implies that $\cA'_{1,2}$ maps to $0$ in $E'$.

Finally we claim that $(\rho'_{1,1} - \kcyc^{-1} \otimes_{\bZ_p} R_\Db)(I_{\bQ_p})$ is in $J^{\ord*}$, or, equivalently, its image in $E'$ vanishes. Let $\gamma \in I_{\Q_p}$ and let $x=\rho(\gamma) - \kcyc^{-1}(\gamma)$. Since 
\[
\rho'_{1,1}(\gamma) - \kcyc^{-1}(\gamma)=ue_1 u^{-1} xue_1 u^{-1}
\]
it suffices to show that $e_1 u^{-1} xue_1$ maps to $0$ in $E'$. However, since $\cA_{1,2}$ maps to $0$ in $E'$, we see that the map $E' \to E'$ given by $y \to e_1 y e_1$ is an algebra homomorphism with commutative image. Hence we have $e_1 u^{-1} xue_1 = e_1 x e_1 = 0$ in $E'$. 
\end{proof}

We will simply write $J^{\ord*}$ for the ideal $J^{\ord*}(e_1,e_2) \subset E_{\Db}^p$ for some choice of lift of idempotents $(e_1,e_2)$. By Lemma \ref{lem:no idem dependence}, $J^{\ord*}$ does not depend on this choice.

Now we check that $J^{\ord*}$ has the correct property. 
\begin{lem}
\label{lem:Jord* key property}
Let $(A, (E,D), \rho : G_{\bQ_p} \ra E^\times)$ be a Cayley-Hamilton representation of $G_{\bQ_p}$ with residual pseudorepresentation $\Db$. Then $\rho$ is ordinary if and only if $J^{\ord*}$ maps to zero under the induced map of Cayley-Hamilton algebras $(E^p_\Db, D^p_\Db) \ra (E,D)$. 
\end{lem}
\begin{proof}
If $J^{\ord*}$ maps to zero in $E$, then we know that $\rho$ is ordinary with respect to any lift of idempotents in $E$ that are in the image under $(E^p_\Db, D^p_\Db) \ra (E,D)$ of a lift of idempotents in $E^p_\Db$. 

Conversely, assume $\rho$ is ordinary with respect to some lift of idempotents $(e_1, e_2)$ in $E$. To prove that $J^{\ord*}$ maps to zero in $E$, we first construct a lift of idempotents $(e_1^*,e_2^*)$ in $E^p_{\Db}$ and let $(e_1',e_2')$ be the image in $E$. Then we will show that $\rho$ is ordinary with respect to $(e_1',e_2')$.

Choose $\delta \in G_{\Q_p}$ such that $\omega(\delta) \neq \theta(\delta)$ (the assumptions of \S \ref{subsec:setup} on $\theta$ imply that $\theta|_{G_{\Q_p}} \ne \omega|_{G_{\Q_p}}$). Let $D^{p,\red}_\Db=D^p_\Db \otimes_{R_\Db^p} R_\Db^p/\cJ$, $E_\Db^{p,\red}=E_\Db^{p}\otimes_{R_\Db^p} R_\Db^p/\cJ$ and $\rho^{p,\red}_\Db =\rho^p_{\Db} \otimes_{R_\Db^p} R_\Db^p/\cJ$, where $\cJ \subset R_\Db^p$ is the reducibility ideal. Then we must have $D^{p,\red}_\Db = \psi(\tilde{\chi}_1 \oplus  \tilde{\chi}_2)$, where $\tilde{\chi}_1,\tilde{\chi}_2$ are characters deforming $\omega^{-1}$ and $\theta^{-1}$, respectively. By Lemma \ref{lem:distinguished gives lifts}(1), there is a lift of idempotents $(e_1^{*,\red},e_2^{*,\red})$ in $E^{p,\red}_\Db$ given by 
\[
e_1^{*,\red} = \frac{\rho^{p,\red}_\Db(\delta)-\tilde{\chi}_2(\delta)}{\tilde{\chi}_1(\delta)- \tilde{\chi}_2(\delta)}, e_2^{*,\red}=1-e_1^{*,\red}.
\]
Now we let $e_1^*,e_2^* \in E^p_{\Db}$ be any complete set of primitive idempotents lifting $e_1^{*,\red},e_2^{*,\red}$ (such lifts exist because $\cJ \cdot E_\Db^p \subset \Jac(E_\Db^p)$ and $E^p_\Db$ is semiperfect).

Now we write $E$ and $\rho$ as
\[
E=\ttmat{A}{A_{1,2}}{A_{2,1}}{A}, ~ \rho = \ttmat{\chi_1}{0}{\rho_{2,1}}{\chi_2}
\]
with respect to the given idempotents $(e_1, e_2)$, where $\chi_1, \chi_2:G \to A^\times$ are characters that deform $\omega^{-1}$ and $\theta^{-1}$, respectively. We see that $\rho$ is reducible, so $E_\Db^p \to E$ factors through $E_\Db^{p,\red}$. By Lemma \ref{lem:distinguished gives lifts}(2), we have $\chi_i = \tilde{\chi}_i \otimes_{R_\Db} A$. 

With this notation, we see that
\[
e_1' = \frac{\rho(\delta)-\chi_2(\delta)}{\chi_1(\delta) - \chi_2(\delta)} = \ttmat{1}{0}{x}{0}, ~ e_2'=1-e_1' = \ttmat{0}{0}{-x}{1},
\]
where $x= \frac{\rho_{2,1}(\delta)}{\chi_1(\delta) - \chi_2(\delta)}$, and where $e_1',e_2'$ are the respective images of $e_1^*, e_2^*$. We easily compute that $e_1'\rho e_2'=0$ and $e_1'\rho e_1' =\chi_1 e_1'$, so $\rho$ is also ordinary with this choice of lift of idempotents. This implies that $J^{\ord*}(e_1^*,e_2^*) \subset E_\Db^p$ maps to zero in $E$, so we are done by Lemma \ref{lem:no idem dependence}.
\end{proof}

Now we construct a universal ordinary Cayley-Hamilton representation of the global Galois group $G_{\bQ,S}$. We could do the same for $G_{\Q_p}$ but we omit this. We implicitly use the map of Cayley-Hamilton algebras $(E^p_\Db, D^p_\Db) \ra (E_\Db, D^u_\Db)$, arising from restriction from $G_{\bQ,S}$ to $G_{\Q_p}$, in what follows. In particular, we write $J^{\ord*}$ for its image under $E^p_\Db \ra E_\Db$.  

\begin{defn}
\label{defn:univ_ord_objs}
Let $J_{R_\Db}^\ord \subset R_\Db$ be the ideal generated by the subsets $\Tr_{D^u_\Db}(J^{\ord*})$ and $D^u_\Db(J^{\ord*})$, and let $J^\ord$ be the two-sided ideal of $E_\Db$ generated by $J^{\ord*}$ and $J^\ord_{R_\Db}$. Let $E^\ord_\Db := E_\Db/J^\ord$ and let $R^\ord_\Db = R_\Db/J_{R_\Db}^\ord$.
\end{defn}

\begin{lem} 
\label{lem:ord_setup}
\begin{enumerate}[leftmargin=2em]
\item The kernel of the composite $R_\Db \to E_\Db \to E^\ord_\Db$ is $J_{R_\Db}^\ord$. 
\item $D^u_\Db : E_\Db \ra R_\Db$ descends to a unique pseudorepresentation 
\[
D^\ord_\Db : E^\ord_\Db \lra R^\ord_\Db 
\]
which is Cayley-Hamilton. 
\end{enumerate}
\end{lem}

\begin{proof}
For (1), it is clear that the kernel contains $J_{R_\Db}^\ord$. To show that it is not larger, it suffices to show that the resulting map $R_\Db^\ord = R_\Db/J_{R_\Db}^\ord \to E^\ord_\Db$ is injective. This will follow from (2) by Lemma \ref{lem:str-inj}.

Let $D': E_\Db \to  R^\ord_\Db$ be the composite of $D^u_\Db : E_\Db \ra R_\Db$ with $R_\Db \rsurj R^\ord_\Db$; it is a multiplicative polynomial law of degree $2$. To prove (2), we will apply Lemma \ref{lem: ker of pseudo} to show that $D'$ factors through $E^\ord_\Db$. By Lemma \ref{lem: ker of pseudo}, it is the same to show that $J^\ord \subset \ker(D')$, where $\ker(D')$ is defined in \eqref{eq:ker_defn}. That is, we must prove that $\Tr_{D',B}(xy) = 0$ and $D'_B(xy) = 0$ for all $x \in J^\ord$ and all $y \in E_\Db \otimes_{R_\Db} B$ for all commutative $R_\Db$-algebras $B$. Here $\Tr_{D',B}=(\Tr_{D'})_B: E_{\bar{D}}\otimes_{R_\Db} B\to B$ is the function associated to the polynomial law $\Tr_{D'}$.

Firstly, we observe that it suffices to restrict $x$ to elements of $J^{\ord *}$. Indeed, the additional generators $J^{\ord}_{R_\Db}$ of $J^\ord$ are scalars, so that they can be factored out of $\Tr_{D'}$ and $D'$ by homogeneity. Then, they are clearly sent to $0$ in $R^\ord_\Db$ by $\Tr_{D'}$ and $D'$ by definition of $R^\ord_\Db$. 

For $x \in J^{\ord *}$, the multiplicativity of $D'$ immediately implies that $D'_{B}(xy) = 0$ for any $y\in E_\Db \otimes_{R_\Db} B$. Indeed $D'_B(x \otimes 1_B)$ is $D'_{R_\Db}(x) \otimes 1_B \in R^\ord_\Db \otimes_{R_\Db}  B$, which is zero by construction of $R^\ord_\Db$. 

Likewise, to prove that $\Tr_{D',B}(xy) = 0$, it suffices to consider the case of a pure tensor $y = z \otimes b \in E_\Db \otimes_{R_\Db} B$ because the trace is linear. Because $J^{\ord *} \subset E_\Db$ is a two-sided ideal, we have $xz \in J^{\ord *}$, which makes $\Tr_{D',B}(xy) = \Tr_{D',R_\Db}(xz) \otimes b = 0 \otimes b = 0$ in $R^\ord_\Db \otimes_{R_\Db} B$. 

By Lemma \ref{lem: ker of pseudo} there is a unique degree $2$ multiplicative polynomial law  $D^\ord_\Db: E_\Db^\ord= E_\Db/J^\ord \to R^\ord_\Db$ over $R_\Db$ such that $D'$ factors through $D^\ord_\Db$. Since the $R_\Db$-structure on $E_\Db^\ord$ factors through $R_\Db^\ord$, $D^\ord_\Db$ can also be thought of as a degree $2$ multiplicative polynomial law over $R_\Db^\ord$ (see Example \ref{exmps: pseudo basic exmps}(6)). In other words, $D^\ord_\Db$ is a pseudorepresentation. In addition, it is clear that $D^\ord_\Db$ is Cayley-Hamilton because $D^u_\Db$ is Cayley-Hamilton. 
\end{proof}

Now that we have shown that $(E^\ord_\Db, D^\ord_\Db)$ is a Cayley-Hamilton algebra, we justify why $(E^\ord_\Db, D^\ord_\Db)$ deserves to be called the universal ordinary Cayley-Hamilton algebra. 

\begin{prop}
\label{prop:ord_C-H}
A Cayley-Hamilton representation of $G_{\bQ,S}$ valued in the Cayley-Hamilton algebra $(E,D)$ with residual pseudorepresentation $\Db$ is ordinary if and only if the associated Cayley-Hamilton representation $(E_\Db, D^u_\Db) \ra (E,D)$ factors through $(E^\ord_\Db, D^\ord_\Db)$. 
\end{prop} 

\begin{proof}
Let $(A, (E,D), \rho)$ be a Cayley-Hamilton representation of $G_{\bQ,S}$ with residual pseudorepresentation $\Db$. By Proposition \ref{prop:univ_C-H}, there is a local homomorphism $g: R_\Db \to A$ and an $A$-algebra homomorphism $f: E_\Db\otimes_{R_\Db}A \to E$ such that $D^u_\Db \otimes_{R_\Db} A= D \circ f$ and such that $\rho = f \circ \rho^u$.  Let $f':E_{\Db} \to E$ be the composite of the natural map $E_{\Db} \to E_\Db\otimes_{R_\Db}A$ with $f$. By Lemma \ref{lem:Jord* key property}, $(A, (E,D), \rho)$ is ordinary if and only if $J^{\ord *} \subset \ker(f')$.

Suppose that $(E_\Db, D^u_\Db) \ra (E,D)$ factors through $(E^\ord_\Db, D^\ord_\Db)$. Then 
\[
\ker(f') \supset \ker(E_\Db \to E^\ord_\Db) = J^\ord \supset J^{\ord *}
\]
and so $(A, (E,D), \rho)$ is ordinary.

Conversely, suppose that $(A, (E,D), \rho)$ is ordinary and so $J^{\ord *} \subset \ker(f')$. To show that $(E_\Db, D^u_\Db) \ra (E,D)$ factors through $(E^\ord_\Db, D^\ord_\Db)$, we need to show that $\ker(f') \supset J^\ord$ and $\ker(g) \supset J_{R_\Db}^\ord$. Let $x \in J^{\ord*}$. Then, by assumption, we have $0=f'(x)=f(x\otimes 1)$. Since $D^u_\Db \otimes_{R_\Db} A= D \circ f$, we have 
\[
(D^u_\Db \otimes_{R_\Db} A)(x \otimes 1) = D(f'(x))=D(0)=0.
\] 
Then the commutative square
\[
\xymatrix{
E_\Db \ar[rr]^-{D_\Db^u} \ar[d] & & R_\Db \ar[d]^-g \\
E_\Db \otimes_{R_\Db} A \ar[rr]^-{D^u_\Db \otimes_{R_\Db} A} & & A 
}
\]
implies that $D_\Db^u(x) \in \ker(g)$. 

Now, the equality $D^u_\Db \otimes_{R_\Db} A= D \circ f$ also implies that $Tr_{D^u_\Db \otimes_{R_\Db} A} = Tr_D \circ f$, and so a similar argument shows that $Tr_{D_\Db^u}(x) \in \ker(g)$. This shows that $D_\Db^u(J^{\ord *}) \subset \ker(g)$ and $Tr_{D_\Db^u}(J^{\ord *}) \subset \ker(g)$, so $J_{R_\Db}^\ord \subset \ker(g)$.

Consider the commutative square
\[
\xymatrix{
R_\Db \ar[r]^{g} \ar[d] & A \ar[d] \\
E_\Db \ar[r]^{f'} & E
}
\]
Then $\ker(f') \supset \ker(g)$ (the downward arrows are inclusions by Lemma \ref{lem:str-inj}) and so $J_{R_\Db}^\ord \subset \ker(f')$ as well. Since we already assume $J^{\ord *} \subset \ker(f')$, this implies $J^\ord \subset \ker(f')$. This completes the proof.
\end{proof}

\subsection{The ordinary pseudodeformation ring}
\label{subsec:RDb}

The guiding principle is that a pseudorepresentation should be called ordinary if it is induced by some ordinary representation. This notion suffices when the coefficient ring is a field, but requires rescuing over general coefficient rings: we replace ``representation'' with ``Cayley-Hamilton representation.'' We remark that Cho and Vatsal \cite{CV2003} also defined a notion of ordinary pseudorepresentation. 

\begin{defn}
\label{defn:ord_PsR}
Let $A$ be a Noetherian local $W(\bF)$-algebra and let $D: G_{\bQ,S} \ra A$ be a pseudorepresentation deforming $\Db$. Then we call $D$ \emph{ordinary} if there exists an $A$-valued ordinary Cayley-Hamilton representation $(A, (E, D_E), \rho: G_{\bQ,S} \ra E^\times)$ such that $D$ is equal to the composition $D_E \circ \rho$. The \emph{ordinary pseudodeformation functor} $\mathrm{PsDef}_\Db^\ord : \hat\cC_{W(\F)} \ra \mathrm{Sets}$ is 
\[
A \mapsto \{\text{pseudodeformations } D \text{ of } \Db \text{ over } A \mid D \text{ is ordinary}\}.
\]
A pseudorepresentation $D$ valued in a finite extension $F/\bQ_p$ with residual pseudorepresentation $\Db$ has an ordinary pseudodeformation functor $\mathrm{PsDef}_D^\ord : \cC_F \ra \mathrm{Sets}$ defined similarly. 
\end{defn}

\begin{exmp}
\label{exmp:ord_PsR_field} 
Let $F/\bQ_p$ be a finite extension and let $D: G_{\Q,S} \to F$ be a deformation of $\Db$. Theorem \ref{thm:PsR_ss} (with Remark \ref{rem:no_extn}) explains that there exists a unique semi-simple 2-dimensional $F$-linear representation $\rho^{ss}_D$ of $G_{\bQ,S}$ inducing $D$, inheriting idempotents from $E_\Db$ as explained in Example \ref{exmp:classical_case}. If $\rho^{ss}_D$ is ordinary (as defined in Definition \ref{defn:ord_C-H}), then $D$ is ordinary by definition. We claim that if $D$ is ordinary, then $\rho^{ss}_D$ is ordinary. Indeed, if $D$ is ordinary, then there is a generalized matrix $F$-algebra $(E,\cE)$ of the form $E=(\begin{smallmatrix}
F & B \\ C & F
\end{smallmatrix})$, and a homomorphism $\rho_E: G \to E^\times$ such that $D=D_\cE \circ \rho_E$. We know that $\rho_{1,2}(G_{\Q,S}) \cdot \rho_{2,1}(G_{\Q,S}) = 0$ if and only if $D$ is reducible (Proposition \ref{prop:red_ideal}). For convenience, replace $B$ by the subspace generated by $\rho_{1,2}(G_{\Q,S})$, and likewise for $C$, so we have $B \cdot C = 0$. 

First consider the case where $D$ is reducible. Then, since $B \cdot C=0$, the map $\pi: E \to M_2(F)$ given by
\[
\ttmat{a}{b}{c}{d} \mapsto \ttmat{a}{0}{0}{d}
\]
is an algebra homomorphism. Then $\pi \circ \rho_E$ is a semi-simple $F$-valued representation, and it is clearly ordinary. Moreover, since $D=D_\cE \circ \rho_E$, we have $D=\psi(\pi \circ \rho_E)$. By the uniqueness of Theorem \ref{thm:PsR_ss}, this implies that $\pi \circ \rho_E=\rho^{ss}_D$, and so $\rho^{ss}_D$ is ordinary.

Now assume that $D$ is irreducible. Then \cite[Defn.-Prop.\ 2.18]{chen2014} implies that there is an $F$-algebra isomorphism $\pi: E \simeq M_2(F)$. Again, $\pi \circ \rho_E$ is an ordinary semi-simple $F$-valued representation, and again $\pi \circ \rho_E=\rho^{ss}_D$.
\end{exmp}

\begin{rem}
This example shows why it is vexing to define ordinary pseudorepresentations of $G_{\bQ_p}$ if one wants the definition of an ordinary $G_{\bQ,S}$-pseudorepresentation to be ``$D$ such that $D\vert_{G_{\bQ_p}}$ is ordinary.'' Indeed, since many references (for example \cite[Intro.]{SW1999}) define ordinary for representations $\rho$ in terms of $\rho|_{I_p}$, where $I_p$ is the inertia group, one might hope to define ordinary $I_p$-pseudorepresentations. However, one runs into many difficulties in trying to do this sensibly. We illustrate these difficulties with an example based on a suggestion of Patrick Allen.

If $\rho$ is the representation attached to a Hida family, the weight $k$ specialization $\rho_k$ has the property that $(\rho_k)|_{I_p}$ is an extension of $\kcyc^{-1}$ by $\kcyc^{k-2}$ (see \S \ref{subsec: H prelims} for our normalizations). In particular $\psi((\rho_k)|_{I_p})=\psi(\kcyc^{k-2} \oplus \kcyc^{-1})$. Then, letting $\rho'=\rho_{2-p} \otimes \kcyc^{p-1}$, we see that $\rho'$ has the same residual representation as $\rho_p$, and that $\psi((\rho_p)|_{I_p})=\psi(\rho'|_{I_p})$. 

There are two reasons why we do not want to think of $\psi(\rho')$ as an ``ordinary weight $p$ pseudorepresentation" of $G_\bQ$. The deformation-theoretic reason is clear: if we construct the weight $p$ pseudodeformation ring with the condition that $D|_{I_p} = \psi(\kcyc^{p-2} \oplus \kcyc^{-1})$, this will map onto the weight $p$ ordinary Hecke algebra, but the map will not be an isomorphism because of the existence of $\psi(\rho')$. The Iwasawa-theoretic reason is visible in the case that $\rho$ is residually reducible, so that the $\rho_k$ produce extensions of characters by Ribet's method. One often wants to study extensions that are ``geometric", and, for weight $k\ge2$, these extensions are geometric. However, since $\rho'$ comes from a form of negative weight, the corresponding extension will be transcendental in nature.
\end{rem}

In light of Remark \ref{rem:H-ord}, the following proof of the representability of $\mathrm{PsDef}_\Db^\ord$ implies Theorem D of the introduction. Recall $R^\ord_\Db$ from Definition \ref{defn:univ_ord_objs}. 
\begin{thm}
\label{thm:ord_psdef_ring}
For $\Db = \psi(\omega^{-1} \oplus \theta^{-1})$, the quotient $R^\ord_\Db$ of $R_\Db$ represents $\mathrm{PsDef}_\Db^\ord$ with universal object $D^\ord_\Db: G_{\bQ,S} \ra R^\ord_\Db$. Given some $p$-adic ordinary pseudorepresentation $D: G_{\bQ,S} \ra F$ deforming $\Db$, where $\m \subset R^\ord_\Db[1/p]$ is the maximal ideal induced by $D$, then $\mathrm{PsDef}_D^\ord$ is pro-represented by $R^\ord_\Db[1/p]^\wedge_\m \otimes_{k(\m)} F$.
\end{thm}

\begin{proof}
Let $D: G_{\bQ,S} \ra A$ of $\Db$ be an ordinary pseudodeformation of $\Db$. Then, by definition, there is an ordinary Cayley-Hamilton representation $(A, (E,D_E), \rho)$ such that $D= D_E \circ \rho$. Proposition \ref{prop:univ_C-H} supplies a morphism $(E_\Db, D^u_\Db) \ra (E,D_E)$ of Cayley-Hamilton algebras lying over the homomorphism $R_\Db \ra A$ corresponding to $D$. Because we have assumed that $\rho$ is ordinary, Proposition \ref{prop:ord_C-H} factors this morphism of Cayley-Hamilton algebras by $(E^\ord_\Db, D^\ord_\Db) \ra (E, D_E)$. This morphism includes the data of a homomorphism $R^\ord_\Db \ra A$, showing that $R_\Db \ra A$ factors through $R_\Db \rsurj R^\ord_\Db$ when $D$ is ordinary. 

On the other hand, if $D: G_{\bQ,S} \ra A$ is a pseudodeformation of $\Db$ such that the corresponding map $R_\Db \ra A$ factors through $R^\ord_\Db$, then the ordinary Cayley-Hamilton representation 
\[
(A, (E^\ord_\Db \otimes_{R^\ord_\Db} A, D^{\ord}_\Db \otimes_{R^\ord_\Db} A), (E_\Db \ra E^\ord_\Db \otimes_{R_\Db^\ord} A) \circ \rho^u)
\]
induces $D$, making $D$ ordinary. 

Now let $D: G \ra F$ be the $p$-adic pseudorepresentation of the statement. Applying Theorem \ref{thm:chen-PsR}, we observe that for any $A \in \cC_F$ and a pseudorepresentation $D_A : G \ra A$ deforming $D$, the same argument applies to show that $D_A$ is ordinary if and only if the induced map $R_\Db \ra A$ factors through $R^\ord_\Db$. The quotient $R^\ord_D$ of $R_D$ arising as $R^\ord_\Db \otimes_{R_\Db} R_D$ via the map \eqref{eq:residual_to_0} is the ordinary pseudodeformation ring for $D$. 
\end{proof}

\section{Galois cohomology} 
\label{sec:galois_cohom}
In this section, we compute some Galois cohomology groups that will be related, in \S\ref{sec: modularity}, to the universal ordinary Cayley-Hamilton algebra. We first review Iwasawa cohomology and Galois duality theory. The results and techniques of this section are well-known to experts. See also \cite[\S9]{FK2012}.

\subsection{Iwasawa cohomology} 

In this subsection, we let $\cL$ denote the ring $\Z_p\lb\Z_{p,N}^\times\rb$. We consider the $\cL[G_{\bQ,S}]$-modules $\cL^{\dia{-}}$ and $\cL^\#$ given as follows. As $\cL$-modules, $\cL^{\dia{-}}=\cL^\#=\cL$, and the Galois action on $\cL^{\dia{-}}$ (resp.\ $\cL^\#$) is given by the homomorphism
$$
G_{\bQ,S} \onto \Gal(\Q_\infty/\Q) \simeq \Z_{p,N}^\times \subset \cL^\times
$$
(resp.\ its inverse).

For a finitely generated $\Z_p$-module $M$ with continuous $G_{\bQ,S}$-action, the cohomology with coefficients in $\cL^\# \hat{\otimes}_{\bZ_p} M$ is known as the Iwasawa cohomology of $M$. It can be computed using the following version of Shapiro's lemma.

\begin{lem}
\label{lem: shapiros lem}
Let $M$ be a finitely generated $\Z_p$-module with continuous $G_{\bQ,S}$-action. Then there are isomorphisms
\[
H^i(\Z[1/Np], \cL^\# \hat{\otimes}_{\bZ_p} M) \cong \varprojlim_r H^i(\Z[1/Np,\zeta_{Np^r}], M)
\]
of $\cL$-modules.
\end{lem}
\begin{proof}
Follows from Shapiro's lemma -- see \cite[Lem.\ 5.3.1]{lim2012}, for example.
\end{proof}

In particular,  $H^i(\Z[1/Np], \cL^\# \hat{\otimes}_{\bZ_p} M)=0$ for $i>2$, since $\Z[1/Np,\zeta_{Np^r}]_{et}$ has cohomological dimension $2$.

When $M=\Z_p(1)$, we can compute the groups appearing on the right hand side using Kummer theory.

\begin{lem}\label{lem: cohomology and class groups lem}
For each $r\ge 0$, there is an isomorphism
$$
H^1(\Z[1/Np,\zeta_{Np^r}],\Z_p(1)) \cong \Z[1/Np,\zeta_{Np^r}]^\times \otimes_\bZ \Z_p
$$
of $\cL$-modules and an exact sequence 
$$
0 \to \Pic(\Z[1/Np,\zeta_{Np^r}])\{p\} \to H^2(\Z[1/Np,\zeta_{Np^r}],\Z_p(1)) \to \bigoplus_{v \in S} \Z_p \xrightarrow{\Sigma} \Z_p  \to 0,
$$
where ``$\{p\}$" means the $p$-power torsion.
\end{lem}
\begin{proof}
By Kummer theory -- see \cite[Lem.\  2.1]{sharifi2011}, for example.
\end{proof}

Let $X_S=\displaystyle \varprojlim_r \Pic(\Z[1/Np,\zeta_{Np^r}])\{p\}$. We combine Lemmas \ref{lem: shapiros lem} and \ref{lem: cohomology and class groups lem} to get
\begin{cor}
\label{cor: Lambda cohomology and class groups}
There is an isomorphism
\[
H^1(\Z[1/Np],\cL^\#(1)) \cong \varprojlim_r (\Z[1/Np,\zeta_{Np^r}]^\times \otimes \Z_p)
\]
and an exact sequence
\[
0 \lra X_S \lra H^2(\Z[1/Np],\cL^\#(1)) \lra \bigoplus_{v \in S} \Z_p \buildrel\Sigma\over\lra \Z_p  \lra 0
\]
\end{cor}

Lemma \ref{lem: shapiros lem} also allows us to compute twists.

\begin{cor}\label{cor: Lambda cohomology and twists}
Let $M$ be a finitely generated $\Z_p$-module with continuous $G_{\bQ,S}$-action. Then there is an isomorphism
\[
H^i(\Z[1/Np],\cL^\# \hat{\otimes} M(1)) \cong H^i(\Z[1/Np],\cL^\# \hat{\otimes} M)(1)
\]
of $\cL$-modules.
\end{cor}

\subsection{Duality}
\label{subsec: tate}
 We recall Poitou-Tate duality following \cite[\S9.3]{FK2012},  \cite[\S1.6.12]{FK2006}.

Let $K$ be a number field, and let $U$ be a dense open subset of $\Spec O_K[1/p]$.  Let $T$ be a finite abelian $p$-group with a continuous action of $G_K$. Let $v$ be a place of $K$ and let $C(U,T)$ and $C(K_v,T)$ be the standard complexes that compute $H^i(U,T)$ and $H^i(K_v,T)$, respectively. Let $C_{(c)}(U,T)$ be the mapping cone of the map of complexes 
$$
C(U,T) \lra \bigoplus_{s \not \in U} C(K_v,T)
$$
and let $H_{(c)}^i(U,T)$ be the cohomology of the complex $C_{(c)}(U,T)$. By definition, there is a long exact sequence
$$
\dots \lra H_{(c)}^i(U,T) \lra H^i(U,T) \lra \bigoplus_{s \not \in U}  H^i(K_v,T) \lra  H_{(c)}^{i+1}(U,T) \lra \dots
$$

As usual in \'etale cohomology, we can use the finite coefficients version to define $p$-adic coefficient versions. Namely, if $T$ is a finitely generated $\Z_p$-module with a continuous action of $G_K$, we define
\[
H_{(c)}^i(U,T) = \varprojlim_r H_{(c)}^i(U,T/p^rT) \text{ and } H_{(c)}^i(U,T \otimes_{\Z_p} \Q_p) = H_{(c)}^i(U,T) \otimes_{\Z_p} \Q_p.
\]

We require a version of Poitou-Tate duality, which was first formulated using (something like) $H^i_{(c)}$ by Mazur \cite[\S2.4, pg.\ 538]{mazur1973}. The version we need is for $\cL$-modules, and there the Poitou-Tate duality is ``derived'' in a non-trivial way. If $M$ is a finitely generated $\cL$-module, let $M^*=\Hom_\cL(M,\cL)$ with $\cL$-module structure given as follows. An element $\sigma \in \Gamma \subset \cL$ acts on $h \in M^*$ by $(\sigma.h)(m)=h(\sigma^{-1}m)$. Let $\EcL^i(-) = \Ext^i_\cL(-,\cL)$ be the derived functors of $M \mapsto M^*$ (with their induced $\cL$-module structure). These are sometimes called the {\em (generalized) Iwasawa adjoint} functors.
 
\begin{prop}
\label{prop: Lambda Poitou-Tate duality} 
Let $T$ be a finitely generated projective $\cL$-module equipped with a continuous action of $G_K$, unramified at places outside $U$. Then there is a spectral sequence
\[
E_2^{i,j}=\EcL^i(H^{3-j}(U,T^*(1))) \implies \Hc^{i+j}(U,T). 
\] 
\end{prop}

\begin{proof}
See \cite[Prop.\ 5.4.3, pg.\ 99]{nekovar2006}, for example. 
\end{proof}

\subsection{Cohomology of $\Lambda$} Taking the $\theta$-component of the above computations, we can compute cohomology with coefficients in $\Lambda=\cL_\theta$.

\begin{cor}\label{cor: iwasawa cohom comp}
There are isomorphisms 
\begin{align*}
H^2(\Z[1/Np],\Lambda^\#(2)) \cong X_\chi(1), \qquad H^2(\Z[1/Np],\Lambda^\#(1)) \cong X_\theta, \\
H^1(\Z[1/Np],\Lambda^\#(2))\cong \E_\chi(1)=0, \qquad H^1(\Z[1/Np],\Lambda^\#(1)) \cong \E_\theta.
\end{align*}
\end{cor}
\begin{proof}
To prove the first line, we use Corollary \ref{cor: Lambda cohomology and class groups} and the fact that, as in \cite[Lem.\ 4.11]{sharifi2011}, taking $\theta$-component annihilates both the kernel of the map $X_S \to X$ and the cokernel of the map 
$$
X_S \lra H^2(\Z[1/Np],\cL^\#(1)).
$$

For the second line, note that
\[
\E \cong \varprojlim (\Z[1/Np,\zeta_{p^r}]^\times \otimes \Z_p),
\]
as in \cite[Lem.\ 2.1]{sharifi2011}, and $\E_\chi(1)=0$ since $\E^-=\Z_p(1)$ and $\chi \ne \kcyc$. So the second line also follows from Corollary \ref{cor: Lambda cohomology and class groups}.
\end{proof}

We also need to compute the compactly supported cohomology groups 
\[\Hc^i(\Z[1/Np],\Lambda^{\dia{-}}(-1))\] for $i=1,2$. For this we use the Poitou-Tate duality of Proposition \ref{prop: Lambda Poitou-Tate duality} applied to $U=\Spec \Z[1/Np] $ and $T=\cL^{\dia{-}}(-1)$; taking $\theta$-components, we obtain a spectral sequence
\begin{equation}
\label{eq: EL spectral sequence}
E_2^{i,j}=\EL^i(H^{3-j}(\Z[1/Np],\Lambda^\#(2))) \implies \Hc^{i+j}(\Z[1/Np],\Lambda^{\dia{-}}(-1)),
\end{equation}
where $\EL^i(-)$ is the $\theta$-component of $\EcL^i(-)$. We also use the following result of Jannsen, which characterizes the vanishing of certain $\EL^i(-)$.

\begin{prop}
\label{prop: Lambda-modules}
Let $M$ be a finitely generated $\Lambda$-module. Then $\EL^0(M)=\EL^2(M)=0$ if and only if $M$ is torsion and has no non-zero finite submodule.
\end{prop}
\begin{proof}
See \cite[\S3]{jannsen1989}.
\end{proof}

\begin{cor}
\label{cor: iwasawa compact cohom comp}
There are isomorphisms 
\[
\Hc^1(\Z[1/Np],\Lambda^{\dia{-}}(-1))=0, \qquad \Hc^2(\Z[1/Np],\Lambda^{\dia{-}}(-1))\cong\X(-2)_{\theta^{-1}}
\]
\end{cor}
\begin{proof}
We analyze the spectral sequence \eqref{eq: EL spectral sequence}. Cohomological dimension considerations imply that $E_2^{i,j}=0$ unless $i \in \{0,1,2\}$ and $j \in \{1,2,3\}$. We immediately see that
\[
\Hc^1(\Z[1/Np],\Lambda^{\dia{-}}(-1))\cong E_2^{0,1}=\EL^0(H^2(\Z[1/Np],\Lambda^\#(2))).
\]
But by Corollary \ref{cor: iwasawa cohom comp}, we have that $H^2(\Z[1/Np],\Lambda^\#(2))\cong X_\chi(1)$, which is a $\Lambda$-torsion module. It follows that
\[
\Hc^1(\Z[1/Np],\Lambda^{\dia{-}}(-1)) \cong \EL^0(H^2(\Z[1/Np],\Lambda^\#(2)))\cong \EL^0(X_\chi(1))=0.
\]

To compute $\Hc^2$, we again apply Corollary \ref{cor: iwasawa cohom comp} to see that
\[
E_2^{0,2}=\EL^0(H^1(\Z[1/Np],\Lambda^\#(2)))=0,
\]
and
\[
E_2^{2,1}=\EL^2(H^2(\Z[1/Np],\Lambda^\#(2)))\cong \EL^2(X_\chi(1))=0
\]
where the last equality follows from Proposition \ref{prop: Lambda-modules} and the theorem of Ferrero-Washington \cite{FW1979} that $X_\chi(1)$ has no non-zero finite submodule. Hence we have
\[
\Hc^2(\Z[1/Np],\Lambda^{\dia{-}}(-1)) \cong E_2^{1,1}=\EL^1(H^2(\Z[1/Np],\Lambda^\#(2))) \cong \EL^1(X_\chi(1)).
\]
The result follows from a well-known Iwasawa adjunction (cf.\ \cite[Cor.\ 4.4]{wake1}) 
\[
\EL^1(X_\chi(1)) \cong \X(-2)_{\theta^{-1}}. \qedhere
\]
\end{proof}

Now let $(f) \subset \Lam$ be a height one prime ideal. These computations allow us to compute cohomology of certain characters with with coefficients in the localization $\Lamfnochi$, and its finite length quotients.  

\begin{cor}
\label{cor: cohom Lamf comp}
There are isomorphisms
\[
H^1(\Z[1/Np],\Lamfnochi^\#(1)) \cong \E_{\theta,(f)} \quad \text{and} \quad H^1_{(c)}(\Z[1/Np],\Lamfnochi^{\dia{-}}(-1)) = 0.
\]
Moreover, for each $r \ge 0$, there are isomorphisms
$$
H^1_{(c)}(\Z[1/Np],\Lamfnochi^{\dia{-}}/(f^r)(-1)) \cong \,_{f^r} \X(-2)_{\theta^{-1},(f)}.
$$
Furthermore, if we assume $X_{\theta,(f)}=0$, then
$$
H^1(\Z[1/Np],\Lamfnochi^\#/(f^r)(1)) \cong \E_{\theta,(f)}/f^r\E_{\theta,(f)}.
$$
\end{cor}
\begin{proof}
The first two isomorphisms follow immediately from Corollary \ref{cor: iwasawa cohom comp} and Corollary \ref{cor: iwasawa compact cohom comp}. 

Next, consider the exact sequence
\[
0 \to \Lamfnochi^{\dia{-}}(-1) \xrightarrow{f^r} \Lamfnochi^{\dia{-}}(-1) \to \Lamfnochi^{\dia{-}}/(f^r)(-1) \to 0.
\]
Using the fact that $H^1_{(c)}(\Z[1/Np],\Lamfnochi^{\dia{-}}(-1)) = 0$, resulting long exact sequence in cohomology yields an isomorphism 
\[
H^1_{(c)}(\Z[1/Np],\Lamfnochi^{\dia{-}}/(f^r)(-1)) \isoto _{f^r}H^2_{(c)}(\Z[1/Np],\Lamfnochi^{\dia{-}}(-1)).
\]
Now applying Corollary \ref{cor: iwasawa compact cohom comp}, we have 
\[
H^1_{(c)}(\Z[1/Np],\Lamfnochi^{\dia{-}}/(f^r)(-1)) \cong \,_{f^r} \X(-2)_{\theta^{-1},(f)},
\]
as desired.

Finally, consider the exact sequence
\[
0 \to \Lamfnochi^\#(1) \xrightarrow{f^r} \Lamfnochi^\#(1) \to \Lamfnochi^\#/(f^r)(1) \to 0.
\]
Using the fact that $H^1(\Z[1/Np],\Lamfnochi^\#(1)) \cong \E_{\theta,(f)}$, resulting long exact sequence in cohomology yields a short exact sequence
\[
0 \to \E_{\theta,(f)}/f^r\E_{\theta,(f)} \to H^1(\Z[1/Np],\Lamfnochi^\#/(f^r)(1))  \to _{f^r} H^2(\Z[1/Np],\Lamfnochi^\#(1)) \to 0.
\]
By Corollary \ref{cor: iwasawa cohom comp}, we have $H^2(\Z[1/Np],\Lamfnochi^\#(1)) \cong X_{\theta,(f)}$. Hence, if we assume $X_{\theta,(f)}=0$, we have
\[
H^1(\Z[1/Np],\Lamfnochi^\#/(f^r)(1)) \cong \E_{\theta,(f)}/f^r\E_{\theta,(f)}. \qedhere 
\]
\end{proof}

Finally, we record a simple lemma that allows us to identify an $\Hc^1$ group.

\begin{lem}
\label{lem: Hc1 = ker}
Fix an ideal $\fn \subset \Lamfnochi$ and let
\[
\kappa :=\ker \bigg( H^1(\Z[1/Np],\Lamfnochi^{\dia{-}}/\fn(-1)) \to H^1(\Q_p,\Lamfnochi^{\dia{-}}/\fn(-1))\bigg).
\]
Then the natural map 
\[
H^1_{(c)}(\Z[1/Np],\Lamfnochi^{\dia{-}}/\fn(-1)) \onto \kappa 
\]
is an isomorphism.
\end{lem}
\begin{proof}
By the definition of $\Hc^i$ as a cone, there is an exact sequence
\[
\bigoplus_{\ell | Np} H^0(\Q_\ell,\Lamfnochi^{\dia{-}}/\fn(-1)) \to H^1_{(c)}(\Z[1/Np],\Lamfnochi^{\dia{-}}/\fn(-1)) \to \kappa \to \bigoplus_{\ell | N} H^1(\Q_\ell,\Lamfnochi^{\dia{-}}/\fn(-1)).
\]
Then it is enough to show $H^i(\Q_\ell,\Lamfnochi^{\dia{-}}/\fn(-1))=0$ for all pairs $(i,\ell)$ with $i \in \{0,1\}$ and $\ell |Np$ with $(i,\ell) \ne (1,p)$. Since there is an injection 
\[
H^i(\Q_\ell,\Lamfnochi^{\dia{-}}/\fn(-1)) \otimes_{\Lamfnochi}\Lamfnochi/(f) \to H^i(\Q_\ell,\Lamfnochi^{\dia{-}}/(f)(-1))
\] 
it suffices by Nakayama's lemma to consider the case $\fn=(f)$. Write $\nu_f$ for the $1$-dimensional $\Lamfnochi/(f)$-valued representation $\Lamfnochi^{\dia{-}}/(f)(-1)$, and $\nu_f^\vee$ for the dual representation.

By our assumptions on $\theta$ (see \S\ref{subsec:setup}), $\nu_f\vert_{G_{\Q_\ell}} \ne 1$ for all $\ell | Np$, so we have $H^0(\Q_\ell,\nu_f)=0$ for all $\ell \mid Np$. Now assume $\ell \mid N$. Then by Tate duality and the local Euler characteristic formula, we have
\[
\dim(H^1(\Q_\ell,\nu_f)) = \dim(H^0(\Q_\ell,\nu_f))+\dim(H^0(\Q_\ell,\nu_f^\vee(1))).
\]
Again our assumptions imply that $(\nu_f^\vee(1))\vert_{G_{\Q_\ell}} \ne 1$ for all $\ell | N$, so we have $H^0(\Q_\ell,\nu_f^\vee(1))=0$. This implies that $H^1(\Q_\ell,\nu_f)=0$.
\end{proof}

\section{Pseudo-modularity}
\label{sec: modularity}
In this section, we prove that the Galois action on $H$ is an ordinary Cayley-Hamilton representation and deduce that there is a surjective homomorphism $R_\Db^\ord \onto \fH$. We use this and Theorem D to deduce Theorem A. 

\subsection{The modular Cayley-Hamilton representation} 
\label{subsec:modCHrep}

\begin{defn}
For this section, let $D: G_{\bQ,S} \ra \h$ denote the 2-dimensional pseudorepresentation arising from the $\h[G_{\bQ,S}]$-module $H$. 
\end{defn}

This modular pseudorepresentation is described in \S\ref{subsec: H prelims}. In particular, it is shown to have values in $\h$ in Lemma \ref{lem: det and tr of H}, and, by Chebotarev density, is characterized by the formula \eqref{eq:modular_psr}. We now determine further important characteristics of $H$ (resp.\ $D$), showing in particular that it is an ordinary Cayley-Hamilton representation (resp.\ ordinary pseudorepresentation).

\begin{thm}
\label{thm:H_ord}
The action of $G_{\bQ,S}$ on $H$ induces an ordinary Cayley-Hamilton representation with induced ordinary pseudorepresentation $D$. More specifically, 
\begin{enumerate}[leftmargin=2em]
\item The residual pseudorepresentation $D \otimes_\h \h/\m_\h$ is equal to $\Db = \psi(\omega^{-1} \oplus \theta^{-1})$, i.e.~$D$ is a deformation of $\Db$.
\item $\End_\h(H)$ admits an $\h$-GMA structure making $\rho_H : G_{\bQ,S} \ra \Aut_\h(H)$ an ordinary Cayley-Hamilton representation.
\item As $D$ is the pseudorepresentation induced by $\rho_H$, $D$ is ordinary with $\Lambda$-valued determinant $\kcyc^{-1} \langle-\rangle^{-1}$.
\item The reducibility ideal of $D$ is $I \subset \h$, and $D \otimes_\h \h/I$ splits into $\psi(\kcyc^{-1} \oplus \langle-\rangle^{-1})$ valued in $\h/I \cong \Lambda/\xi_\chi$.
\end{enumerate}
\end{thm}

\begin{proof} 
Corollary \ref{cor:H-structure} gives us precisely what we require to show that $\End(H)$ is an ordinary Cayley-Hamilton representation and induces $D$. Let $e_2 \in \End_\h(H)$ be the idempotent arising from a choice of $\h$-linear left inverse
\[
H \rsurj H_{sub}, \quad x \mapsto e_2 \cdot x 
\]
of $H_{sub} \rinj H$, and let $e_1 = 1 - e_2$. The resulting GMA structure on $\End_\h(H)$ is of the form
\[
\End_\h(H) \cong 
\begin{pmatrix}
\End_\h(H_{quo}) & \Hom_\h(H_{sub}, H_{quo}) \\
\Hom_\h(H_{quo}, H_{sub}) & \End_\h(H_{sub})
\end{pmatrix}
\simeq
\begin{pmatrix}
\h & \h^\vee \\
\Hom_\h(\h^\vee,\h) & \h
\end{pmatrix}.
\]
The 2-dimensional $\h$-valued pseudorepresentation $D_H : \End_\h(H) \ra \h$ induced by this GMA structure (Theorem \ref{thm:CH_is_GMA}) has trace and determinant that are compatible with the trace and determinant on $\End_\h(H) \otimes_\h Q(\h) \simeq M_{2}(Q(\h))$ used in \S\ref{subsec: H prelims}. Consequently, the Cayley-Hamilton representation $\rho_H: G_{\bQ,S} \ra \Aut_\h(H)$ has induced pseudorepresentation $D_H \circ \rho_H$ identical to $D$. 

Theorem \ref{thm:H/I_sub} shows that $D$ has reducibility ideal contained in $I$ and that one factor has $G_{\bQ,S}$-action given by $\kcyc^{-1}$ modulo $I$. Because Lemma \ref{lem: det and tr of H} gives the determinant of $\rho_H$, proving (3), we compute that $D \equiv \psi(\kcyc^{-1} \oplus \langle-\rangle^{-1})$ modulo $I$. This establishes (1), because $\kcyc \equiv \omega$ and $\langle-\rangle \equiv \theta$ modulo $\m_\h$. 

Having established (1), it is clear that the idempotents $(e_1,e_2)$ are a valid lift of idempotents (in the sense of Definition \ref{defn:LoI}) for the ordering $\omega^{-1}, \theta^{-1}$ of the factors of $\rho^{ss}_\Db$. Then it is evident from Corollary \ref{cor:H-structure} that $\End(H)$ is an ordinary Cayley-Hamilton representation of $G_{\Q,S}$, since the matrix coordinates satisfy $\rho_{1,2}(G_{\Q_p}) = 0$ and $(\rho_{1,1} - \kcyc^{-1})(I_p) = 0$. 

The reducibility ideal of $H$ as a $G_{\Q_\infty,S}$-Cayley-Hamilton representation contains $I$ by \cite[Prop.\ 2.2]{wake1}, showing that the reducibility ideal as a $G_{\bQ,S}$-representation is precisely $I$, proving (4). 
\end{proof}

This pseudorepresentation can be extended to the Eisenstein locus as well, using the well-known fact that the Galois representation induced by $\Lambda$-adic Eisenstein series is reducible with factors $\kcyc^{-1}$ and $\langle -\rangle^{-1}$. We cannot directly use the Galois action on the $\fH$-module $\tilde H$ to build this pseudorepresentation. For note that as an $\fH$-module, $\tilde H$ has a composition series with three graded pieces isomorphic as $\fH$-modules to $\h$, $\h^\vee$, and $\Lambda$ (see \eqref{HtHL} and Theorem \ref{e-s theorem}). 

\begin{cor}
\label{cor:H-psrep}
There exists a unique ordinary pseudorepresentation $\tilde D: G_{\bQ,S} \ra \fH$ such that $\tilde D \otimes_{\fH} \h \simeq D$ and $\tilde D \otimes_\fH \fH/\I \simeq \psi(\kcyc^{-1} \oplus \langle -\rangle^{-1})$. Moreover,
\begin{enumerate}[leftmargin=2em]
\item The residual pseudorepresentation valued in $\fH/\m_\fH$ is identical to $\Db$. 
\item The determinant $\det(\tilde D): G_{\bQ,S} \ra \fH^\times$ of $\tilde D$ is $\kcyc^{-1} \langle-\rangle^{-1}$, valued in $\Lambda \subset \fH$.  
\item The reducibility ideal of $\tilde D$ is $\I \subset\fH$, and $D \otimes_\fH \fH/\I$ splits into $\psi(\kcyc^{-1} \oplus \langle-\rangle^{-1})$ valued in $\fH/\I \cong \Lambda$.
\end{enumerate}
\end{cor}

\begin{proof}
By the universal property of $R^\ord_\Db$ (Theorem \ref{thm:ord_psdef_ring}), there exists a canonical map $R^\ord_\Db \ra \h$ induced by $D$. Likewise, the $\Lambda$-valued representation $\kcyc^{-1} \oplus \langle -\rangle^{-1}$ is ordinary (with a lift of idempotents respecting the sum structure) and residually $\omega^{-1} \oplus \theta^{-1}$, so the associated pseudorepresentation $\psi(\kcyc^{-1} \oplus \langle -\rangle^{-1})$ induces a canonical map $R^\ord_\Db \ra \Lambda$. By Theorem \ref{thm:H_ord}(4), we have $D \otimes_\h \h/I \simeq \psi(\kcyc^{-1} \oplus \langle -\rangle^{-1}) \otimes_\Lambda \Lambda/\xi_\chi$. Therefore, because $\fH$ is a fiber product as in Lemma \ref{lem:H_pullback}, we have a canonical map from $R^\ord_\Db$. Consequently, we have determined a canonical ordinary pseudorepresentation valued in $\fH$, satisfying part (1). Claims (2) and (3) follow immediately from this construction. 
\end{proof}

In fact, these maps $R_\Db \ra \fH$ and $R_\Db \ra \h$ are surjective. 
\begin{lem}
\label{lem:R-H_surj}
The homomorphism $R_\Db \ra \fH$ corresponding to $\tilde D$ is surjective. Consequently, $R_\Db \ra \h$, which corresponds to $D$, is also surjective. 
\end{lem}
\begin{proof}
For $\Fr_q$ a choice of arithmetic Frobenius element for a prime $q \not\in S$, we know from Lemma \ref{lem: det and tr of H} that the homomorphism $R_\Db \ra \h$ sends the universal determinant $\det(\Fr_q) \in R_\Db$ to the image of $q^{-1}\langle q^{-1} \rangle$, and sends the universal trace $\Tr(\Fr_q)$ to the Hecke operator $q^{-1}T^*(q) \in \h$. 

Since $T^*(q)E_\Lambda=(q^{-1}+\dia{q^{-1}})E_\Lambda$, we see that the isomorphism 
$$
\fH \lrisom \h \times_{\Lambda/\xi_\chi} \Lambda
$$
is sends $T^*(q)$ to $(T^*(q),q^{-1}+\dia{q^{-1}})$. We see that the homomorphism $R_\Db \ra \fH$ sends $\det(\Fr_q)$ to $q^{-1}\langle q^{-1} \rangle$ and sends $\Tr(\Fr_q)$ to $q^{-1}T^*(q) \in \fH$.

We know that the $q^{-1}\langle q^{-1} \rangle$ generate $\Lambda$ over $\bZ_p$, and that $T^*(q)$ for $q \not\in S$ generate $\fH$ over $\Lambda$ \cite[Prop.\ 4.1.1]{ohta2007}. This establishes the surjectivity.
\end{proof}

\subsection{An Eisenstein intersection point on the ordinary eigencurve}
\label{subsec:intersect}

The previous subsection yields a canonical surjection $R^\ord_\Db \rsurj \fH$ corresponding to $\tilde D$. We are interested in comparing these rings when localized at an Eisenstein intersection point on the ordinary eigencurve $\cC^\ord$.

From now on, we work locally at a chosen Eisenstein intersection point. That is, we choose a prime divisor $f_\chi$ of $\xi_\chi$ as in the introduction. In the notation of \S\ref{subsec: weak gorenstein}, we are choosing an element $f_\chi \in \cP_\Lambda$ and we let $\p \in \cP_\fH$ be the corresponding element. Then $\p$ is an Eisenstein intersection point on $\cC^\ord = \Spec \fH[1/p]$, in the sense that it is in the Eisenstein locus $\Spec \fH/\I[1/p] = \Spec \Lambda[1/p]$ and in the ordinary cuspidal eigencurve $\cC^{\ord, 0} = \Spec \h[1/p]$.

Our goal is to prove that $\fH_\p$ is complete intersection using deformation theory. To apply deformation theory, we will work with the completion $\hat\fH_\p$ of $\fH_\p$ (resp.\ $\hat\h_\p$ of $\h_\p$). The property of being complete intersection is preserved under completion (see \S\ref{subsec: comm alg}).

The surjections $R_\Db \rsurj R^\ord_\Db \rsurj \fH \rsurj \h$ induce surjective local homomorphisms 
\begin{equation}
\label{eq:EI_surj}
\widehat{R_\Db[1/p]}_{\m'} \rsurj \widehat{R_\Db^\ord[1/p]}_\m \rsurj \hat\fH_\p \rsurj \hat\h_\p
\end{equation}
of the completions, where $\m', \m$ denote the maximal ideals of the respective rings corresponding to the closed point $\p$ of $\Spec \fH[1/p]$. The residual pseudorepresentation at $\p$ (which is valued in the residue field at $\p$, a finite extension of $\bQ_p$) will be denoted $\bar D_\p$. It is the specialization of the ``Eisenstein pseudorepresentation'' $\psi(\kcyc^{-1} \oplus \langle-\rangle^{-1}) : G_{\bQ,S} \ra \Lambda$ of Corollary \ref{cor:H-psrep} to the residue field of $\Lamf$. 
\begin{rem}
In contrast to the constancy of $\kcyc^{-1}$ in the Eisenstein family, $\langle-\rangle^{-1}$ varies. If $\xi_\chi$ is prime, then $f_\chi=\xi_\chi$ and the specialization of $\dia{-}$ at $(f_\chi)$ may be thought of as the power of the cyclotomic character that is the ``zero'' of $\xi_\chi$. 
\end{rem}

As discussed in Theorems \ref{thm:chen-PsR} and \ref{thm:ord_psdef_ring}, $R_\Db[1/p]_{\m'}^\wedge$ (resp.\ $R_\m^\ord := R_\Db^\ord[1/p]_\m^\wedge$) is the universal (resp.\ universal ordinary) deformation ring for the pseudorepresentation $\bar D_\p$ at $\p$, and there is an isomorphism $R^\ord_\m \cong R_\Db^\ord \otimes_{R_\Db} R_\Db[1/p]_{\m'}^\wedge$. We write $\tilde D_\p$ for $\tilde D \otimes_{\fH} \hat\fH_\p$ and write $D_\p$ for $D \otimes_\h \hat\h_\p$ for the modular pseudorepresentations. 

We now have a commutative diagram 
\[\xymatrix{
R^\ord_\m \ar[dr]_{\pi_R} \ar@{->>}[r]^\varphi & \hat\fH_\p \ar[d]^{\pi} \\
& \Lamf
}\]
where $\Lamf$ is a DVR and $\hat\fH_\p$ is a finite and flat $\Lamf$-algebra. We wish to prove that $\varphi$ is an isomorphism of complete intersections (that is, that $R^\ord_\m$ and $\hat\fH_\p$ are both complete intersection rings, and that $\varphi$ is an isomorphism). To do this, we apply a version of Wiles's numerical criterion \cite[Appendix]{wiles1995}. The version we use is due to Lenstra \cite{lenstra1993}, and is also explained in \cite{lci}. For a $\Lamf$-module $M$ of finite length, we denote by $\ell(M)$ the length of $M$. 

\begin{prop}
\label{prop: lci criterion}
Let $r$ denote the greatest integer such that $f_\chi^r | \xi_\chi$. Suppose that $\ell(\cJ_\m/\cJ_\m^2) \ge r$. Then $\varphi$ is an isomorphism of complete intersections.
\end{prop}
\begin{proof}
Note that $\cJ_\m=\ker(\pi_R)$ and $\I_\p=\ker(\pi)$. Let $\eta=\pi(\Ann_{\hat\fH_\p}(\I_\p)) \subset \Lamf$. The main theorem of \cite{lenstra1993} implies that $\varphi$ is an isomorphism of complete intersections if 
\[
\ell(\cJ_\m/\cJ_\m^2) \ge \ell(\Lamf/\eta).
\]
So, it suffices to show that $\ell(\Lamf/\eta)=r$. By Lemma \ref{lem:H_pullback}, we see that $\Ann_\fH(\I)=\ker(\fH \to \h)$, which, by Lemma  \ref{lem: ker h to fh}, is generated by $T_0$. The image of $T_0$ under the map $\fH \to \Lambda$ is the constant term of the Eisenstein series, which is $\xi_\chi$. It follows that $\eta=\xi_\chi \Lamf$, and hence $\ell(\Lamf/\eta)=r$.
\end{proof}

It remains to show that $\ell(\cJ_\m/\cJ_\m^2) \ge r$, which we will do, under the assumption $X_{\theta,(f_\chi)}=0$, by relating $\cJ_\m$ to Galois cohomology.

\subsection{Reducible ordinary Cayley-Hamilton representations}
\label{subsec: reducible CH}

To relate the reducibility ideal $\cJ \subset R^\ord_\Db$ to Galois cohomology, we first identify the reducible ordinary deformation ring $R^\red_{\Db}$.

We now fix a lift of idempotents in $E^\ord_\Db$ arising as the image of a lift of idempotents in $E^p_\Db$. We will follow Example \ref{exmp:2x2GMA} and use these idempotents to write $E^\ord_\Db$ in $2\times 2$-matrix notation, and write $\rho_{i,j}$ for the coordinates of $\rho$. By Lemma \ref{lem:no idem dependence}, we know that $\rho_{1,2}(G_{\Q_p})$ vanishes with respect to these idempotents. This defines a GMA structure on all ordinary GMA representations of $G_{\Q,S}$ with residual pseudorepresentation $\Db$; we will use this GMA structure without further comment. 

\begin{prop}
\label{prop:ord_red_CH}
\ 
\begin{enumerate}[leftmargin=2em]
\item
There is a \emph{universal reducible ordinary Cayley-Hamilton algebra} $E^\red_\Db$, a quotient of $E_\Db$, such that a Cayley-Hamilton representation of $G_{\bQ,S}$ valued in the Cayley-Hamilton algebra $(E,D)$ with residual pseudorepresentation $\Db$ is reducible ordinary if and only if its induced map $E_\Db \ra E$ factors through $E^\red_\Db$.
\item There is a universal reducible ordinary pseudodeformation $D^\red$ of $\Db$ valued in $R^\red_\Db$, which is defined to be the image of $R_\Db$ in $E^\red_\Db$, making $E^\red_\Db$ an $R^\red_\Db$-GMA. There is a natural isomorphism $R^\red_\Db \risom \Lambda$ associated to the equivalence $D^\red \cong \psi(\kcyc^{-1} \oplus \langle-\rangle^{-1})$. 
\item The natural map $R^\ord_\Db \ra R^\red_\Db$ is a split surjection of $\Lambda$-algebras inducing an isomorphism $R^\ord_\Db/\cJ \risom R^\red_\Db \cong \Lambda$.
\item The natural homomorphism $E^\red_\Db \ra E^\ord_\Db/\cJ E^\ord_\Db$ is surjective.
\end{enumerate}

\end{prop}

\begin{proof}
In analogy with the construction of the universal ordinary Cayley-Hamilton algebra (Definition \ref{defn:univ_ord_objs}), we consider the two-sided ideal $J^\red \subset E_\Db$ generated by both $J^\ord$ and also $\rho_{1,2}(G_{\bQ,S}) \cdot \rho_{2,1}(G_{\bQ,S})$, in the notation of Remark \ref{rem:mult_dot}. Set $E^\red_\Db := E_\Db/J^\red$. 

Using the same arguments as the proof of Theorem \ref{thm:ord_psdef_ring}, the image $R^\red_\Db$ of $R_\Db$ in $E^\red_\Db$ may be checked to be the universal reducible ordinary pseudodeformation ring $R^\red_\Db$. 

To prove (2), we first show that there is an isomorphism between the following two deformation functors $\hat\cC_{W(\F)} \to \mathrm{Sets}$:
\[
F_1:A \mapsto \{\text{Ordinary reducible deformations of } \Db _{/A} \} 
\]
and
\[
F_2:A \mapsto \{(\nu_1,\nu_2) \ | \ \bar{\nu}_1 =  \bar{\nu}_2=1, \nu_1|_{I_p}=1\}
\]
where $\nu_i :G_{\Q,S} \to A^\times$ are characters. (Here and elsewhere in the proof, a bar indicates reduction modulo the maximal ideal of $A$.)

The isomorphism sends $(\nu_1,\nu_2) \in F_2(A)$ to $\psi(\nu_1\kcyc^{-1} \oplus \nu_2\theta^{-1}) \in F_1(A)$, which is clearly ordinary. Conversely, given $D \in F_1(A)$, choose an ordinary Cayley-Hamilton representation $(A, (E, D_E), \rho: G_{\bQ,S} \ra E^\times)$ such that $D$ is equal to the composition $D_E \circ \rho$. We can write $\rho$ in GMA notation as
\[
\rho: \sigma \mapsto \ttmat{a(\sigma)}{b(\sigma)}{c(\sigma)}{d(\sigma)}
\]
with $\bar{a}=\omega^{-1}$ and $\bar{d}=\theta^{-1}$. Since $D$ is reducible, we have $b(\sigma) \cdot c(\tau)= 0$ for all $\sigma,\tau \in G_{\Q,S}$. This implies that $a,d:G_{\Q,S} \to A^\times$ are characters. Moreover, since $\rho$ is ordinary, we have $a|_{I_p}=\kcyc^{-1}$. Then we send $D$ to the pair $(a \cdot \kcyc,d \cdot \theta)$. It is easily checked that these maps are between $F_1(A)$ and $F_2(A)$ are mutually inverse and functorial in $A$.

Since $F_1$ is represented by $R_\Db^\red$, to prove (2) it remains to see that $F_2$ is represented by $\Lambda$. First we note that the image of $\nu_1$ is contained in the pro-$p$ group $1+\m_A$, so the kernel of $\nu_1$ is the absolute Galois group of a pro-$p$ abelian extension of $\Q$ that is unramified outside primes dividing $N$. Since $p \nmid N\phi(N)$, we see by class field theory that $\nu_1=1$. Then $F_2$ is the deformation functor of the trivial character of $G_{\Q,S}$. Since $\Gamma$ is canonically isomorphic to maximal abelian pro-$p$ quotient of $G_{\bQ,S}$, the description of the deformation ring of a character given in \cite[\S1.4]{mazur1989} shows that $\Lambda$ canonically represents $F_2$. Tracing though these identifications, we see that $D^\red$ pulls back to the pseudorepresentation $\psi(\kcyc^{-1} \oplus \dia{-}^{-1})$. This completes the proof of (2).

On one hand, $D^\red$ induces a canonical $\Lambda$-algebra homomorphism $R^\ord_\Db \ra R^\red_\Db$. On the other hand, there is a canonical splitting homomorphism $R^\red_\Db \ra R^\ord_\Db/\cJ$ by the universal property of $R^\red_\Db$. This establishes (3).

The kernel of $E_\Db \rsurj E^\ord_\Db/\cJ \E^\ord_\Db$ clearly contains $J^\ord$ and the reducibility ideal of $R_\Db$, so it contains $J^\red \subset E_\Db$, which we defined to be the two-sided ideal generated by them. This proves (4).
\end{proof}

We now revert back to the setting of \S\ref{subsec:intersect}, localizing all of the rings and algebras via $\otimes_\Lambda \Lambda_{(f_\chi)}$ and completing at the respective maximal ideals of each ring. We write these GMAs as 
\[
E^\ord_\m \cong \begin{pmatrix}
R^\ord_\m & B^\ord \\
C^\ord & R^\ord_\m
\end{pmatrix}
, \qquad
E^\red_\m \cong \begin{pmatrix}
\Lambda_{(f_\chi)} & B^\red \\
C^\red & \Lambda_{(f_\chi)}
\end{pmatrix}.
\]

By Proposition \ref{prop:red_ideal}, the multiplication map $m = \varphi_{1,2,1}$ of the GMA structure of $E^\ord_\m$ induces a surjection $B^\ord \otimes C^\ord \onto \cJ_\m$. To determine $\cJ_\m/ \cJ_\m^2$, we  first compute $E^\red_\m$, and then use the map of Proposition \ref{prop:ord_red_CH}(4) to relate it to $B^\ord \otimes C^\ord$.

\begin{prop}
\label{prop:ord_red}
Assume that $X_{\theta,(f_\chi)}=0$. There exists an isomorphism 
\[
E^\red_\m \cong \begin{pmatrix}
\Lambda_{(f_\chi)} & B^\red \\
C^\red & \Lambda_{(f_\chi)}
\end{pmatrix}
\lrisom 
\begin{pmatrix} 
\Lambda_{(f_\chi)} & X_{\chi, (f_\chi)}(1) \\
\X_{\chii,(f_\chi)}^\#(1) & \Lambda_{(f_\chi)}
\end{pmatrix}
\]
Moreover, $C^\red$ is free of rank $1$ over $\Lamf$.
\end{prop}
\begin{proof}
We start by proving a lemma.
\begin{lem}
For any ideal $\fn \subset \Lambda_{(f_\chi)}$, there are natural isomorphisms
\[
\Hom_{\Lambda_{(f_\chi)}}(B^\red, \Lambda_{(f_\chi)}/\fn) \lrisom \Hc^1(\Z[1/Np],\Lamf^{\dia{-}}/\fn(-1)) 
\]
and
\[
\Hom_{\Lambda_{(f_\chi)}}(C^\red, \Lambda_{(f_\chi)}/\fn) \lrisom  H^1(\Z[1/Np],\Lamf^\#/\fn(1)).
\]
\end{lem}
\begin{proof}

The construction of an injection 
\begin{equation}
\label{eq: Bdual to ext}
\Hom_{\Lambda_{(f_\chi)}}(B^\red, \Lambda_{(f_\chi)}/\fn) \rinj \Ext^1_{\Lambda_{(f_\chi)}/\fn[G_{\bQ,S}]}(\langle-\rangle^{-1}, \kcyc^{-1})
\end{equation}
is the content of \cite[Thm.\ 1.5.5]{BC2009}. We have
\[
\Ext^1_{\Lambda_{(f_\chi)}/\fn[G_{\bQ,S}]}(\langle-\rangle^{-1}, \kcyc^{-1}) = H^1(\Z[1/Np],\Lamf^{\dia{-}}/\fn(-1)),	
\]
and, by Lemma \ref{lem: Hc1 = ker}, it remains to show that the image of \eqref{eq: Bdual to ext} is the kernel
\[
\ker \bigg( H^1(\Z[1/Np],\Lamf^{\dia{-}}/\fn(-1)) \to H^1(\Q_p,\Lamf^{\dia{-}}/\fn(-1))\bigg).
\]
The ordinary condition implies that the image of \eqref{eq: Bdual to ext} is contained in this kernel; indeed, the restriction to $G_{\Q_p}$ of the cocycles vanishes. To show the other inclusion, we follow the argument of \cite[Thm.\ 1.5.6]{BC2009}. An element of this kernel may be though of as the class of an extension $\mathcal{E}$ of $\Lamf^\#/\fn$ by $\Lamf/\fn(-1)$ that is split as a $G_{\Q_p}$-extension. The form of this extension and the fact that it is split as a $G_{\Q_p}$-extension imply that $\End_{\Lamf}(\mathcal{E})$ is a reducible ordinary Cayley-Hamilton representation with residual pseudorepresentation $\Db$. Hence the map $\Lamf[G_{\Q,S}] \to \End_{\Lamf}(\mathcal{E})$ factors through $E^\red_\m$. The $(1,2)$-coordiate (in GMA-notation) of this map is a homomorphism $f_\mathcal{E}:B^\red \to \Lamf/\fn$. By construction, the image of $f_\mathcal{E}$ under $\eqref{eq: Bdual to ext}$ is the class of $\mathcal{E}$. This shows that the image of $\eqref{eq: Bdual to ext}$ is $H^1_{(c)}(\Z[1/Np],\Lamf^{\dia{-}}/\fn(-1))$.

The proof for $C^{\red}$ is the same, except that there are no local restrictions whatsoever. 
\end{proof}

In Corollary \ref{cor: cohom Lamf comp}, we computed cohomology groups that appear in the lemma. In the remainder of this proof, we use those computations without further comment. We start by analyzing $B^\red$. Taking $\fn=0$, we have 
\[
\Hom_{\Lambda_{(f_\chi)}}(B^\red, \Lambda_{(f_\chi)})  \cong \Hc^1(\Z[1/Np],\Lamf^{\dia{-}}(-1)) = 0.
\]
In particular, $B^\red$ is $\Lamf$-torsion. Taking $\fn=(f_\chi^r)$, we have
\[
\Hom_{\Lambda_{(f_\chi)}}(B^\red, \Lambda_{(f_\chi)}/(f_\chi^r)) \cong \Hc^1(\Z[1/Np],\Lamf^{\dia{-}}/(f_\chi^r)(-1))  \cong \,_{f_\chi^r} (\X(-2)_{\theta,(f_\chi)}).
\]
In particular, $B^\red$ is finitely generated, and its annihilator is the same as the annihilator of $\X(-2)_{\theta,(f_\chi)}$ -- say the annihilator is $(f_\chi^n)$. Then, since $\Lamf$ is a PID and $B^\red$ is a finitely generated torsion module, the natural map
$$
B^\red \lra \Hom_{\Lambda_{(f_\chi)}}( \Hom_{\Lambda_{(f_\chi)}}(B^\red, \Lambda_{(f_\chi)}/(f_\chi^n))  , \Lambda_{(f_\chi)}/(f_\chi^n))
$$
is an isomorphism. Then we have
$$
B^\red \cong \Hom_{\Lambda_{(f_\chi)}}(\X(-2)_{\theta,(f_\chi)},\Lambda_{(f_\chi)}/(f_\chi^n)).
$$
Since $f_\chi^n$ kills $\X(-2)_{\theta,(f_\chi)}$, we have
$$
\Hom_{\Lambda_{(f_\chi)}}(\X(-2)_{\theta,(f_\chi)},\Lambda_{(f_\chi)}/(f_\chi^n)) \cong \,_{f_\chi^n} \EL^1(\X(-2)_{\theta,(f_\chi)})
$$
But $\EL^1(\X(-2)_{\theta,(f_\chi)}) \cong X_{\chi,(f_\chi)}(1)$ and both are annihilated by $f_\chi^n$  (see \cite[Prop.\ 2.2]{wake2}). We have $B^\red \cong X_{\chi,(f_\chi)}(1)$.

Now we turn our attention to $C^\red$. For $\fn = 0$ or $\fn = (f_\chi^r)$, we have
\[
\Hom_{\Lambda_{(f_\chi)}}(C^\red, \Lambda_{(f_\chi)}/\fn) \cong H^1(\Z[1/Np],\Lamf^\#/\fn(1))  \cong \E_{\theta,(f_\chi)}/\fn \E_{\theta,(f_\chi)}.
\]
From this, we see that $C^\red$ is free of rank $1$ over $\Lamf$. Then we have
\begin{align*}
C^\red & \cong \Hom_{\Lambda_{(f_\chi)}}(\Hom_{\Lambda_{(f_\chi)}}(C^\red, \Lambda_{(f_\chi)}), \Lambda_{(f_\chi)}) \\
& = \Hom_{\Lambda_{(f_\chi)}}(\E_{\theta,(f_\chi)}, \Lamf).
\end{align*}

Finally, the equality $\Hom_{\Lambda_{(f_\chi)}}(\E_{\theta,(f_\chi)}, \Lamf) \cong \X_{\chii,(f_\chi)}^\#(1)$ can be proven using Proposition \ref{prop: Lambda Poitou-Tate duality}. We omit this here, since all we use in the proof of the main theorems is the fact that $C^\red$ is free of rank $1$ over $\Lamf$.
\end{proof}

Now we can achieve our goal of relating $\cJ_\m/\cJ_\m^2$ to Galois cohomology. 
\begin{cor}
\label{cor:JmodJsquared}
Assume that $X_{\theta,(f_\chi)}=0$. Then there is a canonical surjection of $\Lambda_{(f_\chi)}$-modules 
\[
X_{\chi, (f_\chi)}(1) \otimes_{\Lambda_{(f_\chi)}} \X_{\chii,(f_\chi)}^\#(1) \rsurj \cJ_\m/\cJ_\m^2.
\]
Moreover, there is a non-canonical surjection $X_{\chi, (f_\chi)}(1) \onto \cJ_\m/\cJ_\m^2$.
\end{cor}
\begin{proof}
Combining the surjective map $E^\red_\m \rsurj E^\ord_\m/\cJ_\m E^\ord_\m$ induced by Proposition \ref{prop:ord_red_CH}(4) with Proposition \ref{prop:ord_red}, we have a surjection 
\[
X_{\chi, (f_\chi)}(1) \otimes_{\Lambda_{(f_\chi)}} \X_{\chii,(f_\chi)}^\#(1) \rsurj B^\ord/\cJ_\m B^\ord \otimes_{\Lambda_{(f_\chi)}} C^\ord/\cJ_\m C^\ord.
\]
Proposition \ref{prop:red_ideal} shows that multiplication on the off-diagonal components of $E^\ord_\m$ has image in $R^\ord_\m$ equal to the reducibility ideal, i.e.\ 
\[
m: B^\ord/\cJ_\m B^\ord \otimes_{\Lambda_{(f_\chi)}} C^\ord/\cJ_\m C^\ord \rsurj \cJ_\m/\cJ_\m^2.
\]
Composing these, we obtain the desired canonical surjection. The non-canonical surjection $X_{\chi, (f_\chi)}(1) \onto \cJ_\m/\cJ_\m^2$ can be defined by choosing a basis for the rank $1$ free module $C^\red$, and then repeating the same argument.
\end{proof}

Below, we will only use the non-canonical surjection $X_{\chi, (f_\chi)}(1) \onto \cJ_\m/\cJ_\m^2$. The construction of this surjection does not use the fact that $C^\red \cong \X_\chii^\#(1)$, but merely that it is free of rank 1 (assuming that $X_{\theta,(f_\chi)}=0$).

\subsection{Pseudo-modularity}
\label{subsec: pseudomod}

Our goal is to prove the following main theorem (Theorem E of the introduction), which inspires the term ``pseudo-modularity.'' 

\begin{thm}
\label{thm: psmod}
Assume that $X_{\theta,(f_\chi)}=0$. Then the map $\varphi : R^\ord_\m \onto \hat\fH_\p$ is an isomorphism of complete intersections.
\end{thm}

By Proposition \ref{prop: lci criterion}, it suffices to show that $\ell(\cJ_\m/\cJ_\m^2) \ge r$ (recall that $\ell(M)$ denotes the length of a $\Lamf$-module $M$ and that $r=\ell(\Lamf/\xi_\chi \Lamf)$). We start with a lemma.

\begin{lem}
\label{lem: length of I/I2}
We have $\ell(I_\p/I_\p^2) \ge r$.
\end{lem}
\begin{proof}
The proof uses Fitting ideals; for the definition and basic information on Fitting ideals, see \cite[\S1]{lci}.

First note that $\Fitt_{\h_p}(I_\p) \subset \Ann_{\h_\p}(I_\p)$ \cite[Prop.\ 1.1(i)]{lci}. Note that, by Lemma \ref{lem:H_pullback}, we have $\ker(\fH \onto \h)=\Ann_\fH(\I)$. Then, since $I$ is the image of $\I$ in $\h=\fH/\Ann_\fH(\I)$, we have $\Ann_\h(I)=0$. It follows that $\Fitt_{\h_p}(I_\p)=0$.

Then, applying \cite[Prop.\ 1.1(ii)]{lci}, we have
\[
0 = \Fitt_{\h_p}(I_\p) \cdot \h_\p/I_\p = \Fitt_{\h_p/I_\p}(I_\p/I_\p^2) = \Fitt_{\Lamf/\xi_\chi \Lamf}(I_\p/I_\p^2).
\]
This implies $\Fitt_{\Lamf}(I_\p/I_\p^2) \subset \xi_\chi \Lamf = f_\chi^r \Lamf$. 

Now, it follows from the structure theorem for modules over a PID that, for a finite length $\Lamf$-module $M$, we have $\Fitt_{\Lamf}(M)=f_\chi^{\ell(M)} \Lamf$. Hence we have $f_\chi^{\ell(I_\p/I_\p^2)}\Lamf \subset f_\chi^r \Lamf$, and so $\ell(I_\p/I_\p^2) \ge r$.
\end{proof}

\begin{prop}
\label{prop: X=J=I}
Assume that $X_{\theta,(f_\chi)}=0$. Then there are isomorphisms $X_{\chi, (f_\chi)}(1) \simeq \cJ_\m/\cJ_\m^2 \cong\I_\p/\I_\p^2 \cong I_\p/I_\p^2$, and these all have length $r$ as $\Lamf$-modules.
\end{prop}
\begin{proof}
By Corollary \ref{cor:JmodJsquared}, there is a surjection
\[
X_{\chi, (f_\chi)}(1) \onto \cJ_\m/\cJ_\m^2.
\]
By Corollary \ref{cor:H-psrep}, $\I \subset \fH$ is the reducibility ideal of $\tilde{D}$, and so the map $\varphi : R^\ord_\m \onto \hat\fH_\p$ induces a surjection $\cJ_\m \onto \I_\p$. By Lemma \ref{I=I cor}, there is an isomorphism $\I_\p \cong I_\p$. Hence we have
\begin{equation}
\label{eq:X to I/I2}
X_{\chi, (f_\chi)}(1) \onto \cJ_\m/\cJ_\m^2 \onto \I_\p/\I_\p^2 \lrisom I_\p/I_\p^2.
\end{equation}
Since $X_\chi(1)$ has characteristic power series $\xi_\chi$, we have $\ell(X_{\chi, (f_\chi)}(1))=r$, and hence $\ell(I_\p/I_\p^2) \le r$, since it is a quotient of $X_{\chi, (f_\chi)}(1)$. On the other hand, Lemma \ref{lem: length of I/I2} says that $\ell(I_\p/I_\p^2) \ge r$ and so $\ell(I_\p/I_\p^2) = r$. 

Now \eqref{eq:X to I/I2} is a sequence of surjections of finite length $\Lamf$-modules where the first term and last term have the same length. This implies that all the surjections are isomorphisms, and the proposition follows.
\end{proof}

By Proposition \ref{prop: lci criterion}, this completes the proof of Theorem \ref{thm: psmod}.

\section{Sharifi's conjecture and other applications}
\label{sec:sharifi}

In this section, we discuss applications of our main theorem. These include new results on Sharifi's conjecture, the geometry of the ordinary eigencurve, and a noncommutative enhancement of Theorem \ref{thm: psmod}'s pseudo-modularity result. 

\subsection{Sharifi's conjecture}
\label{subsec: sharifi's conj}
In this subsection, we discuss Sharifi's conjecture and the known results on it \cite{sharifi2011, FK2012, FKS2014}.

Sharifi has constructed a homomorphism of $\Lambda$-modules
$$
\U: X_\chi(1) \lra H^-/IH^-
$$
that he conjectures to be an isomorphism \cite[Conj.\ 4.12]{sharifi2011}. In our notation, it may be described as follows. The $(1,2)$-coordinate of the Cayley-Hamilton representation $\rho_H \otimes_\h \h/I$ is a map
$$
G_{\bQ,S} \lra \Hom_{\h/I}(H^+/IH^+,H^-/IH^-).
$$
Composing with the canonical isomorphism $\h/I \cong H^+/IH^+$, we get a map
$$
b: G_{\bQ,S} \lra H^-/IH^-.
$$
When we restrict to $G_{\Q_\infty,S}$, then $b$ becomes a homomorphism, and moreover, since $\rho_H$ is ordinary, it factors through $X_\chi(1)$. This is the map $\U$.

Sharifi also constructed a map
$$
\tilde \varpi: H^- \lra X_\chi(1)
$$
that he conjectured to factor through $H^-/IH^-$, and that the resulting map $\varpi$ will be the inverse of $\U$ \cite[Conj.\ 5.8 and the remark following]{sharifi2011}. Fukaya and Kato have proven that $\tilde \varpi$ factors through $H^-/IH^-$ \cite[\S\S5.1-5.2]{FK2012} and have partial results toward Sharifi's conjecture. Their main result is the following.

\begin{thm}[{\cite[Thm.\ 7.2.3]{FK2012}}]
\label{thm:f-k}
The map
$$
\U \circ \varpi: H^-/IH^- \otimes \Q_p \lra H^-/IH^- \otimes \Q_p
$$
satisfies $\xi_\chi' (\U \circ \varpi) = \xi_\chi'$, where $\xi_\chi' \in \Lambda$ is the derivative of $\xi_\chi$. 
\end{thm}

This result is very deep, and uses, among other things, a detailed analysis of Kato's Euler system. Under the assumption that $\xi_\chi'$ is a unit in $\Lam/(\xi_\chi)$ (which is equivalent to the assumption that $\xi_\chi$ has no prime factor of multiplicity greater than one), this result implies that $\U$ and $\varpi$ are isomorphisms modulo $p$-torsion. However, this assumption on $\xi_\chi$ implies that $X_\chi(1) \otimes \Q_p$ is cyclic (see Lemma \ref{multi roots and cyclicity lem}). In particular, to see that $X_\chi(1) \otimes \Q_p$ is described by modular symbols using this result, one has to first assume that $X_\chi(1) \otimes \Q_p$ is cyclic.

As a consequence of our main theorem, we have the following new result towards Sharifi's conjecture without assuming anything about $X_\chi(1) \otimes \Q_p$. 

\begin{cor}
\label{cor: Upsilon iso main} 
\ 
\begin{enumerate}[leftmargin=2em]
\item If $X_{\theta,(f_\chi)}=0$, then the localization of $\Upsilon$ at $(f_\chi)$ is an isomorphism. 
\item If $X_\theta \otimes_\Lambda (\Lambda/\xi_\chi)$ is finite, then $\U$ is injective and has finite cokernel.
\end{enumerate}
\end{cor}
\begin{proof}
(1) By Theorem \ref{thm: psmod}, we have that $\fH_\p$ is complete intersection, and hence Gorenstein. The proof of \cite[Thm.\ 1.3]{wake2} shows that if $\fH_\p$ is Gorenstein, then $\coker(\U_{(f_\chi)})=0$. Then we have
\[
\U_{(f_\chi)}: X_{\chi,(f_\chi)}(1) \onto H^-/IH^- \otimes_\Lam \Lamf.
\]
To see that it is an isomorphism, note that $\ell(X_{\chi,(f_\chi)}(1))= \ell(\Lamf/(\xi_\chi))$, and a similar argument to the proof of Lemma \ref{lem: length of I/I2} shows that $\ell(H^-/IH^- \otimes_\Lam \Lamf) \ge  \ell(\Lamf/(\xi_\chi))$. We see that $\U_{(f_\chi)}$ is a surjection of $\Lamf$-modules of the same length, so it must be an isomorphism.

(2) The assumption implies that $X_{\theta,(g)}=0$ for all prime elements $g \in \Lam$ dividing $\xi_\chi$. As in the proof of (1), this and Theorem \ref{thm: psmod} imply that $\fH_\p$ is Gorenstein for all $\p$. As $\U$ is a map of $\Lam/\xi_\chi$-modules, $\coker(\U)$ is supported at primes dividing $\xi_\chi$. But by (1), it is not supported at any such height $1$ primes, and so $\coker(\U)$ is only supported at the maximal ideal of $\Lam$. This implies that $\coker(\U)$ is finite, and $\U$ is therefore injective by \cite[Prop.\ 7.4]{wake2}.
\end{proof}

Note that our result says nothing about $\varpi$. However, combining our main theorem with Fukaya and Kato's Theorem \ref{thm:f-k}, we can prove new results about $\varpi$ as well.

\begin{cor}
\label{cor: corC stronger}
\ 
\begin{enumerate}[leftmargin=2em]
\item Assume that $X_{\theta,(f_\chi)}=0$ and that $X_{\chi,(f_\chi)}(1)$ is cyclic as a $\Lamf$-module. Then the localizations $\U_{(f_\chi)}$ and $\varpi_{(f_\chi)}$ are isomorphisms. 

\item Assume that that $X_\chi(1)[1/p]$ is cyclic as a $(\Lambda/\xi_\chi)[1/p]$-module and that $X_\theta \otimes_\Lambda (\Lambda/\xi_\chi)[1/p]=0$. Then the induced maps
$$
\Upsilon: X_\chi(1) \lra (H^-/I H^-)/\mathrm{(tor)}, \ \  \varpi: ( H^-/IH^-)/\mathrm{(tor)} \lra  X_\chi(1)
$$
are isomorphisms. Here $\mathrm{(tor)} \subset H^-/I H^-$ is the $p$-power-torsion subgroup.
\end{enumerate}
\end{cor}

\begin{proof}
(1) By Corollary \ref{cor: Upsilon iso main}, we have that $\U_{(f_\chi)}$ is an isomorphism. To show that $\varpi_{(f_\chi)}$ is an isomorphism, it suffices to show that $\U_{(f_\chi)} \circ \varpi_{(f_\chi)}$ is an isomorphism. However, we assume that $X_{\chi,(f_\chi)}(1)$ is cyclic, and so $X_{\chi,(f_\chi)}(1) \simeq \Lamf/(f_\chi^r)$ for some $r \ge 1$; since $\xi_\chi$ is a characteristic power series for $X_\chi(1)$, we know that $f_\chi^r$ is the highest power of $f_\chi$ dividing $\xi_\chi$. 
We fix a generator of $X_{\chi,(f_\chi)}(1)$ and use the isomorphism $\U_{(f_\chi)}$ to give a generator of $(H^-/IH^-)_{(f_\chi)}$.
Using these generators, we have
\[
\U_{(f_\chi)} \circ \varpi_{(f_\chi)} \in \End_{\Lamf}(\Lamf/(f_\chi^r)) \cong \Lamf/(f_\chi^r). 
\]
Let $x \in \Lamf/(f_\chi^r)$ denote the corresponding element. We want to show that $x$ is a unit, and this follows from Theorem \ref{thm:f-k}. 
To see this, we follow the argument of \cite[\S10.2]{FK2012}.

By Theorem \ref{thm:f-k}, $(1-x) \xi_\chi' = 0$. Since the image of $\xi_\chi'$ in $\Lamf/(\xi_\chi)$ is a generator of the ideal $(f_\chi^{r-1})$, we see that $1-x$ is in the maximal ideal of $\Lamf/(f_\chi^r)$.  Hence $1-x$ is nilpotent, and so $x$ is a unit and $\U_{(f_\chi)} \circ \varpi_{(f_\chi)}$ is an isomorphism. 

(2) Notice that
$$
(\Lambda/\xi_\chi)[1/p] \cong \prod \Lambda_{(f)}/\xi_\chi \Lambda_{(f)}
$$
where the product is over prime divisors $f$ of $\xi_\chi$. Then the assumption implies that the conditions of part (1) hold for all $f$. As in the proof of (1), this implies that $(H^-/IH^-)_{(f)}$ is cyclic. Since $H^-$ is a dualizing module for $\h$, this implies that $\h_\p$ is Gorenstein for the prime $\p$ of $\h$ lying over $f$. Hence $\h_\p$ is Gorenstein for all height $1$ primes $\p$ of $\h$ containing $I$. Then the result follows from Theorem \ref{thm:f-k} as in \cite[Thm.\ 7.2.8]{FK2012}.
\end{proof}

\subsection{Geometry of the ordinary Eigencurve}
\label{subsec:eigencurve}
Our main result shows that the geometry of the eigencurve is closely related to the arithmetic of cyclotomic fields. Recall that $\p \in \Spec(\fH)$ denotes the height one prime lying over our fixed choice $f_\chi$ of prime divisor of $\xi_\chi$. Our first corollary follows from our main result and the main result of \cite{wake2}.

\begin{cor}
\label{cor: gor iff}
The ring $\fH_\p$ is Gorenstein if and only if $X_{\theta,(f_\chi)}=0$.
\end{cor}
\begin{proof}
If $X_{\theta,(f_\chi)}=0$, then Theorem \ref{thm: psmod} implies that $\fH_\p$ is complete intersection and hence Gorenstein. For the other implication, we will need a modified version of \cite[Thm.\ 1.3]{wake2}. That theorem states that, if $\fH$ is weakly Gorenstein, then $X_{\theta}/\xi_\chi X_{\theta}$ is finite. Analyzing the proof, we see that the same argument works to show the same result for each prime $\p$. The result is that, if $\fH_\p$ is Gorenstein, then $X_{\theta,(f_\chi)}=0$.
\end{proof}

\begin{rem}
Keep in mind the running assumption that $\xi_\chi$ is not a unit in $\Lambda$. Without this assumption, $\fH = \Lambda$ and only the ``if'' of Corollary \ref{cor: gor iff} is true. 
\end{rem}

We get finer results if we assume more conditions on class groups. The following theorem summarizes the relations. 

\begin{thm}
\label{thm: eigencurve equivalences}
The following conditions are equivalent:
\begin{list}{}{}
\item[\rm{(1)}] $X_{\chi,(f_\chi)}(1)$ is cyclic as a $\Lambda_{(f_\chi)}$-module and $X_{\theta,(f_\chi)}=0$.
\item[\rm{(2)}] The Eisenstein ideals $I_\p$ and $\I_\p$ are principal.
\item[\rm{(3)}] The ideal $I_\p$ is generated by a single non-zero divisor.
\item[\rm{(4)}] $\fH$ has a plane singularity at $\p$, i.e.~$\mathrm{embdim}(\fH_\p)=2$.
\item[\rm{(5)}] Both $\fH_\p$ and $\h_\p$ are complete intersection rings.
\item[\rm{(6)}] Both $\fH_\p$ and $\h_\p$ are Gorenstein.
\end{list}
Moreover, if any of {\em(1)-(6)} hold, then the following are equivalent:
\begin{list}{}{}
\item[\rm{(7)}] $f_\chi^2 \nmid \xi_\chi$, i.e.\ $\xi_\chi$ has a simple zero at $(f_\chi)$. 
\item[\rm{(8)}] $\h_\p$ is regular and the intersection of components of $\Spec \fH[1/p]$ at $\p$ is transverse.
\end{list}
\end{thm}
\begin{proof}
Assume (1). Then, by Proposition \ref{prop: X=J=I}, we have $X_{\chi,(f_\chi)}(1) \simeq \I_\p/\I_\p^2 \simeq I_\p/I_\p^2$. Moreover, since we assume that $X_{\chi,(f_\chi)}(1)$ is cyclic as a $\Lambda_{(f_\chi)}$-module, we have that $\I_\p/\I_\p^2$ and $I_\p/I_\p^2$ are cyclic as $\Lambda_{(f_\chi)}$-modules, and hence as $\fH_\p$-modules. By Nakayama's lemma, this implies that $\I_\p$ and $I_\p$ are principal, implying (2).

The equivalence of (2)-(6) follows from Theorem \ref{thm:simple_equiv}.

Assume (6). Then, by Corollary \ref{cor: gor iff}, we have $X_{\theta,(f_\chi)}=0$. By Corollary \ref{cor: Upsilon iso main}, this implies
\[
X_{\chi,(f_\chi)}(1) \cong (H^-/IH^-)_{(f_\chi)} = H_\p^-/I_\p H_\p^-.
\]
Since $\h_\p$ is Gorenstein and $H_\p^-$ is a dualizing module for $\h_\p$, we have $H_\p^- \simeq \h_p$. Putting these together, we get
\[
X_{\chi,(f_\chi)}(1)  \simeq \h_\p/I_\p \cong \Lamf/(\xi_\chi)
\]
and so $X_{\chi,(f_\chi)}(1)$ is cyclic. This shows that (6) $\Rightarrow$ (1). This completes the proof of the equivalence of (1)-(6). 

In the statement of (8), the two components meeting at $\p$ are $\Spec \h_\p $ and $\Spec \fH_\p/\I_\p=\Spec \Lamf$. By definition, they meet transversely if and only if the ring $\h_\p \otimes_{\fH_\p} (\fH_\p/\I_\p)$ is a field. But, by Lemma \ref{lem:H_pullback}, we know that $\h_\p \otimes_{\fH_\p} (\fH_\p/\I_\p) = \Lamf/\xi_\chi\Lamf$, which is a field if and only if $f_\chi^2 \nmid \xi_\chi$. So we see that (8) $\Rightarrow$ (7).

Now assume (1)-(7), and we will show (8). As in the previous paragraph, the assumption that $f_\chi^2 \nmid \xi_\chi$ implies that the intersection of components of $\Spec \fH[1/p]$ at $\p$ is transverse. It remains to show that $\h_\p$ is regular, and this follows from (6) and (7) by Theorem \ref{thm:simple_equiv}.
\end{proof}

\subsection{Noncommutative modularity} 
Under the assumption of Greenberg's conjecture \ref{conj: greenberg}, we can prove a noncommutative version of the pseudo-modularity theorem \ref{thm: psmod} -- namely, that the universal ordinary cuspidal GMA is given by the linear closure of the image of Galois in the endomorphism ring of the cohomology of modular curves. 

From now on, assume that condition (1) of Theorem \ref{thm: eigencurve equivalences} is satisfied, which follows from Greenberg's conjecture. In particular, condition (1) implies that $\h_\p$ is Gorenstein, making $H_\p$ a free $\h_\p$-module and $\End_{\h_\p}(H_\p) \simeq M_2(\h_\p)$ a matrix algebra (Corollary \ref{cor:H-structure}). Let 
$$
E_{\h_\p} = \text{Image}(\Lamf[G_{\bQ,S}] \lra \End_{\h_\p}(H_\p)).
$$
Let us write
\begin{equation}
\label{eq: E_h GMA}
E_{\h_\p} = \ttmat{\h_\p}{B_{\h_\p}}{C_{\h_\p}}{\h_\p} \subset \End_{\h_\p}(H_\p) \simeq M_2(\h_\p)
\end{equation}
so that $B_{\h_\p}, C_{\h_\p} \subset \h_\p$ may be thought of as ideals. Even though $\End_{\h_\p}(H_\p) \simeq M_2(\h_\p)$, it is not clear that $B_{\h_\p}$ and $C_{\h_\p}$ are free of rank 1 over $\h_\p$. If we were to assume that $\h_\p$ is a DVR, then this would be clear, but in general we require the following argument.

\begin{lem}
\label{lem: B C are free}
Assume condition $(1)$ of Theorem \ref{thm: eigencurve equivalences} holds. Then the ideals $B_{\h_\p}, C_{\h_\p} \subset \h_\p$ are free of rank $1$ as $\h_\p$-modules.
\end{lem}
\begin{proof}
The map $\Lamf[G_{\bQ,S}] \to \End_{\h_\p}(H_\p)$ factors through $E^\ord_\m \otimes_{R_\m^\ord} \h_\p$ by Proposition \ref{prop:ord_C-H}.
By Theorem \ref{thm:H_ord}(4), we have that  $I_\p$ is the reducibility ideal of $E_{\h_\p}$. By Proposition \ref{prop:red_ideal}, we have $B_{\h_\p} \cdot C_{\h_\p}=I_\p$. On the other hand, by Theorem \ref{thm:H/I_sub}, we have that $(H^-/IH^-)_\p \subset (H/IH)_\p$ as $G_{\bQ,S}$-modules and so $C_{\h_\p} \subset I_\p$. Since $B_{\h_\p} \cdot C_{\h_\p}=I_\p$ and since $B_{\h_\p}$ and $C_{\h_\p}$ are ideals, this implies that $C_{\h_\p} = I_\p$ and $B_{\h_\p}=\h_\p$. In particular, since $I_\p \simeq \h_\p$ by Theorem \ref{thm: eigencurve equivalences}, $C_{\h_\p}$ and $B_{\h_\p}$ are both free of rank $1$ as $\h_\p$-modules. 
\end{proof}

Our assumption also implies that Theorem \ref{thm: psmod} holds and we have an isomorphism $R^\ord_\m \risom \hat\fH_p$. In the following statement of noncommutative modularity and its proof, every algebra and module needs to be completed at the maximal ideal of $\h_\p$. In particular, this statement will address the image of $\hat E^\ord_\m = E^\ord_\m \otimes_{\fH_\p} \hat\fH_\p$ in $\hat E_{\hat \h_\p} = E_{\h_\p} \otimes_{\h_\p} \hat\h_\p$.  We omit the many ``$\hat\ $'' for notational convenience. 

\begin{thm}
\label{thm: GMA R=T}
Assume condition $(1)$ of Theorem \ref{thm: eigencurve equivalences} holds. Then the natural surjective map $E^\ord_\m \otimes_{R_\m^\ord} \h_\p \onto E_{\h_\p}$ is an isomorphism.
\end{thm}
\begin{proof}
Let $B^\ord_{\h_\p} = B^\ord  \otimes_{R_\m^\ord} \h_\p$ and $C^\ord_{\h_\p} = C^\ord  \otimes_{R_\m^\ord} \h_\p $. By Theorem \ref{thm: psmod}, we have
$$
E^\ord_\m \otimes_{R_\m^\ord} \h_\p \cong \ttmat{\h_\p}{B^\ord_{\h_\p}}{C^\ord_{\h_\p}}{\h_\p}.
$$
The map $E^\ord_\m \otimes_{R_\m^\ord} \h_\p \onto E_{\h_\p}$ is the identity in the upper-left and lower-right coordinates, and so it remains to show that $B^\ord_{\h_\p} \onto B_{\h_\p}$ and $C^\ord_{\h_\p} \onto C_{\h_\p}$ are isomorphisms.

We know by Propositions \ref{prop:ord_red_CH} that there is a surjection
\[
B^\red \otimes_{\Lamf} C^\red \onto B^\ord \otimes_{R_\m^\ord} C^\ord \otimes_{R_\m^\ord} R_\m^\ord/\cJ_\m
\]
Moreover, we know by Proposition \ref{prop:ord_red} that $B^\red \otimes_{\Lamf} C^\red \simeq X_{\chi,(f_\chi)}(1)$ as a $\Lamf$-module. Since we assume $X_{\chi,(f_\chi)}(1)$ is cyclic as a $\Lamf$-module, we have that 
\[
B^\ord \otimes_{R_\m^\ord} C^\ord \otimes_{R_\m^\ord} R_\m^\ord/\cJ_\m
\]
is cyclic is a $\Lamf$-module. By Nakayama's lemma, this implies that $B^\ord$ and $C^\ord$ are cyclic as $R_\m^\ord$-modules. 

This implies that $B^\ord_{\h_\p}$ and $C^\ord_{\h_\p}$ are cyclic as $\h_\p$-modules. Since $C_{\h_\p}$ and $B_{\h_\p}$ are both free of rank $1$ as $\h_\p$-modules (Lemma \ref{lem: B C are free}), the surjections $B^\ord_{\h_\p} \onto B_{\h_\p}$ and $C^\ord_{\h_\p} \onto C_{\h_\p}$ must be isomorphisms.
\end{proof}


\bibliographystyle{alpha}
\bibliography{CWEbib-0515}

\end{document}